\documentclass[fleqn]{article}

\usepackage[american]{babel}
\usepackage{fullpage}
\usepackage[T1]{fontenc}
\usepackage{etex}
\usepackage{letltxmacro}
\usepackage{import}
\usepackage{latexsym}
\usepackage{amsfonts}
\usepackage{amsmath}
\usepackage{amssymb}
\usepackage{amsthm}
\usepackage{mathtools}
\usepackage{bbm}
\usepackage[shortlabels]{enumitem}
\usepackage[backend=biber,%
	style=alphabetic,%
	sorting=anyt,%
	maxnames=99,%
	maxalphanames=4,%
	minalphanames=3,%
	labelalpha=true,%
	isbn=false,%
	doi=false,%
	url=false%
	]{biblatex}
\renewbibmacro{in:}{} 
\DeclareFieldFormat[article,inbook,incollection,inproceedings,thesis]{title}{#1}
\usepackage{xcolor}
\usepackage{calc} 
\usepackage{etoolbox}
\robustify
\robustify\addtocounter 
\robustify
\robustify
\usepackage{xspace}
\usepackage{graphicx}

\usepackage{xparse}
\usepackage{pdftexcmds}

\usepackage{aliascnt}

\usepackage{caption}
\usepackage[hypcap=true,justification=centering]{subcaption} 

\usepackage{tikz}
\usetikzlibrary{external}
\tikzset{external/system call={lualatex
		\tikzexternalcheckshellescape -halt-on-error -interaction=batchmode
		-jobname "\image" "\texsource"},
		external/only named=true
		}
\usepackage{pgfplots}

\usepackage{hyperref}
\AtBeginDocument{}
\usepackage[all]{hypcap}

\providecommand{\bbC}{\mathbb{C}}

\providecommand{\bbI}{\mathbb{I}}

\providecommand{\bbN}{\mathbb{N}}

\providecommand{\bbR}{\mathbb{R}}
\providecommand{\bbS}{\mathbb{S}}

\providecommand{\bbZ}{\mathbb{Z}}

\providecommand{\CA}{\mathcal{A}}

\providecommand{\CC}{\mathcal{C}}
\providecommand{\CD}{\mathcal{D}}
\providecommand{\CE}{\mathcal{E}}
\providecommand{\CF}{\mathcal{F}}

\providecommand{\CH}{\mathcal{H}}

\providecommand{\CO}{\mathcal{O}}
\providecommand{\CP}{\mathcal{P}}

\providecommand{\CR}{\mathcal{R}}
\providecommand{\CS}{\mathcal{S}}

\providecommand{\CV}{\mathcal{V}}

\newcommand{\VA}{{\mathbf{A}}}

\newcommand{\VW}{{\mathbf{W}}}

\providecommand{\FC}{\mathfrak{C}}

\providecommand{\Fh}{\mathfrak{h}}

\newcommand*{\wt}[1]{\widetilde{#1}}
\newcommand*{\wh}[1]{\widehat{#1}}

\newcommand{\ttt}[1]{\textnormal{\texttt{#1}}}
\newcommand{\eps}{\varepsilon}

\newcommand{\Hs}[1][\vec s]{H^{#1}} 

\renewcommand{\d}{\mathop{}\!\dd} 
\newcommand{\dd}{\mathrm{d}}
\newcommand{\ee}{\mathrm{e}}
\newcommand{\ii}{\mathrm{i}}
\newcommand{\rmr}{\mathrm{r}}
\newcommand{\rms}{\mathrm{s}}

\newcommand{\ind}{\mathbbm{1}}

\newcommand{\compl}{\mathsf{c}}

\newcommand{\supp}{\mathrm{supp}\mathop{}}
\newcommand{\diag}{\mathrm{diag}}
\newcommand{\dist}{\mathrm{dist}}
\newcommand{\trace}{\mathrm{tr}}

\renewcommand{\Re}{\operatorname{Re}}

\newcommand{\kp}{\vec k{}'}

\newcommand{\tpp}{\vec t{}''}

\newcommand{\xip}{\vec \xi{}'}

\newcommand{\xp}{\vec x{}'}

\newcommand{\yp}{\vec y{}'}

\newcommand{\zp}{\vec z{}'}
\newcommand{\zpp}{\vec z{}''}

\newcommand{\spr}{\vec s{}'}
\newcommand{\sst}{\vec s{}^*}

\newcommand{\Ac}{A_{\mathrm{c}}}
\newcommand{\fc}{f_{\mathrm{c}}}

\newcommand{\bbSd}{{\mathbb{S}^{d-1}}}
\newcommand{\jl}{{j\!\?\?,\ell}}

\newcommand{\jlk}{{j\!\?\?,\?\ell\!\?\?,\?\vec k}}

\renewcommand{\Rn}{R_{\vec n}}
\newcommand{\Rs}{R_{\vec s}}

\newcommand{\Rjl}{R_\jl}

\newcommand{\sjl}{\vec s_\jl}

\newcommand{\rs}{\rho_{\vec s}}

\newcommand{\nth}{\textsuperscript{th}\xspace} 

\newcommand{\normalr}[1]{\mathclose{\displaystyle#1}} 

\NewDocumentCommand{\Dn} {m O{} m}  {\frac{\dd^{#1}#2}{\dd#3^{#1}}}
\NewDocumentCommand{\Dpn}{m O{} m}  {\frac{\partial^{#1}#2}{\partial#3^{#1}}}
\NewDocumentCommand{\Dpi}{m O{} m}  {\frac{\partial^{\abs[]{#1}}#2}{\partial{#3}^{#1}}}
\NewDocumentCommand{\Dpp}{m O{} m m}{\frac{\partial^{#1}#2}{\partial{#3}\,\partial{#4}}}
\NewDocumentCommand{\D}  {O{} m}    {\Dn{}[#1]{#2}}
\NewDocumentCommand{\DD} {O{} m}    {\Dn{2}[#1]{#2}}
\NewDocumentCommand{\Dp} {O{} m}    {\Dpn{}[#1]{#2}}
\NewDocumentCommand{\DDp}{O{} m}    {\Dpn{2}[#1]{#2}}

\DeclarePairedDelimiter {\nrmInternal}   {\lVert}    {\rVert}
\DeclarePairedDelimiter {\absInternal}   {\lvert}    {\rvert}
\DeclarePairedDelimiter {\parInternal}   {\lparen}   {\rparen}
\DeclarePairedDelimiter {\braInternal}   {\lbrack}   {\rbrack}
\DeclarePairedDelimiter {\bbrInternal}   {\llbracket}{\rrbracket}
\DeclarePairedDelimiter {\ceiInternal}   {\lceil}    {\rceil}
\DeclarePairedDelimiter {\cceInternal}   {\llceil}   {\rrceil}
\DeclarePairedDelimiter {\flrInternal}   {\lfloor}   {\rfloor}
\DeclarePairedDelimiter {\fflInternal}   {\llfloor}  {\rrfloor}
\DeclarePairedDelimiter {\inpInternal}   {\langle}   {\rangle}
\DeclarePairedDelimiter {\crlInternal}   {\{}        {\}}
\DeclarePairedDelimiterX{\setInternal}[2]{\{}        {\}}        {#1\colon#2} 

\NewDocumentCommand{\norm}    {s o m}  {  \resizerOneInput {\nrmInternal}{#1}{#2}{#3}    }
\NewDocumentCommand{\abs}     {s o m}  {  \resizerOneInput {\absInternal}{#1}{#2}{#3}    }
\NewDocumentCommand{\snorm}   {s o m}  {  \resizerOneInput {\absInternal}{#1}{#2}{#3}    }
\NewDocumentCommand{\parens}  {s o m}  {  \resizerOneInput {\parInternal}{#1}{#2}{#3}    }
\NewDocumentCommand{\bracket} {s o m}  {  \resizerOneInput {\braInternal}{#1}{#2}{#3}    }
\NewDocumentCommand{\bbracket}{s o m}  {  \resizerOneInput {\bbrInternal}{#1}{#2}{#3}    }
\NewDocumentCommand{\ceil}    {s o m}  {  \resizerOneInput {\ceiInternal}{#1}{#2}{#3}    }
\NewDocumentCommand{\cceil}   {s o m}  {  \resizerOneInput {\cceInternal}{#1}{#2}{#3}    }
\NewDocumentCommand{\floor}   {s o m}  {  \resizerOneInput {\flrInternal}{#1}{#2}{#3}    }
\NewDocumentCommand{\ffloor}  {s o m}  {  \resizerOneInput {\fflInternal}{#1}{#2}{#3}    }
\NewDocumentCommand{\inpr}    {s o m}  {  \resizerOneInput {\inpInternal}{#1}{#2}{#3}    }
\NewDocumentCommand{\reg}     {s o m}  {  \resizerOneInput {\inpInternal}{#1}{#2}{#3}    }
\NewDocumentCommand{\curly}   {s o m}  {  \resizerOneInput {\crlInternal}{#1}{#2}{#3}    }
\NewDocumentCommand{\set}     {s o m m}{  \resizerTwoInputs{\setInternal}{#1}{#2}{#3}{#4}}
\NewDocumentCommand{\card}    {s o m m}{\#\resizerTwoInputs{\setInternal}{#1}{#2}{#3}{#4}}

\NewDocumentCommand{\coeff}{s o m}     {\resizerOneInput{\braInternal}{#1}{#2}{#3}\mathop{}\!}
\NewDocumentCommand{\seq}  {s o m O{0}}{{\resizerOneInput{\crlInternal}{#1}{#2}{#3}}_{#4}^\infty}
\NewDocumentCommand{\ops}  {s o m O{0} m}{\mathrm{ops}\resizerOneInput{\parInternal}{#1}{#2}{%
\resizerOneInput{\crlInternal}{#1}{#2}{#3}_{#4}^\infty,#5}}

\NewDocumentCommand{\resizerOneInput}{m m m m}{
	\IfBooleanTF{#2}   
		{#1[\big]{#4}} 
		{\IfNoValueTF{#3} 
			{#1*{#4}}     
			{#1[#3]{#4}}} 
}
\NewDocumentCommand{\resizerTwoInputs}{m m m m m}{
	\IfBooleanTF{#2}
		{#1[\big]{#4}{#5}}
		{\IfNoValueTF{#3}
			{#1*{#4}{#5}}
			{#1[#3]{#4}{#5}}}
}

\newcommand{\phantomrel}{\mathrel{\phantom{=}}}

\makeatletter
\newcommand{\specialcell}[1]{\ifmeasuring@#1\else\omit$\displaystyle#1$\ignorespaces\fi}
\makeatother

\newlength{\mytemplength} 

\newcommand{\lcopywidth}[2]{%
	\phantom{\smash{#2}}\mathllap{#1}}

\newcounter{universalcounter}[section]
\renewcommand{\theuniversalcounter}%
	{\arabic{section}.\arabic{universalcounter}}
\renewcommand{\theequation}%
	{\arabic{section}.\arabic{equation}}

\newcommand*{\mynewtheorem}[2]{
  \newaliascnt{#1}{universalcounter}
  \newtheorem{#1}[#1]{#2}
  \aliascntresetthe{#1}
  \expandafter\def\csname#1autorefname\endcsname{#2}
}

\newcommand*{\mynewdefinition}[2]{
	\mynewtheorem{#1x}{#2}
	\newenvironment{#1}
		{\pushQED{\qed}\renewcommand{\qedsymbol}{$\triangle$}\csname#1x\endcsname}
		{\popQED\csname end#1x\endcsname}
}

\theoremstyle{definition}
\mynewdefinition{definition} {Definition}
\mynewdefinition{example}    {Example}
\mynewdefinition{remark}     {Remark}

\theoremstyle{plain}
\mynewtheorem{theorem}    {Theorem}
\mynewtheorem{lemma}      {Lemma}
\mynewtheorem{corollary}  {Corollary}
\mynewtheorem{proposition}{Proposition}
\mynewtheorem{fact}       {Fact}
\mynewtheorem{assumption} {Assumption}
\mynewtheorem{goal}       {Goal}

\newcounter{step}[universalcounter]
\NewDocumentCommand{\step}{m o}{
	\vspace{\topsep}%
	{\refstepcounter{step}\IfNoValueF{#2}{#2}%
	{\noindent\bfseries Step~\thestep\ -- #1}\par\nobreak%
	}\noindent\ignorespaces%
}

\numberwithin{equation}{section}
\numberwithin{figure}{section}
\numberwithin{table}{section}

\addto\extrasenglish{

}
\addto\extrasamerican{

}

\mathtoolsset{showonlyrefs=true}

\definecolor{RoyalBlue}{cmyk}{1, 0.50, 0, 0}
\hypersetup{
	colorlinks=true,
	pdfstartview=FitV, breaklinks=true,bookmarksnumbered,bookmarksdepth=2,
	bookmarksopen=true,bookmarksopenlevel=0,%
	pdfhighlight=/O,%
	urlcolor=black,
	linkcolor=RoyalBlue,citecolor=RoyalBlue,%
	pdftitle={On the Approximation of Functions with Line Singularities by Ridgelets},%
	pdfauthor={Axel Obermeier, Philipp Grohs},%
	pdfcreator={pdfLaTeX},%
	pdfproducer={LaTeX}%
}

\newsavebox{\mytempbox}
\newlength{\plotsize}

\makeatletter
\g@addto@macro\@floatboxreset\centering
\makeatother

\pgfplotsset{compat=newest,
	every tick label={font=\footnotesize},
	matlab/.style={
		scale only axis,
		axis on top,
		separate axis lines,
		every outer axis line/.append style={black},
		every tick label/.append style={font=\color{black}\footnotesize},
		xticklabel={$\mathclap{\pgfmathprintnumber{\tick}}$},
		axis background/.style={fill=white},
		colormap/jet,
		colorbar,
		colorbar style={separate axis lines,every outer axis line/.append style={black},every tick label/.append style={font=\color{black}\footnotesize}}
	},
	matlabnoclb/.style={
		scale only axis,
		axis on top,
		separate axis lines,
		every outer axis line/.append style={black},
		every tick label/.append style={font=\color{black}\footnotesize},
		xticklabel={$\mathclap{\pgfmathprintnumber{\tick}}$},
		axis background/.style={fill=white}
	}
}

\let\oldleft\left
\let\oldright\right
\renewcommand{\left}{\mathopen{}\mathclose\bgroup\oldleft}
\renewcommand{\right}{\aftergroup\egroup\oldright}

\LetLtxMacro\oldvec\vec
\def\vec#1{\oldvec{#1{}}}

\newcommand{\?}{\:\!} 

\title{On the Approximation of Functions with Line Singularities by Ridgelets}
\author{Axel Obermeier\footnote{Seminar for Applied Mathematics, ETH Z\"urich, R\"amistr. 101, 8092 Z\"urich, Email: \texttt{axel.obermeier@sam.math.ethz.ch}}, %
Philipp Grohs\footnote{Seminar for Applied Mathematics, ETH Z\"urich, R\"amistr. 101, 8092 Z\"urich, Email: \texttt{philipp.grohs@sam.math.ethz.ch}}{\hspace{0.3em}}\textsuperscript{,}\footnote{Fakult\"at f\"ur Mathematik, Universit\"at Wien, Oskar-Morgenstern-Platz 1, 1090 Wien, Email: \texttt{philipp.grohs@univie.ac.at}}}

\begin{document}

\maketitle

\begin{abstract}
	In \cite{compress}, the authors discussed the existence of numerically feasible solvers for advection equations that run in optimal computational complexity. In this paper, we complete the last remaining requirement to achieve this goal --- by showing that ridgelets, on which the solver is based, approximate functions with line singularities (which may appear as solutions to the advection equation) with the best possible approximation rate.
	
	Structurally, the proof resembles \cite{mutilated}, where a similar result was proved for a different ridgelet construction, which is however not well-suited for use in a PDE solver (and in particular, not suitable for the CDD-schemes \cite{cdd} we are interested in). Due to the differences between the two ridgelet constructions, we have to deal with quite a different set of issues, but are also able to relax the (support) conditions on the function being approximated. Finally, the proof employs a new convolution-type estimate that could be of independent interest due to its sharpness.
\end{abstract}

\setcounter{tocdepth}{1}
\tableofcontents

\subsection*{Acknowledgements}

The first author gratefully acknowledges support for this work by the Swiss National Science Foundation, Project 146356.

\newpage
\section{Introduction}\label{sec:introduction}

Over the past two decades, the field of Applied Harmonic Analysis has produced a wide range of multiscale systems which are exceptionally well-adapted to different signal classes. Wavelets \cite{wavelets} are the classical construction in this context of course, which are able to achieve ($N$-term) approximation rates of functions with point singularities as if the singularities were not there (in stark contrast to Fourier series, for example).

\subsection{Higher-Dimensional Singularities}

The number of possible construction increases dramatically when going from point-like singularities to functions that are singular on a higher-dimensional manifold, and has spawned an entire ``zoo'' of constructions, among which are ridgelets \cite{Can98}, curvelets \cite{Candes2005a,Candes2005b,CDDY06}, shearlets \cite{KuLaLiWe,shearlets_book} and contourlets \cite{DV05} --- the mentioned constructions are well-adapted to line singularities (ridgelets, see \cite{mutilated}), resp. curved singularities (the rest, see e.g. \cite{curvelet_cartoon,par-mol}). Since many of the proofs (for example of approximation rates) often resemble each other between constructions, some effort has been made recently to unify them by discovering and clarifying the underlying concepts --- curvelets, shearlets and contourlets fall into the framework of so-called ``parabolic molecules'' \cite{par-mol}, while all of the mentioned systems (including wavelets and ridgelets) are encompassed by the even broader framework of $\alpha$-molecules \cite{alpha-mol}.

The well-adaptedness of such a dictionary (say, $\Phi$) to its target classes is reflected in the fact that, for a function $f$ from the respective class, the coefficient vector $\inpr{\Phi,f}:=(\inpr{\varphi_\lambda,f})_{\lambda\in\Lambda}$ contains few large coefficients and the rest is small or negligible --- allowing tremendously improved performance for data compression, data denoising, data reconstruction etc. compared to other representation systems.

\subsection{Motivation: CDD-Schemes}\label{ssec:cdd}

But aside from image applications (where the dominating features --- mostly edges --- can often be modelled by such singularities), many of these function classes appear in the context of various PDEs as well. The seminal work \cite{cdd} showed that, in the context of elliptic PDEs exhibiting point singularities, wavelets not only sparsify the solution, but also the Galerkin matrix of the discretised operator. This allows, roughly speaking, to construct a numerically feasible algorithm that recovers $N$ of the largest coefficients of the solution in $\CO(N)$ floating point operations (flops) --- which is the best that is even theoretically possible.

This line of thought led to many further papers (e.g. \cite{cdd2,Stevenson2004,quadr_unit_cost_gant_ste,Dahlke2007,dahlke_steepest_descent}) and these methods are now referred to as CDD-schemes. Considering the very strong results obtained for wavelets with this approach, the question arose whether such a scheme could also be developed for a class of PDEs exhibiting higher-dimensional singularities --- with the well-adapted representation systems mentioned above being the obvious candidates for discretising the PDE.

To formulate some sufficient conditions for transferring the results to other PDEs, consider a differential operator $A:\CH\to\CH'$, mapping from a Hilbert space $\CH$ to its dual and inducing a variational formulation
\begin{align}\label{eq:intro_gen_var_prb}
	a(v,u)=\ell(v), \quad \forall v\in \CH.
\end{align}
Based on \cite{cdd,Stevenson2004,Dahlke2007},  the following list of ingredients makes it possible to transfer the results achieved by \cite{cdd} for wavelets to a discretisation for $A$ (see e.g. \cite[Sec. 1.2]{compress}, \cite[Sec. 4.3.1]{my_thesis}):

\begin{enumerate}[(I)]
\item\label{itm:ingr:1}
	The bilinear form $a$ is bounded and coercive with respect to the norm of the Hilbert space $\CH$, i.e.
	\begin{align}
		a(v,v) \sim \norm{v}_{\CH}^2
		\quad \text{and} \quad
		a(v,u)\lesssim \norm{v}_{\CH} \norm{u}_{\CH}.
	\end{align}
\item\label{itm:ingr:2}
	The system is discretised with a frame\footnote{Actually, more specifically, a Gel'fand frame for the Gel'fand triple $(\CH,L^2,\CH')$, see \cite{Dahlke2007}.} $\Phi$ for $\CH$, in other words, there is a diagonal weight $\VW=\diag\parens*{(w_\lambda)_{\lambda\in\Lambda}}$ such that
	\begin{align}
		\norm{u}_{\CH} \sim \norm*{\inpr{\Phi,u}_\CH}_{\ell_{\VW}^2}:= \norm*{\parens*{w_\lambda\inpr{\varphi_\lambda,u}_\CH}_{\lambda\in\Lambda}}_{\ell^2}.
	\end{align}
\item\label{itm:ingr:3}
	The Galerkin matrix $\VA$ --- i.e. the discretisation of \eqref{eq:intro_gen_var_prb} by $\Phi$ (and preconditioned by $\VW$) --- needs to be ``almost diagonal''\footnote{This name tries to encapsulate the underlying principle without going into details; the needed concept is called compressibility, see \cite[Def. 3.6]{cdd}.}.
\item\label{itm:ingr:4}
	The frame $\Phi$ is optimal for a class $\FC$ of model solutions to \eqref{eq:intro_gen_var_prb}, in the sense that it allows $N$-term approximation rates of functions in this class with the benchmark rate $\sigma^*(\FC)$; see \autoref{sec:benchmark}.
\end{enumerate}

These ingredients suffice to formulate the desired algorithm --- under the assumption of unit quadrature cost\footnote{Removing this assumption is the subject of future work based on ideas in \cite{quadr_unit_cost_gant_ste}.} for evaluating the inner products $\inpr{\Phi,u}_\CH$ and within $\VA$ --- see \autoref{th:main_result} below.

\subsection{Linear Advection Equations}

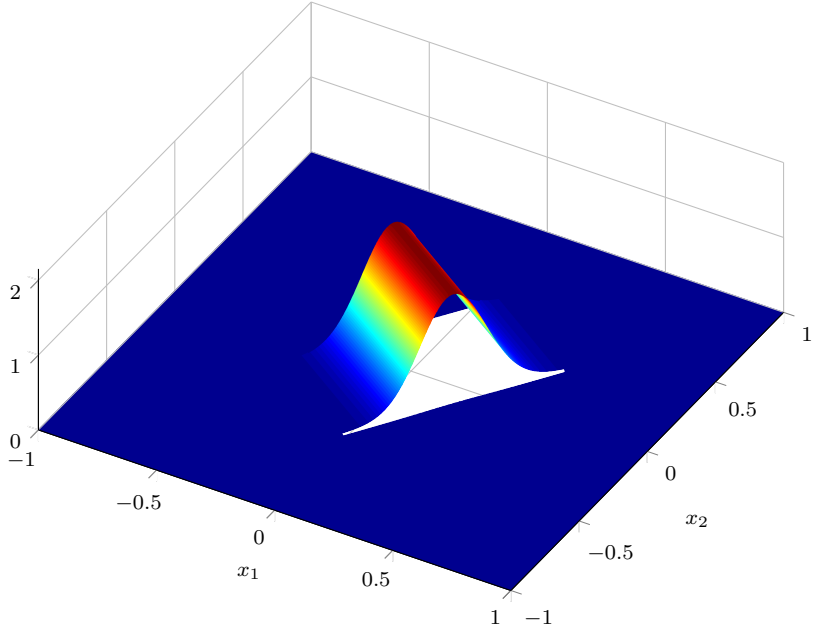
\begin{figure}
	\setlength{\plotsize}{0.6\linewidth}
	\tikzsetnextfilename{box_sol}%
%
%
\definecolor{myblue}{rgb}{0,0,0.56078}%
\begin{tikzpicture}[trim axis left,trim axis right]
\begin{axis}[%
	width=\plotsize,
	xlabel={\footnotesize$x_1$},
	ylabel={\footnotesize$x_2$},
	xlabel shift={-0.75em},
	ylabel shift={-0.75em},
	scale only axis,
	tick align=outside,
	clip=false,
	xmin=-1,
	xmax=1,
	xmajorgrids,
	ymin=-1,
	ymax=1,
	ymajorgrids,
	zmin=0,
	zmax=2,
	zmajorgrids,
	every tick label/.append style={font=\footnotesize},
	axis x line*=bottom,
	axis y line*=left,
	axis z line*=left
	]
	
	\addplot3 graphics[points={
	(-0.99665,-0.98692,9.7769e-11) => (55.7081,105.3937)
	(0.97341,0.98941,1.7626e-10) => (377.9119,156.585)
	(0.023055,0.054493,1.0839) => (223.5853,168.63)
	(0.96699,-0.96354,2.7661e-09) => (260.4731,37.6406)
	}]{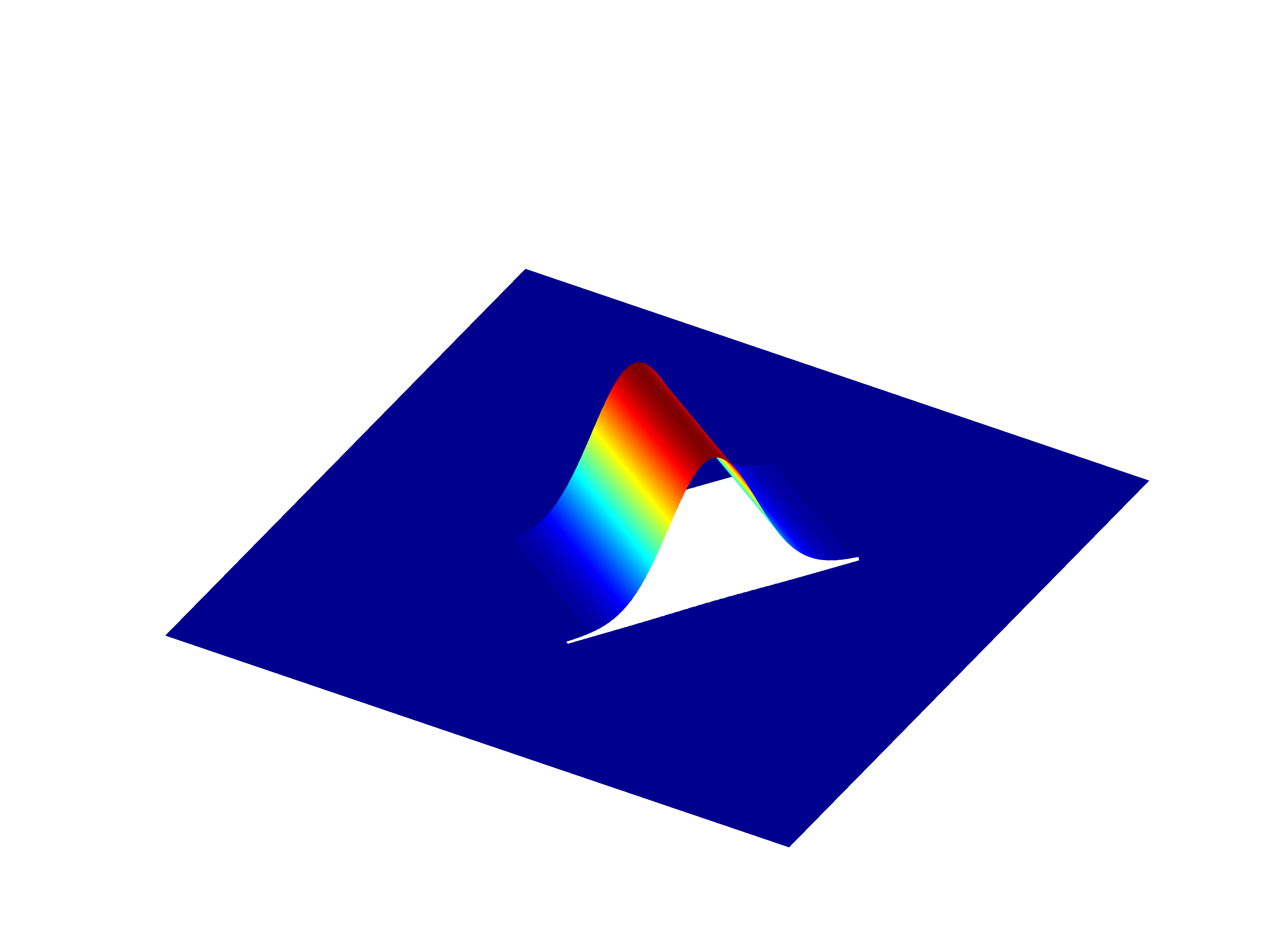};
	
	\def\myxtick{0.046}
	\def\myytick{0.058}
	\def\myztick{0.15}
	\def\mylw{0.005}
	
	\pgfplotsinvokeforeach{0,1,2}
	{
	\draw[white] (-1,-1,#1) -- (-1-\myxtick,-1,#1);
	\draw[white] (-1,-1,#1) -- (-1,-1-\myytick,#1);
	\draw[very thin,gray] (-1,-1,#1) -- (-1,-1-\myytick,#1);
	}
	\pgfplotsinvokeforeach{-1,-0.5,0,0.5,1}
	{
	\draw[white] (#1,-1,0) -- (#1,-1,-\myztick);		
	\draw[white] (#1,-1,0) -- (#1,-1-\myytick,0);		
	\draw[very thin,gray] (#1,-1,0) -- (#1,-1-\myytick,0);
	}
	\pgfplotsinvokeforeach{-1,-0.5,0,0.5,1}
	{
	\draw[white] ( 1,#1,0) -- ( 1,#1,-\myztick);		
	\draw[very thin,gray] (1,#1,0) -- (1+\myxtick,#1,0);
	}
	\draw[very thin,gray] (1,-0.99,0) -- (1,-1-\myytick,0);
	\draw[myblue,thick] (-1+\mylw,-1+\mylw,0) -- (1-\mylw,-1+\mylw,0) -- (1-\mylw,1-\mylw,0) -- (-1+\mylw,1-\mylw,0) -- cycle;
	\draw[white,thick] (-1,1,0.02) -- (1,1,0.02);
	\draw[gray!50] (-1,-1,0) -- (-1,1,0) -- (1,1,0);
	\pgfplotsinvokeforeach{-1,-0.5,0,0.5,1}
	{
	\draw[gray!50] (#1,1,0) -- (#1,1,1);		
	}
	\draw (-1,-1,2+\myztick) -- (-1,-1,0) -- (1,-1,0) -- (1,1,0);
\end{axis}
\end{tikzpicture}%
	\caption[Solution for a singular right-hand side]{Solution to the advection equation \eqref{eq:LinTrans} for a singular right-hand side}\label{fig:sing_sol}
\end{figure}

Regarding differential equations exhibiting solutions with more complicated singularities than points, such features already appear in the context of linear advection equations. More precisely, we will deal with the unidirectional \emph{advection--reaction equation},
\begin{align}\label{eq:LinTrans}
	Au:=\vec s \cdot \nabla u(\vec x)  + \kappa(\vec x) u(\vec x)= f(\vec x),
\end{align}
which describes the stationary distribution of the unknown quantity $u$ under absorption $\kappa$, emission $f$ and transport in direction $\vec s$. The corresponding Hilbert space to \eqref{eq:LinTrans} is the \emph{anisotropic Sobolev space}
\begin{align}
	\Hs(\bbR^d)&:=\set*{f\in L^2(\bbR^d)}{(\vec s\cdot\nabla)f\in L^2(\bbR^d)},
\end{align}
which is equipped with the norm
\begin{align}
	\norm{f}_{\Hs(\bbR^d)}&:=\sqrt{\norm{f}_{L^2(\bbR^d)}^2+\norm{(\vec s\cdot\nabla)f}_{L^2(\bbR^d)}^2}.
\end{align}

Roughly speaking, the operator $A$ smoothes only in the direction of $\vec s$, and singularities orthogonal to this direction may remain, compare for example \autoref{fig:sing_sol}. Typical methods for the numerical solution of \eqref{eq:LinTrans} include (\cite[Chap. 5]{ern}):
\begin{itemize}
	\item Galerkin Least Squares: $\inpr{Av,Au}=\inpr{Av,f}\quad \forall v \in \CV_{\mathrm{test}}$
	\item Discontinuous Galerkin Methods
	\item Streamline Upwinding Petrov Galerkin (SUPG) \cite{supg}
\end{itemize}
All of these suffer from the fact that the transport term $s \cdot \nabla u$ means that \eqref{eq:LinTrans} is not $H^1$-elliptic, which leads to ill-conditioned systems of equations (with no rigorous results about the choice or efficiency of preconditioners).

Even wavelets do not perform optimally for \eqref{eq:LinTrans} --- heuristically, the break-down can also be explained with the fact that it takes many wavelets (which are adapted to point singularities) to resolve a singularity along a line --- picture a ``cliff'' or compare again with \autoref{fig:sing_sol}.
To model this behaviour of the solutions, we will also introduce a class based on ``mutilated'' Sobolev functions (cf. \cite{mutilated}) in \autoref{ssec:prop_advec}.

Ridgelets turned out to be particularly well-suited for discretising \eqref{eq:LinTrans} --- \cite{grohs1} essentially proved \ref{itm:ingr:2} of the ingredients mentioned above, while \cite{compress} showed \ref{itm:ingr:1} $\&$ \ref{itm:ingr:3}. In this sense, the present work complements and completes these previous works by aiming to show the last remaining ingredient \ref{itm:ingr:4}, which will then allow to achieve the desired highly efficient algorithms, see \autoref{th:main_result}.

\subsection{Main Result}

The core of this paper is to show that ridgelets (in the version of \cite{grohs1}) achieve the best possible $N$-term approximation rate for functions in $H^t(\bbR^d)$ that are allowed to be singular across hyperplanes, stated in rough strokes in the following ``theorem''.

\begin{theorem}[Sketch of \autoref{th:approx}]\label{th:approx_intro}
	Let $f$ be a function in $f \in H^t(\bbR^d)$ apart from hyperplanes\footnote{We will make this statement precise in \autoref{def:Ht_except_hyp}.}, such that the decay condition
	\begin{align}\label{eq:intro_f_glob_decay}
		\abs{f(\vec x)} \le \frac{C_m}{\reg{\vec x}^{2m}}
	\end{align}
	holds for all $m\in \bbN$, where $\reg{\vec x}=\sqrt{1+|\vec x|^2}$. Then, for arbitrary $\delta>0$ and $N\in\bbN$, the function $f_N$ --- reconstructed from only the $N$ largest ridgelet coefficients of $f$ --- satisfies
	\begin{align}\label{eq:intro_approx_f}
		\norm{f-f_N}_{L^2} \le C_\delta N^{-\frac{t}{d}+\delta},
	\end{align}
	which (up to $\delta$) is the best theoretically possible rate.
	
	Additionally, if $u$ is the solution to \eqref{eq:LinTrans} with the $f$ from above, then, for $\kappa$ smooth enough, the reconstruction $u_N$ from the $N$ largest ridgelet coefficients of $u$ similarly satisfies (note the different norm)
	\begin{align}\label{eq:intro_approx_u}
		\norm{u-u_N}_{\Hs} \le C'_\delta N^{-\frac{t}{d}+\delta}.
	\end{align}
\end{theorem}

As hinted at in the beginning, the original ridgelet definition by \cite{Can98} already showed \eqref{eq:intro_approx_f} in \cite{mutilated}. However, by its construction, it is not possible to incorporate it into the kind of CDD-schemes we want to achieve. In particular, the frame elements of \cite{mutilated} have unbounded support (which is made mathematically feasible by restricting the domain to $\Omega=[0,1]^d$), and therefore, the necessary sparsity of the matrix that \ref{itm:ingr:3} alludes to is impossible to achieve.

In contrast, \cite{grohs1} constructs a frame for the full $L^2(\bbR^d)$ (resp. $\Hs(\bbR^d)$), but this necessitates frame elements $\varphi_\lambda$ which are intrinsically in $L^2(\bbR^d)$ themselves, and thus, have much more localised support. The price for this is that we need the full grid $\bbZ^d$ (under a certain transformation) of translations to cover all of $\bbR^d$ with our frame elements, while \cite{Can98} is able to make do with a one-dimensional grid. Since the proof involves counting large coefficients (in some sense), these additional $d-1$ dimensions make the proof of our result substantially more involved and require some fairly delicate estimates to work out. In this respect, we believe that the auxiliary estimates (\autoref{th:Imn_est} and \autoref{cor:Imn_higher_dim}) we have proved for this purpose are of independent interest due to their increased strength compared to previous results (see \autoref{rem:grafakos}).

Finally, for the CDD-machinery to work, we need approximation estimates in the norm of the Hilbert space $\CH=\Hs$ in which the solution lives, compare \eqref{eq:intro_approx_u}. This norm corresponds to multiplying the coefficient sequence (element-wise) with a growing weight (namely the $\VW$ from \ref{itm:ingr:2}), which further complicates matters.

However, once we achieve the proof of \autoref{th:approx_intro}, we will have showed the following, compare \cite[Cor. 6.1, Thm. 6.2]{compress}\footnote{Note that the requirements of \cite[Thm. 6.2]{compress} are stated erroneously, in the sense that the decay of the $f_i$ needs to be global (compare \eqref{eq:intro_f_glob_decay}) and not just across the interfaces of the hyperplanes.}. For a more detailed account of combining ingredients \ref{itm:ingr:1}--\ref{itm:ingr:4} into the result below, we refer to \cite[Thm. 7.1.1]{my_thesis}.

\begin{theorem}\label{th:main_result}
	For arbitrary $\vec s\in\bbSd$, consider \eqref{eq:LinTrans} with right-hand side $f\in H^t$ that is allowed to be singular across hyperplanes and satisfies the decay condition
	\begin{align}
		\abs{f(\vec x)} \le \frac{C_m}{\reg{\vec x}^m}
	\end{align}
	for all $m\in\bbN$ (which is possible for compact support or exponential decay, for example).
	
	Assuming that the absorption coefficient satisfies $\kappa\ge\gamma>0$ and $\kappa\in H^{4(t+d+1)}$, an approximand $u_\eps$ to the solution of $Au=f$ satisfying
	\begin{align}
		\norm{u-u_\eps}_{\Hs} \le \eps
	\end{align}
	can be found with the help of a numerically feasible ridgelet-based algorithm, such that, for arbitrary $\sigma<\frac td$ (and ignoring quadrature cost),
	\begin{align}
		\#\curly*{\text{arithmetic operations necessary to compute $u_\eps$}}\lesssim\eps^{-\frac 1\sigma}.
	\end{align}
\end{theorem}

\begin{remark}\label{rem:sing_kappa}
	As a matter of fact, in \cite[Sec. 7.2]{my_thesis}, we show that \autoref{th:main_result} can be extended to $\kappa\in H^{4(t+d+1)}$ that is also allowed to be singular across hyperplanes\footnote{As long as any potential singularities of $f$ lie in hyperplanes that are \emph{parallel} to the singularity in $\kappa$.}, which is somewhat surprising, since \ref{itm:ingr:3} depends crucially on the smoothness of $\kappa$. In a nutshell, one shifts the singularity from $\kappa$ to the right-hand side $f$, where it is harmless (see above), which is possible mainly due to having an explicit formula for the solution of \eqref{eq:LinTrans}.
\end{remark}

\subsection{Impact}\label{ssec:impact}

Although \eqref{eq:LinTrans} is quite simple, having a highly efficient solver for such an equation opens the door to efficiently solve more complicated equations like the radiative transport equation (RTE),
\begin{equation}\label{eq:RTE}
	B u:=\vec s \cdot \nabla u(\vec x, \vec x) + \beta(\vec x) u(\vec x,\vec s) = f(\vec x) + \sigma(\vec x) \int_{\bbS^{d-1}} K(\vec s, \spr) u(\vec x,\spr) \d \spr.
\end{equation}
which couples the different directions $\vec s, \spr$; see e.g. \cite[Sec. 9.5]{modest2013radiative} for an introduction, resp. \cite{appl_RTE,nummeth_RTE,RTE_schwab,RTE_schwab2}, for examples of applications of \eqref{eq:RTE} and existing methods to solve it.

There are several ways to utilise a solver for \eqref{eq:LinTrans} to solve \eqref{eq:RTE}. If we neglect the scattering term for the moment, the ridgelet-based solver we develop could be used to solve \eqref{eq:RTE} by either tensor product or collocation methods in $\vec s$ (similar to techniques used in \cite{grella}, where one can additionally make use of the multiscale structure of $\Phi$ to alleviate the curse of dimensionality by balancing resolution in angle with resolution in space).

One possibility to reintroduce the scattering term is via an iterative scheme --- for example by evaluating the integral for the previous iterand and adding the result to the right-hand side. We refer to \cite{FFT-paper}, where an \ttt{FFT}-based ridgelet discretisation based on this ``source iteration'' has been implemented.

One crucial aspect in these procedures is that solutions for different $\vec s$ can be added together easily, which is satisfied by our construction (since the ridgelets achieve optimal approximation for all directions $\vec s$ simultaneously!). While we have already mentioned that uniformly refined FE methods (even if adaptive) do not work well in this context, one might have the idea to adapt the FEM mesh anisotropically. The problem with this is that the meshes for different directions $\vec s \neq \spr$ would have to be combined somehow, making cumbersome interpolation between such meshes necessary.

Considering the work already carried out in \cite{grohs1,compress}, this paper completes the picture regarding the necessary ingredients for developing a CDD-scheme for \eqref{eq:LinTrans}. The result of this --- \autoref{th:main_result} --- is very strong:  complexity here is measured in terms of arithmetic operations to be carried out by a processor and the solution is even allowed to possess singularities along lines (resp. hyperplanes in  higher dimensions) --- for the right-hand side $f$, as well as the absorption coefficient $\kappa$ (see \autoref{rem:sing_kappa}).

To illustrate the strength of these results, consider a function $f\in H^t$ apart from a finite number of line singularities (in arbitrary directions) in two dimensions. Then, the approximation error of using just $N$ ridgelets is $\CO(N^{-\frac t2})$ and the number of flops to find these coefficients is of order $\CO(N)$. For functions with medium to high Sobolev regularity (apart from the line singularities), this approximation rate represents an improvement of many orders of magnitude over wavelet or FE Methods, with respective $N$-term approximation rates of $\CO(N^{-\frac 12})$ and $\CO(N^{-\frac 14})$, irrespective of the magnitude of $t$. In terms of complexity, the advantage is greater still because the linear systems for other methods cannot usually be solved in linear time.

On the other hand, the convergence results are confined to linear advection equations \eqref{eq:LinTrans} and our analysis assumes that $\vec x$ belongs to the full space $\bbR^d$. The latter fact poses no problem if for instance the source term $f$ is compactly supported but in many applications one needs to restrict $\vec x$ to a finite domain $D\subseteq \bbR^d$ and impose inflow boundary conditions. The efficient incorporation of boundary conditions will require the construction of ridgelet frames on finite domains, which is the subject of future work\footnote{To be more precise, incorporation of inflow boundary conditions is possible with the code developed in \cite{FFT-paper} but a rigorous analysis is still lacking.}. With such a construction at hand the theoretical analysis carried out in this paper would essentially go through also for finite domains. In this regard we mention that, very recently, shearlet frames were successfully constructed on domains (\cite{shearlet-domains}), raising the hope that this approach can be transferred to the closely-related ridgelets.

\subsection{Outline}

The outline of the paper is as follows. Below, we wrap up the section by briefly introducing the most important notational conventions we will use throughout this paper. In \autoref{sec:prelim}, we cover some crucial estimates, the most important properties of the advection equation \eqref{eq:LinTrans}, as well as the model class of ``mutilated'' Sobolev functions. Furthermore, we recall the ridgelet construction of \cite{grohs1} and a few classical results we will need later on.

\autoref{sec:benchmark} deals with a minimalistic introduction to $N$-term approximation and how one can determine the best theoretically possible approximation rate of \emph{any} discretisation for a given class of functions $\FC$.

The core of the thesis is contained in \autoref{sec:approx}; aside from the differences in the ridgelet constructions (see discussion after \autoref{th:approx_intro}) and the different techniques necessary to treat them, we are able to follow the structure of \cite{mutilated} relatively closely --- establishing the localisation of the ridgelet coefficients for a function cut off at a hyper plane first in angle in \autoref{sec:loc_angle}, and then in space in \autoref{sec:loc_space}, before we proceed to the proof of the main result, \autoref{th:approx}, in \autoref{ssec:proof_approx}.

The main part of the paper wraps up with the conclusion in \autoref{sec:conclusion}, while the postponed proof of the crucial estimates in \autoref{ssec:key_est} follows in \autoref{sec:int_est}.

\subsection{Notation}

This subsections lists the most important conventions we will use throughout this paper. As usual, we conclude proofs by $\Box$; additionally, we mark the end of definitions and remarks by $\triangle$.

The letter $\bbN$ denotes the natural numbers \emph{without zero}, while $\bbN_0$ includes it. Similarly, $\bbR^+:=(0,\infty)$, while $\bbR^+_0:=[0,\infty)$.

We let $B_X(x,r):=\{x'\in X:\, \dist_X(x,x')<r\}$ be the open ball in the metric space $X$. Occasionally we omit the space if it is clear from the context. To distinguish the Euclidian norm from the other norms, we denote it by $|\vec x|$. The inner product on $\bbR^d$ is simply denoted by $\vec x \cdot \vec x'$, all other inner products are denoted by $\inpr[\CH]{\cdot,\cdot}$, where the \emph{first} argument is antilinear and the \emph{second} is linear (which is closer to the interpretation as a functional (see e.g. Bra-ket notation) and has several advantages, in our opinion).

The Fourier transform we use is
\begin{align*}
	\hat f(\vec\xi) := \bracket*{\CF (f)} (\vec\xi) := \int_{\bbR^d} f(\vec x)\ee^{-2\pi\ii \vec x \cdot \vec\xi} \d \vec x,
\end{align*}
where we will mostly omit the square brackets for improved legibility if the second term has to be used. In order to limit the amount of constants we have to carry, we define the following relation,
\begin{align*}
	A(y) \lesssim B(y) \, :\Longleftrightarrow \, \exists\, c>0:\, A(y)\le c B(y),
\end{align*}
where the constant has to be independent of $y$.
Similarly, $A\sim B$ denotes the case that both $A\lesssim B$ and $B\lesssim A$ hold.

For vector variables (and occasionally multi-indices), $i$ primes ($i=1,\ldots,3$) will always indicate the last $d-i$ components of that vector, i.e. $\vec k{}''=(k_3,k_4,\ldots,k_d)^\top\in\bbR^{d-2}$.

Square brackets around a vector --- i.e. $[\vec e_1]$ --- denote the linear span, while $[\vec e_1]^\bot$ denotes its orthogonal complement. The orthogonal projection along $\vec n$ is denoted by $\CP_{\vec n}$.

Finally, we let $H(y):=\ind_{\bbR^+}(y)$ denote the Heaviside step function, and define the \emph{regularised absolute value} $\langle\vec x\rangle:=\sqrt{1+|\vec x|^2}$ (to avoid problems with division by zero).


\section{Preparations}\label{sec:prelim}

In this section, we set up the foundations on which the rest of the paper will be built. In \autoref{ssec:key_est}, we introduce a crucial estimate that will be necessary later on (but whose proof we postpone to \autoref{sec:int_est}), while in \autoref{ssec:prop_advec}, we deal with the properties of the advection equation and the model class of solutions we will consider. In \autoref{ssec:ridgeframes}, we briefly recall the ridgelet construction of \cite{grohs1} in the necessary detail to prove our results, and we wrap up this section with some classical results about interpolation in \autoref{ssec:interp}.

\subsection{An Integral Estimate}\label{ssec:key_est}

As one of the key tools for the main proof, we introduce the following integral inequality, the proof of which we postpone to \autoref{sec:int_est}.

\begin{theorem}[{\autoref{app:th:Imn_est}}]\label{th:Imn_est}
	For $m,n\in\bbN$, $a\in\bbR^+_0$, $b\in\bbR$, $c,d\in\bbR^+$, we have
	\begin{align}
		I_{m,n}&:=\int_{-\infty}^{\infty} \frac{1}{\parens{a^2(x-b)^2+c^2}^m} \frac{1}{\parens{x^2+d^2}^n} \d x \notag \\
		&\phantom{:}= \frac{\pi}{\parens{a^2b^2+(ad+c)^2}^{m+n-1}} \frac{1}{c^{2m-1}} \frac{1}{d^{2n-1}} \sum_{\substack{i+j+2k=2(m+n)-3\\i\ge 2m-1 \, \lor\, j\ge 2n-1}} c^{m,n}_{i,j} c^{i} (ad)^{j} (ab)^{2k}\notag \\
		&\phantom{:}\lesssim \frac{a^{2n-1}}{\parens{a^2b^2+a^2d^2+c^2}^n} \frac{1}{c^{2m-1}} +  \frac{1}{\parens{a^2b^2+a^2d^2+c^2}^m} \frac{1}{d^{2n-1}}. \label{eq:Imn_est}
	\end{align}
	For an explicit representation of the constants $c^{m,n}_{i,j}$, as well as the generating functions of $I_{m,n}$ and $c^{m,n}_{i,j}$, see \autoref{sec:int_est}.
\end{theorem}

As a simple corollary for higher dimensions, we also record the following corollary of \autoref{th:Imn_est}, which is proved in \autoref{sec:int_est} as well.

\begin{corollary}[{\autoref{app:cor:Imn_higher_dim}}]\label{cor:Imn_higher_dim}
	For $m,n\ge\ceil{\frac k2}$ $c,d>0$, we have the following inequality,
	\begin{align}
		\int_{\bbR^d} \frac{1}{\parens*{|\vec x - \vec t|^2 +c^2}^m} \frac{1}{\parens*{|\vec x|^2+d^2}^n} \d \vec x  \lesssim \frac{1}{\parens*{|\vec t|^2+c^2+d^2}^n}\frac{1}{c^{2m-k-1}}+\frac{1}{\parens*{|\vec t|^2+c^2+d^2}^m}\frac{1}{d^{2n-k-1}}.
	\end{align}
\end{corollary}

\subsection{The Advection Equation $\&$ Mutilated Functions}\label{ssec:prop_advec}

To rotate the transport direction $\vec s$ in \eqref{eq:LinTrans} --- resp. the normals of the hyperplanes appearing in \autoref{def:Ht_except_hyp} --- into a canonical unit vector, we need the following rotation matrices.

\begin{definition}\label{def:rot_Rs}
	For any vector $\vec s\in\bbSd$, let $\Rs$ be a matrix that maps $\vec s$ to $\vec e_1=(1,0,\ldots)^\top$, and let $\Rs^{-1}=\Rs^\top$ be its inverse. This rotation is not unique in dimensions $d\ge3$, however, the ambiguity will be irrelevant. We also define the respective pullbacks for $f\in L^2(\bbR^d)$ by
	\begin{align}
		\rs f(\vec x):= f(\Rs^{-1} \vec x), \qquad \rs^{-1} f(\vec x) := f(\Rs \vec x),
	\end{align}
	thus (for continuous $f$), $\rs f(\vec e_1) =f (\vec s), \; \rs^{-1} f(\vec s) = f(\vec e_1)$.
\end{definition}

\begin{remark}
	Using these pullbacks, it's easy to see that
	\begin{align}\label{eq:trans_dir_e1}
		\vec s \cdot \nabla u + \kappa u= f\quad \Longleftrightarrow \quad \vec e_1 \cdot \nabla \rs u  + \rs \kappa \rs u= \rs f.
	\end{align}
	This makes an explicit calculation possible (see \cite[Sec. 1]{compress}), namely that with
	\begin{align}\label{eq:lintrans_expl_sol}
		y(x_1,\xp) := \ee^{-K(x_1,\xp)} \int_{-\infty}^{x_1} \rs f (t,\xp) \ee^{K(t,\xp)} \d t , \quad \text{where } \quad K(t,\xp)= \int_0^t \rs \kappa (r,\xp) \d r,
	\end{align}
	setting $u:=\rs^{-1}y$ yields an explicit solution to \eqref{eq:LinTrans} for arbitrary $f\in L^2(\bbR^d)$, as long as $0<\kappa_0 \le \kappa(\vec x) < \infty$ almost everywhere. With the $y$ from \eqref{eq:lintrans_expl_sol}, we define the solution operator for $A$ as follows,
	\begin{align}
		\CS[f]:=\rs^{-1}y, \qquad \CS\colon L^2\to \Hs.\tag*{\qedhere}
	\end{align}
\end{remark}

The following theorem shows that $\CS$ is bounded from $H^t$ to itself.

\begin{proposition}[{\cite[Thm. 2.2]{compress}, \cite[Thm. 3.3.3]{my_thesis}}]\label{prop:sol_op_bounded}
	For $\kappa\in H^{\ceil{t}+\frac d2}(\bbR^d)$ with $t\ge 0$, such that $\kappa\ge \gamma>0$, the operator $\CS$ and is bounded (at least) from $H^t\to H^t$. Furthermore
	\begin{align}\label{eq:Hs_elliptic}
		\norm{\CS[f]}_{\Hs}\sim \norm{Au}_{L^2}=\norm{f}_{L^2}.
	\end{align}
\end{proposition}

However, what we are mainly interested in in this paper are functions of the following form.

\begin{definition}\label{def:Ht_except_hyp}
	We say that a function $f$ is in $H^t$ \emph{except for $N$ hyperplanes} if there are $N$ hyperplanes $h_i$ with corresponding (normalised) orthogonal vectors $\vec n_i$ and offsets $v_i$, as well as functions $f_0,f_1,\ldots, f_N\in H^t(\bbR^d)$ such that
	\begin{align}\label{eq:Ht_except_hyp}
		f(\vec x)=f_0(\vec x)+\sum_{i=1}^N f_i(\vec x) H(\vec x \cdot \vec n_i - v_i),
	\end{align}
	where, as mentioned, $H$ is the Heaviside step function. To abbreviate the concept notationally, we will sometimes write $f\in H^t(\bbR^d\setminus\{(h_i)_{i=1}^N\})$, or $f\in H^t(\bbR^d\setminus\{h_i\})$ for short. This is justified in the sense that --- by factoring with the right choice of equivalence relation --- these are in fact Hilbert spaces again (see \cite[Rem. 3.2.2]{my_thesis}).
\end{definition}

The analysis of the ridgelet coefficients of such a function will require calculating the Fourier transform of such mutilated Sobolev functions, which --- for each term --- splits into a regular and a singular contribution (from the function and the cut-off, respectively). This follows immediately from \cite[Eq. (3.3)]{mutilated}, but we formulate it in the version we will use later on.

\begin{lemma}\label{lem:fourier_decomp}
	The Fourier transform of a function $f(\vec x)=H(\vec x\cdot \vec n-v) g(\vec x)$, where $g\in H^t$ and $t> \frac 12$, in terms of a singularity aligned with $\vec e_1$ is
	\begin{align}\label{eq:sing_fourier_rot}
		\hat f(\Rn^{-1} \vec \xi)
		&= - \frac{\ii}{2\pi\abs*{\vec\xi}^2} \vec \xi \cdot \CF[H(x_1-v)\nabla g(\Rn^{-1}\vec x)](\vec \xi) - \frac{\ii\, \xi_1}{2\pi\abs*{\vec\xi}^2} \wh{g\normalr|_{h}}(\CP_{\vec e_1}\vec \xi),
	\end{align}
	where $g\normalr|_{h}\in H^{t-\frac 12}(\bbR^{d-1})$ is the restriction to the hyperplane $h=\set{\vec x\in\bbR^d}{\vec x\cdot\vec  n=v}$. Furthermore, $\CP_{\vec n}$ is the orthogonal projection along $\vec n$ and we identify $\CP_{\vec n}\vec \xi$ with $\xip_h\in \bbR^{d-1}$, while $\wh{g\normalr|_{h}}$ is shorthand for $\CF_{\bbR^{d-1}}[g\normalr|_{h}]$.
\end{lemma}

Since we are dealing with half-spaces due to the cut-off with the Heaviside function, the following extension result (see e.g. \cite[Sec. 4.5]{triebel2}) will be useful.

\begin{theorem}\label{th:half_space_ext}
	For a function $f\in H^t$, the restriction to the half-space $\set{\vec x\in\bbR^d}{x_1>0}$ can be extended to the full space in a bounded fashion, i.e. there is $g\in H^t$ such that
	\begin{align}
		\norm{g}_{H^t(\bbR^d)}  \lesssim \norm{f}_{H^t(\set{\vec x\in\bbR^d}{x_1>0})} \qquad \text{as well as} \qquad f(\vec x) = g(\vec x) \quad \forall \vec x \in \set{\vec x\in\bbR^d}{x_1>0}.
	\end{align}
	Obviously, by rotating and translating $f$ (which leaves the norms invariant), this result also holds for half-spaces separated by arbitrary hyperplanes.
\end{theorem}

\begin{remark}\label{rem:hyp_normals_neg}
	Using \autoref{th:half_space_ext}, we see that, for arbitrary $\vec s$, we can restrict the representation of \eqref{eq:Ht_except_hyp} to hyperplanes satisfying $\vec s \cdot \vec n_i \le 0$, since, if $\vec s \cdot \vec n_i>0$, we can extend the mutilated $f_i$ to the full space --- absorbing it into $f_0$ --- and subtracting its extension on the other side (hence, the vector $n_i$ in the Heaviside function flips signs).
\end{remark}

The following result extends the boundedness of $\CS$ from \autoref{prop:sol_op_bounded} to the spaces $H^t(\bbR^d\setminus\{h_i\})$, which can be seen as confirmation that this class of functions is well-chosen for the behaviour of \eqref{eq:LinTrans}.

\begin{proposition}\label{prop:sol_smooth_except_hyp}
	For $f\in H^t(\bbR^d\setminus\{h_i\})$, the solution $u=\CS[f]\in H^t(\bbR^d\setminus\{h_i\})$ is of the same form \eqref{eq:Ht_except_hyp}, and furthermore, $\norm{u_i}_{H^t} \lesssim \norm{f_i}_{H^t}$ for $i=1,\ldots,N$, as well as $\norm{u_0}_{H^t}\lesssim \sum_{i=0}^N \norm{f_i}_{H^t}$. In other words, the solution operator $\CS$ is bounded on this space,
	\begin{align}
		\norm{\CS}_{H^t(\bbR^d\setminus\{h_i\})\to H^t(\bbR^d\setminus\{h_i\})}<\infty.
	\end{align}
\end{proposition}

\begin{proof}
	We begin by noting that, since the differential equation \eqref{eq:LinTrans} is linear, it suffices to deal with one term $f(\vec x) H(\vec x \cdot \vec n-v)$, and the rest follows via superposition. Furthermore, using \eqref{eq:trans_dir_e1} and \autoref{rem:hyp_normals_neg}, we can assume without loss of generality that $\vec s=\vec e_1$, and that $(\vec n)_1\le 0$. 
	
	Assuming $(\vec n)_1<0$ (i.e.~a strict inequality) for the moment, we use \eqref{eq:lintrans_expl_sol} with this right-hand side, and observe that, if $x_1$ is before the cut-off, the Heaviside function does not matter, and for everything beyond that, the integral goes just until the interface of the hyperplane,
	\begin{align}
		u(\vec x)
		&= \ee^{-K(\vec x)} \int_{-\infty}^{x_1} f (t,\xp) H\parens*{(t,\xp)^\top \cdot \vec n -v} \ee^{K(t,\xp)} \d t\\
		&=H(\vec x\cdot \vec n-v) \ee^{-K(\vec x)} \int_{-\infty}^{x_1} f (t,\xp) \ee^{K(t,\xp)} \d t +H(v-\vec x\cdot \vec n) \ee^{-K(\vec x)} \int_{-\infty}^{\frac{v-\xp \cdot \vec n'}{(\vec n)_1}} f (t,\xp) \ee^{K(t,\xp)} \d t\\
		&=H(\vec x\cdot \vec n-v) \CS[f](\vec x) +H(v-\vec x\cdot \vec n) \CS[f]\binom{\frac{v-\xp \cdot \vec n'}{(\vec n)_1}}{\xp} \ee^{K\parens*{(\frac{v-\xp \cdot \vec n'}{(\vec n)_1},\xp)^\top}} \ee^{-K(\vec x)}.
	\end{align}
	Due to \autoref{prop:sol_op_bounded}, $\CS[f]\in H^t(\bbR^d)$, with norm bounded by $\norm{f}_{H^t}$. In order for the second term to appear, $v-\vec x\cdot \vec n$ has to be greater than zero, and therefore (since $(\vec n)_1<0$) we have $x_1>\frac{v-\xp \cdot \vec n'}{(\vec n)_1}$. As the absorption satisfies $\kappa\ge\gamma>0$, $K(\cdot,\xp)$ is (strictly) monotonically increasing, and this implies that the term $\ee^{K\parens*{(\frac{v-\xp \cdot \vec n'}{(\vec n)_1},\xp)^\top}} \ee^{-K(\vec x)}$ is bounded. Since the coordinate transformation in the first component is linear, we conclude that
	\begin{align}
		\CS[f]\binom{\frac{v-\xp \cdot \vec n'}{(\vec n)_1}}{\xp} \ee^{K\parens*{(\frac{v-\xp \cdot \vec n'}{(\vec n)_1},\xp)^\top}-K(\vec x)} \in H^t\parens*{\set*{\vec x\in\bbR^d}{\vec x\cdot \vec n-v<0}},
	\end{align}
	by since multiplication with smooth enough $\kappa$ is bounded from $H^t$ to itself (which can be showing that the derivatives remain in $L^2$ and interpolation), and noting that the exponential function does not decrease the smoothness. Again by \autoref{th:half_space_ext}, we can extend this term to $y\in H^t(\bbR^d)$, and thus
	\begin{align}
		u(\vec x) &= H(\vec x\cdot \vec n-v) \parens*{\CS[f](\vec x)-y(\vec x)} + y(\vec x).
	\end{align}
	Due to the boundedness of $\CS$ on $H^t$ and $\norm{y}_{H^t} \lesssim \norm{f}_{H^t}$, this implies $\norm{u}_{H^t(\bbR^d\setminus h)}\lesssim \norm{f}_{H^t(\bbR^d\setminus h)}$.
	
	In the last remaining case --- that $(\vec n)_1=0$ (i.e.~$\vec n \bot \vec s$) --- the path of the integral never crosses the hyperplane, and thus $u=H(\vec x\cdot \vec n-v) \CS[f](\vec x)$.
\end{proof}

Lastly, we discuss how decay properties transfer to the solution.

\begin{lemma}\label{lem:decay_sol}
	If $f\in L^2$ satisfies
	\begin{align}
		\abs{f(\vec x)} &\lesssim \reg{\vec x}^{-2n},
	\end{align}
	then the solution $u=\CS[f]$ to \eqref{eq:LinTrans} also satisfies
	\begin{align}
		\abs{u(\vec x)} &\lesssim \reg{\vec x}^{-2n}.
	\end{align}
\end{lemma}

\begin{proof}
	Using the solution formula from \eqref{eq:lintrans_expl_sol} (for $\vec s =\vec e_1$), we calculate
	\begin{align}
		\abs{u(\vec x)} &\lesssim \int_{-\infty}^{x_1} \reg*{(t,\xp)^\top}^{-2n} \ee^{K(t,\xp)-K(x_1,\xp)} \d t
		\le \int_{-\infty}^{x_1} \reg*{(t,\xp)^\top}^{-2n} \ee^{-\gamma(x_1-t)} \d t,
	\end{align}
	since the strict monotonicity of $K$ (due to $\gamma>0$) implies $K(t,\xp)-K(x_1,\xp)\le \gamma(t-x_1)$ for $t\le x_1$. At this point, we use the fact that for arbitrary $m$,
	\begin{align}
		\ee^{-\gamma y} \lesssim \reg{y}^{-2m}, \qquad \text{for} \qquad y\ge 0.
	\end{align}
	Inserting this into the above and resolving the square root from $\reg{y}=\sqrt{1+|y|^2}$, we continue
	\begin{align}
		\abs{u(\vec x)}
		&\lesssim \int_{-\infty}^{x_1} \reg*{(t,\xp)^\top}^{-2n} \reg{x_1-t}^{-2m} \d t \le \int_{-\infty}^{\infty} \frac{1}{(t^2+|\xp|^2+1)^{n}} \frac{1}{((t-x_1)^2+1)^{m}} \d t \\
		&\lesssim \frac{1}{(x_1^2+|\xp|^2+2)^n} + \frac{1}{(x_1^2+|\xp|^2+2)^m}\frac{1}{(|\xp|^2+1)^{n-\frac 12}}
		\le \frac{1}{(|\vec x|^2+1)^n} = \reg{\vec x}^{-2n},
	\end{align}
	where the last line uses \eqref{eq:Imn_est}, and $m\ge n$ (which we we are free to choose this large).
\end{proof}

\subsection{The Ridgelet Construction \cite{grohs1}}\label{ssec:ridgeframes}

We recall the definition of a ridgelet frame from \cite{grohs1,compress}, where details on the construction can be found. The starting point is a (sufficiently) smooth partition of unity in frequency (constructed from a given window function),
\begin{align}\label{eq:partititon_unity}
	\parens*{\psi_\jl}_{j\in\bbN_0,\, \ell\in\{0,\ldots,L_j\}} \quad \text{such that} \quad \sum_{j=0}^\infty \sum_{\ell=0}^{L_j} \hat \psi_\jl^2 = 1,
\end{align}
supported on (approximate) polar rectangles
\begin{align}\label{eq:supp_Pjl}
	\supp \hat \psi_\jl 
	&\subseteq P_\jl:= \set[\Big]{
	\vec \xi \in \bbR^d}{ 2^{j-1}<|\vec \xi| < 2^{j+1}, \, \smash{\frac{\vec \xi}{|\vec \xi|}} \in B_\bbSd(\vec s_\jl, 2^{-j+1})}, 
\end{align}
where the vectors $\vec s_\jl$ are an approximately uniform sampling of points on the sphere for each scale $j$, with average distance $\sim 2^{-j}$ and cardinality $\#\{\vec s_\jl\}_{\ell} \lesssim 2^{j(d-1)}$, see \cite{borup} for the construction or \cite{grohs1,compress} for more details on their relation to ridgelets.

\begin{definition}\label{def:phi_jlk}
	Using the functions $\hat \psi_\jl$, a Parseval frame for $L^2(\bbR^d)$ is defined by 
	\begin{align}
		\varphi_\jlk = \, 2^{-\frac{j}{2}} T_{U_\jl \vec k}\,\psi_\jl, \quad
		j\in\bbN_0, \, \ell \in \{0,\ldots,L_j\}, \, \vec k \in \bbZ^d, 
	\end{align}
	with $T$ the translation operator, $T_{\vec y}f(\cdot) := f(\cdot-\vec y)$, and $U_\jl:= \Rjl^{-1} D_{2^{-j}}$, where $\Rjl:=R_{\vec s_\jl}$ is the transformation introduced at the beginning of the section, and $D_a$ dilates the first component, $D_a \vec k := (a \, k_1, k_2,\ldots,k_d)^\top$. The rotation $\Rjl$ is arbitrary (to the extent that it is ambiguous) but fixed. Whenever possible, we will subsume the indices of $\varphi$ by $\lambda=(\jlk)$.
\end{definition}

\begin{assumption}\label{assump:psi_smooth}
	The window functions 
	\begin{align}\label{eq:psi_trafo}
		\hat \psi_{(\jl)} (\vec \eta):= \hat\psi_\jl(U_\jl^{-\top}\vec \eta) 
	\end{align}
	have bounded derivatives \emph{independently} of $j$ and $\ell$. Thus, for all $n$ up to an upper bound $N$ dependent on the differentiability of the window functions (or possibly for all $n\in\bbN$ if the window functions are $\CC^\infty$), we have the estimate
	\begin{align}
		\norm*{\hat \psi_{(\jl)}}_{\CC^n} \le\beta_{n}.
	\end{align}
	In \cite[Lem. B.1]{compress}, it is shown that this assumption can be satisfied with a reasonable (and still quite flexible) choice of window functions.
\end{assumption}

\begin{definition}
	We consider the diagonal matrix
	\begin{align}
		\VW=(\VW_{\lambda,\lambda'})_{\lambda,\lambda'\in\Lambda} \quad \text{with} \quad \VW_{\lambda,\lambda'}:= \begin{cases}
		0, & \lambda\neq\lambda',\\
		1+2^j \abs{\vec s\cdot \vec s_\jl}, & \lambda=\lambda',
	\end{cases}
	\end{align}
	which is the right choice of weight to make $\Phi$ a frame for $\Hs$ (see \cite[Thm. 10]{grohs1}, \cite[Thm. 4.3]{compress}), i.e.
	\begin{align}\label{eq:RidgeletFrame}
		\norm{f}_{\Hs} \sim \norm{\VW \inpr{\Phi,f}}_{\ell^2}.
	\end{align}
\end{definition}

\subsection{Interpolation between Spaces}\label{ssec:interp}

We recall the scale of sequence spaces necessary for our analysis, as well as some classical interpolation results in the context we will work in (the mentioned sequence spaces, Sobolev spaces, and operators between them).

\begin{definition}
	The \emph{Lorentz spaces} of sequences are defined as follows (\cite[p. 89]{DeVore1998}),
	\begin{align}
		\ell^p_q:=\set[\bigg]{(c_n)_{n\in\bbN}}{\snorm{c_\bbN}_{\ell^p_q}:=\parens[\bigg]{\sum_{n\in\bbN}(c^*_n)^q n^{\frac{q}{p}-1}}^{\frac{1}{q}}<\infty},
	\end{align}
	where $c^*_n$ is the decreasing rearrangement of $(\abs{c_n})_{n\in\bbN}$, $0< p <\infty$ and $0< q < \infty$. For $q=\infty$, we have
	\begin{align}
		\ell^p_\infty:=\set{(c_n)_{n\in\bbN}}{\snorm{c_\bbN}_{\ell^p_\infty}:=\sup_{n\in\bbN} c^*_n n^{\frac 1p} <\infty}=:\ell^p_w,
	\end{align}
	which is also called the \emph{weak $\ell^p$-space}. It should be noted that, technically, the norms involved are generally only quasinorms, i.e.
	\begin{align}\label{eq:lpw_triangle_ineq}
		\snorm{a_\bbN+b_\bbN}_{\ell^p_q}\le K_{p,q} \parens*{\snorm{a_\bbN}_{\ell^p_q}+\snorm{b_\bbN}_{\ell^p_q}}.
	\end{align}
	Additionally, it can be shown that $\ell^p_p=\ell^p$, as well as $\ell^p_{q_1}\subseteq \ell^p_{q_2}$ for $q_1\le q_2$.
\end{definition}

\begin{theorem}\label{thm:interp}
	For Sobolev spaces with $t>0$, the real interpolation with $L^2$ depending on $0<\theta<1$ yields (\cite[Thm. 6.4.5]{interp_sp})
	\begin{align}\label{eq:interp_sob}
		(L^2,H^t)_{\theta,2}=H^{t\theta}.
	\end{align}
	On the sequence space side, we have that (\cite[Thm. 5.3.1]{interp_sp})
	\begin{align}\label{eq:interp_seq_sp}
		(\ell^2,\ell^p_w)_{\theta,2}=\ell^{\bar{p}}_2, \quad \text{where} \quad \bar{p}=\parens[\Big]{\frac{1-\theta}{2} + \frac{\theta}{p}}^{-1}.
	\end{align}
	Finally, for compatible couples\footnote{I.e. both spaces are linear subspaces of a larger Hausdorff topological vector space and the embeddings are continuous, c.f. \cite[Def. 3.1.1]{interp_op}} $(X_0,X_1)$ and $(Y_0,Y_1)$ and an admissible linear operator $T$, i.e.
	\begin{align}
		T:X_0+X_1 \to Y_0+Y_1, \quad \text{such that} \quad T\bigr|_{X_i}:X_i\to Y_i \quad \text{and} \quad \norm{T}_{X_i\to Y_i}<\infty \quad \text{for}\quad  i=0,1,
	\end{align}
	it holds that (\cite[Thm. 3.1.2]{interp_sp}, \cite[Thm. 5.1.12]{interp_op}), for all $\theta\in(0,1)$,
	\begin{align}\label{eq:interp_op}
		\norm{T}_{(X_0,X_1)_{\theta,q}\to(Y_0,Y_1)_{\theta,q}}< \infty.
	\end{align}
\end{theorem}

\section{$N$-Term Approximation $\&$ Benchmark Rates}\label{sec:benchmark}

We briefly recall some core concepts of (non-linear) approximation theory, see \cite{DeVore1998} for a survey.

\begin{definition}\label{def:nonlin_basics}
	Consider a (relatively) compact class $\FC\subseteq \CH$, where $\CH$ is a separable Hilbert space, as well as a dictionary $\Phi=(\varphi_\lambda)_{\lambda\in\Lambda}\in\CH$. To measure the mentioned quality of approximation in dependence of $N$, we introduce the set
	\begin{align}
		\Sigma_N(\Phi):=\set[\bigg]{f=\sum_{\lambda\in\Lambda_N} c_\lambda \varphi_\lambda}{\#\Lambda_N\le N},
	\end{align}
	which is non-linear (as $f,g\in \Sigma_N(\Phi)$ generally don't imply $f+g\in\Sigma_N(\Phi)$) --- consequently, analysis of the quantity
	\begin{align}
		\varsigma_N(f,\Phi):=\inf_{g\in \Sigma_N(\Phi)} \norm{f-g}_\CH
	\end{align}
	is referred to as the study of \emph{non-linear approximation}.
\end{definition}

Comparing, $\varsigma_N(f,\Phi)$ with $N^{-\alpha}$, we have the following definition of approximation spaces (which can be done in much greater generality, see \cite[Sec. 4.1]{DeVore1998}).

\begin{definition}
	Let
	\begin{align}
		\CA^\alpha_q(\Phi):=\set{f\in \CH}{\norm{f}_{\CA^\alpha_q}:=\snorm{f}_{\CA^\alpha_q}+\norm{f}_\CH <\infty},
	\end{align}
	where
	\begin{align}
		\snorm{f}_{\CA^\alpha_q} :=
		\begin{cases}
			\parens{\sum_{n\in\bbN} \parens*{n^\alpha \varsigma_n(f,\Phi)}^q \frac 1n}^{\frac 1q}, & 0<q<\infty,\\
			\sup_{n\in\bbN} n^\alpha \varsigma_n(f,\Phi), & q=\infty.
		\end{cases}\tag*{\qedhere}
	\end{align}
	%
\end{definition}

These spaces are fairly well-understood --- especially if $\Phi$ forms a basis for $\CH$ --- and we will see that the similarity to the definition of Lorentz spaces is more than incidental.

\begin{theorem}
	For $\alpha>0$ and $0<p<q\le \infty$, we have $\CA^\alpha_p(\Phi)\subseteq \CA^\alpha_q(\Phi)$. More importantly, the approximation spaces are interpolation spaces; for $0<\alpha<r$ and $0<p,q\le \infty$, we have (\cite[Thm. 2]{DeVore1998})
	\begin{align}
		\CA^\alpha_q(\Phi)=(\CH,\CA^r_p(\Phi))_{\frac{\alpha}{r},q}.
	\end{align}
	Finally, if $\Phi$ is a basis for $\CH$, we have the following complete characterisation of the approximation spaces $\CA^\alpha_q(\Phi)$, namely that (\cite[Thm. 4]{DeVore1998})
	\begin{align}\label{eq:approx_charact_basis}
		f\in \CA^\alpha_q(\Phi) \quad \Longleftrightarrow \quad \inpr{\Phi,f}:=\parens{\inpr{\varphi_\lambda,f}}_{\lambda\in\Lambda} \in \ell^{p(\alpha)}_q \quad \text{where} \quad p(\alpha)=\parens[\Big]{\alpha+\frac 12}^{-1}.
	\end{align}
\end{theorem}

For arbitrary dictionaries, \eqref{eq:approx_charact_basis} does not hold in general, but at least we can salvage one implication for frames (follows from \cite[Lem. 3.1]{kutyniok_lemvig}, which is carried out in \cite[Prop. 2.3.5]{my_thesis}; see also \cite[Sec. 8]{DeVore1998}).

\begin{proposition}\label{prop:wlp_Nterm_frame}
	For a frame $\Phi$ with canonical dual frame $\wt \Phi$, we have the implication that
	\begin{align}
		\inpr{\Phi,f} \in \ell^{p(\alpha)}_w \quad \text{with} \quad p(\alpha)=\parens[\Big]{\alpha+\frac 12}^{-1} \quad \Longrightarrow \quad f\in \CA^\alpha(\wt\Phi).
	\end{align}
\end{proposition}

Consequently, we can ask how well a given dictionary approximates $\FC$. This is easily done by defining
\begin{align}
	\sigma^*(\FC,\Phi):=\sup\set{\sigma>0}{\FC\subseteq A^\sigma(\wt \Phi)}.
\end{align}
The more interesting question is, how well \emph{any} dictionary could possibly approximate $\FC$ (in terms of $N$-term approximation). This can be done very abstractly with encoding/decoding schemes, see \cite{hypercubes_don,hypercubes}.

\begin{definition}\label{def:encoding_rate}
For a class  of signals $\FC\subseteq \CH$ with $\CH$ a separable Hilbert space, we define:
	\begin{itemize}
	\item
		An \emph{encoding/decoding pair} $(E,D)\in \CE\CD(R)$ consists of two mappings
		\begin{align}
			E:\FC\to \{0,1\}^R, \quad D:\{0,1\}^R\to \CH,
		\end{align}
		where we call $R$ the \emph{runlength} of $(E,D)$ and $\CE\CD:=\bigcup_{R\in\bbN} \CE\CD(R)$ is the set of all such pairs.
	\item
		The \emph{distortion} of $(E,D)$ is defined as
		\begin{align}
			\delta(E,D):=\sup_{f\in\FC} \norm{ f-D(E(f))}_\CH.
		\end{align}
	\item
		The \emph{optimal encoding rate} is defined as
		\begin{align}
			\sigma^*(\FC):= \sup\set{\sigma>0}{ \exists C>0 \;\forall N\in \bbN \;\exists (E_N,D_N)\in \CE\CD(N): \delta(E_N,D_N)\le C N^{-\sigma}}\tag*{\qedhere}
		\end{align}
	\end{itemize}
\end{definition}

\begin{remark}
	The quantity $\sigma^*(\FC)$ limits the best approximation rate of $\FC$ for any decoding scheme, in particular, for the discretisation with \emph{any} dictionary (as long as the natural restriction of polynomial depth search\footnote{This is to exclude pathological behaviour --- namely, choosing arbitrary functions from a countable set $\Phi$ which is dense in $\CH$ This would give a perfect $1$-term approximation, but no way to efficiently compute the approximation or even store the dictionary. Polynomial depth search means that the set $\Lambda_N$ we select for our $N$-term approximation may only search through the first $\pi(N)$ dictionary elements, where $\pi$ is an arbitrary polynomial.} is imposed), i.e.
	\begin{align}
		\sigma^*(\FC,\Phi)\le \sigma^*(\FC).
	\end{align}

	It is also intimately related to the \emph{Kolmogorov} (or metric) $\eps$-entropy (see e.g. \cite{kolmog_entr}) --- which we will denote by $H_\eps(\FC)$ --- in the sense that (\cite[Rem. 5.10]{hypercubes})
	\begin{align}
		\sigma^*(\FC)=\sup \set[\Big]{\sigma>0}{\sup_{\eps>0} \eps^{\frac 1\sigma} H_\eps(\FC) < \infty}.\tag*{\qedhere}
	\end{align}
\end{remark}

The interesting thing --- despite its very general definition --- is that $\sigma^*(\FC)$ can be estimated with the help of the following definition and \autoref{thm:hypercubes}.

\begin{definition}\label{def:hypercubes}
	Let $\FC\subseteq \CH$ be a signal class in a separable Hilbert space $\CH$.
	\begin{itemize}
	\item
		We say $\FC$ \emph{contains an embedded orthogonal hypercube of dimension $m$ and sidelength $\delta$} if there exist $f_0\in\FC$ and orthogonal functions $(\psi_i)_{i=1}^m\in \CH$ with $\norm{\psi_i}_\CH=\delta$, such that the collection of vertices
		\begin{align}
			\Fh\parens*{f_0,(\psi_i)_{i=1}^m}:=\set[\bigg]{h=f_0+\sum_{i=1}^{m}\eps_i \psi_i}{\eps_i\in\{0,1\}}
		\end{align}
		can be embedded into $\FC$.
	\item
		For $p>0$, $\FC$ is said to \emph{contain a copy of} $\ell^p_0$ if there exists a sequence of orthogonal
		hypercubes $(\Fh_k)_{k\in\bbN}$ with dimensions $m_k$ and sidelength $\delta_k$ embedded in $\FC$, such that $\delta_k\to0$ and for some constant $C>0$
		\begin{align}\label{eq:def_lp0}
			\delta_k\ge C m_k^{-\frac 1p}, \quad \forall k\in\bbN.
		\end{align}
	\end{itemize}
\end{definition}

\begin{remark}\label{rem:hypercubes}
	The motivation behind \autoref{def:hypercubes} is, on the one hand, that we know precisely how many bits we need to encode the vertices of a hypercube, and if $\FC$ contains such hypercubes, then it must be at least as complex.
	
	On the other hand, $\ell^p$ trivially contains orthogonal hypercubes of dimension $m$ and sidelength $m^{-\frac 1p}$; in fact every $\ell^p_q$ with $q>0$ contains a copy of $\ell^p_0$, which partly motivates the choice of notation (since $\ell^p_{q_1}\subseteq \ell^p_{q_2}$ for $q_1\le q_2$).
	
	Even more, since $\ell^{p_1}\subseteq \ell^{p_2}$ for $p_1\le p_2$, it is perhaps not surprising (and easy to check from \eqref{eq:def_lp0}) that $\ell^p_q$ with $q>0$ contains a copy of $\ell^\tau_0$ for all $\tau \le p$.
\end{remark}

As hinted in \autoref{rem:hypercubes}, being able to precisely understand the complexity of hypercubes allowed \cite[Thm. 2]{hypercubes_don} to prove the following landmark result. Since the proof is highly non-trivial, we also refer to \cite[Thm. 5.12]{hypercubes}, where an elementary proof of this result was given.

\begin{theorem}\label{thm:hypercubes}
	If the signal class $\FC\subseteq \CH$ contains a copy of $\ell^p_0$ for $p\in(0,2]$, the optimal encoding rate satisfies
	\begin{align}
		\sigma^*(\FC) \le \frac{2-p}{2p}.
	\end{align}
\end{theorem}

Finally, since we are interested in functions in the Sobolev space $H^t$, computing the optimal encoding rate of this class with the tools we have just introduced is a simple adaptation of \cite[Thm. 5.17]{hypercubes}.

\begin{lemma}[{\cite[Lem. 2.3.11]{my_thesis}}]\label{lem:Ht_complexity}
	The Sobolev space $H^t(\bbR^d)\subseteq L^2(\bbR^d)$ contains a copy of $\ell^{p^*}_0$, where $p^*=\parens*{\frac td + \frac 12}^{-1}$. In particular, $\sigma^*\parens*{H^t(\bbR^d)}\le \frac td$.
\end{lemma}

\begin{remark}\label{rem:benchmark_Ht_ex_hyp}
	The model class of functions $H^t(\bbR\setminus\{h_i\})$ --- see \autoref{def:Ht_except_hyp} --- we will deal with regarding \eqref{eq:LinTrans} consists of functions in $H^t$ that are allowed to have cut-offs across hyperplanes. This is trivially a superset of $H^t$, and so it is clear that
	\begin{align}
		\sigma^*\parens*{H^t(\bbR\setminus\{h_i\})}\le \sigma^*\parens*{H^t(\bbR^d)} \le \frac td.
	\end{align}
	In other words, these ``mutilated'' $H^t$-functions have --- potentially --- even worse approximation rates than functions in $H^t$.
	
	Nevertheless, although the possibilities of achieving this benchmark seem slim (considering the very abstract and extremely general definition of $\sigma^*(\FC)$), there is some hope: there \emph{are} constructions for several classes $\FC$ that \emph{do} achieve $\sigma^*(\FC)$, cf. \cite[Chap. 5]{hypercubes}, and even in the more complicated case of allowing cut-offs, wavelets (for example) can be shown to achieve the best possible approximation rate for piece-wise $\CC^k$-functions over an interval, cf. \cite[Cor. 5.36]{hypercubes}, and point-like singularities in $\bbR^d$ in general.
	
	To summarise, we have now found the definitive benchmark to aim for, since, in view of \autoref{prop:wlp_Nterm_frame}, we know that $\inpr{\Phi,f}\in \ell^{p^*}_w$ with $p^*=\parens*{\frac td + \frac 12}^{-1}$ is the \emph{absolute best} that is (even theoretically) possible for generic functions in $H^t$ (or the superset  $H^t(\bbR\setminus\{h_i\})$) --- any $p<p^*$ would imply an $N$-term approximation rate that we have proved (see \autoref{lem:Ht_complexity}) to be impossible.
\end{remark}

\section{Approximation Properties}\label{sec:approx}

Compared to the discussion in \autoref{sec:introduction}, the only addition we are now able to make is to formulate the goals in terms of conditions on the ridgelet coefficient sequence, as well as a finer statement about the relation between the decay of $f$ \eqref{eq:f_glob_decay} and the deviation $\delta$ in \autoref{th:approx_intro}. \cite[Thm. 4.1]{mutilated} shows that for a function $f\in H^t(\bbR^d\setminus \{h_i\})$,  supported (compactly) on $[0,1]^d$, the ridgelet coefficient sequence $\parens*{\inpr*{\phi_{j,\ell,k},f}}_{j,\ell,k}$  belongs to the $\ell^p_w$ space with the best possible $p=p^*:=\parens*{\frac td + \frac 12}^{-1}$, which implies the optimal $N$-term approximation rate $\frac td$ in the $L^2$-norm.

In contrast, we will show that the coefficient sequence $\inpr{\Phi,f}:=\parens*{\inpr{\varphi_\lambda,f}}_\lambda$ with $f\in H^t(\bbR^d\setminus \{h_i\})$ satisfying the following polynomial decay condition,
\begin{align}
	\abs{f(\vec x)} \lesssim \reg{\vec x}^{-2m},
\end{align}
is in $\ell^{p^*+\frac dm}_w$. In particular, if the above decay holds for all $m\in\bbN$ (which is satisfied for example if $f$ has exponential decay or compact support), then $\inpr{\Phi,f}\in \ell^{p^*+\delta^*}_w$ for arbitrary $\delta^*>0$.

Finally, in the context of CDD-schemes (see \autoref{ssec:cdd}) ---  we need to consider $N$-term approximation with regard to the $\Hs$-norms instead of the $L^2$ norm. Due to \eqref{eq:RidgeletFrame}, this corresponds to bounding the $\ell^p_w$-norm of the weighted coefficient sequence $\VW\inpr{\Phi,f}$, compare also with \autoref{cor:approx} in this regard. Since the weights $w_{\lambda}\sim \reg*{2^j \vec s\cdot \sjl}$ grow with scale $j$, this causes extra work as well.

\subsection{Main Theorem}

We begin straight away by formulating the goal of this section, \autoref{th:approx}. Once we will have proved this result, we will have satisfied ingredient \ref{itm:ingr:4} from \autoref{ssec:cdd}.

\begin{theorem}\label{th:approx}
	Let $f$ be a function in $L^2(\bbR^d)$ such that $f \in H^t(\bbR^d)$ apart from $N$ hyperplanes --- i.e.~of the form \eqref{eq:Ht_except_hyp} --- such that the following decay condition (recalling $\reg{y}=\sqrt{1+|y|^2}$)
	\begin{align}\label{eq:f_glob_decay}
		\abs{f(\vec x)} \lesssim \reg{\vec x}^{-2m}
	\end{align}
	holds for a certain $m\in \bbN$. Then, if $\kappa\in H^{\ceil{t}+\frac d2}$ (for the terms involving $\CS[f]$),
	\begin{align}
		\inpr{\Phi,f}_{L^2} &:= \parens*{\inpr{\varphi_\lambda,f}_{L^2}}_{\lambda\in\Lambda}\in \ell^{p^* + \frac dm}_w,\\
		\VW\inpr{\Phi,\CS[f]}_{L^2} &:= \parens*{w_\lambda\inpr{\varphi_\lambda,\CS[f]}_{L^2}}_{\lambda\in\Lambda}\in \ell^{p^* + \frac dm}_w,
	\end{align}
	i.e.~the (weighted) ridgelet coefficient sequences belong to the $\ell^p_w$-space with $p=p^* + \frac dm$, where $p^*= (\frac td + \frac 12)^{-1}$ would be the best possible $p$ for $f\in H^t(\bbR^d)$, \emph{even without singularities}(!), and the deviation from the optimal value\footnote{For $\inpr{\Phi,f}$, the absence of $\VW$ is not completely without effect, in that the deviation from the optimal $p$ can be bounded by $\frac{d}{2n}+\eps$ with arbitrarily small $\eps>0$.} decays linearly with $m$.
	
	If the decay condition \eqref{eq:f_glob_decay} holds for arbitrarily large $m\in\bbN$ --- which is satisfied for exponential decay or compact support, for example --- then, for arbitrary $\delta^*>0$, $\inpr{\Phi,f}_{L^2} \in \ell^{p^* + \delta^*}_w$ and  $\VW \inpr{\Phi,\CS[f]}_{L^2} \in \ell^{p^* + \delta^*}_w$.
\end{theorem}

We split some essential parts of the proof into separate results, namely the \emph{localisation in angle} of the ridgelet coefficients in \autoref{prop:loc_angle} and --- after proving decay rates for (modified) ridgelets in space in \autoref{lem:decay_ridgelet} --- the \emph{localisation in space} for the response to the singularity in \autoref{prop:loc_space}, as well as the general decay of the coefficients for functions satisfying \eqref{eq:f_glob_decay} in \autoref{lem:ridge_coeff_tail_decay}.

Before we continue, let us briefly state the following corollary of \autoref{th:approx}, recalling the non-linear set $\Sigma_N$ of functions being linear combinations of less than $N$ terms from $\Phi$ (from \autoref{def:nonlin_basics}).

\begin{corollary}\label{cor:approx}
	For $f \in H^t(\bbR^d\setminus \{h_i\})$ such that \eqref{eq:f_glob_decay} is satisfied for all $n\in\bbN$, we have, for arbitrary $\delta>0$,
	\begin{align}
		&\inf_{g\in \Sigma_N(\Phi)} \norm{f-g}_{L^2} \lesssim N^{-\frac{t}{d}+\delta},\\
		&\inf_{g\in \Sigma_N(\Phi)} \norm{\CS[f]-g}_{\Hs} \lesssim N^{-\frac{t}{d}+\delta}.
	\end{align}
\end{corollary}

\begin{proof}
	This statement is a direct consequence of \autoref{prop:wlp_Nterm_frame} and \eqref{eq:RidgeletFrame}, respectively, that $\Phi$ is also a frame (without the weight $\VW$) for $L^2$.
\end{proof}

\subsection{Localisation in Angle}\label{sec:loc_angle}

In \autoref{lem:fourier_decomp}, we considered how the Fourier transform of a function $f(\vec x)=H(\vec x\cdot \vec n-v) g(\vec x)$ splits into two parts, compare \eqref{eq:sing_fourier_rot}. Based on this decomposition, we will likewise split the coefficient sequence in two, where our main interest is with the \emph{singular} contribution arising from the cut-off.

The crucial property of ridgelets we want to explore in this section concerns the fact that large coefficients will only appear for ridgelets that are aligned with the singularity, and we are able to describe the decay of the coefficient in terms of the smoothness of $g$ and the angle between $\sjl$ and $\vec n$.

\begin{proposition}\label{prop:loc_angle}
	Let $f(\vec x)=H(\vec x\cdot \vec n-v) g(\vec x)$ with $g\in H^t$, $t\in\bbN$. Furthermore, let $\kappa\in H^{t+\frac d2}$ and $\kappa\ge \gamma>0$. Then, the sequences
	\begin{align}
		\alpha_\lambda&:=\inpr*{\varphi_\lambda,f}=\inpr*{\hat\varphi_\lambda,\hat f}=\alpha^\rmr_\lambda+\alpha^\rms_\lambda,\\
		\beta_\lambda&:=\inpr*{\varphi_\lambda,\CS[f]}=\inpr*{\hat\varphi_\lambda,\CF[\CS[f]]}= \beta^\rmr_\lambda+\beta^\rms_\lambda,
	\end{align}
	can be split into a \emph{regular contribution} $\alpha^\rmr$ resp.~$\beta^\rmr$ (coming from the function) and a \emph{singular contribution} $\alpha^\rms$ resp.~$\beta^\rms$ (coming from the cut-off). The regular part of both sequences shows decay with scale $j$ proportional to $t$,
	\begin{align}
		\sum_{\ell,\vec k} \abs*{\alpha^\rmr_{\smash[t]{\jlk}}}^2 &\lesssim \eps_j^2 2^{-2jt}\norm{g}_{H^t}^2,\\
		\sum_{\ell,\vec k} \abs*{w_\jlk \beta^\rmr_{\smash[t]{\jlk}}}^2 &\lesssim \eps_j^2 2^{-2jt}\norm{g}_{H^t}^2, \label{eq:gamma_est_reg}
	\end{align}
	where $\sum_j \eps_j^2 \le 1$. Defining the angle $\theta_\jl(\vec n):=\arccos (\sjl\cdot\vec n)$ that the vectors $\sjl$ enclose with $\vec n$ and setting
	\begin{align}\label{eq:Lambda_jr}
		\Lambda^j_r(\vec n):=\begin{cases}
			\set{\ell}{2^{-r}\le \abs*{\sin \theta_\jl(\vec n)} \le 2^{-r+1}}, & 1\le r<j, \\
			\set{\ell}{\abs*{\sin \theta_\jl(\vec n)} \le 2^{-j+1}}, & r=j,
		\end{cases}
	\end{align}
	the singular part of both sequences is \emph{localised in angle} for a given scale $j$ as follows,
	\begin{align}
	\begin{split}\label{eq:gamma_est_sum_Lambda}
		&\sum_{\ell\in\Lambda^j_r(\vec n)} \!\! \sum_{\vec k} \abs*{\alpha^\rms_{\smash[t]{\jlk}}}^2 \lesssim 2^{-j} 2^{-(j-r)(2t-1)}\norm{g}_{H^t}^2, \\
		&\sum_{\ell\in\Lambda^j_r(\vec n)} \!\! \sum_{\vec k} \abs*{w_\jlk \beta^\rms_{\smash[t]{\jlk}}}^2\lesssim 2^{-j} 2^{-(j-r)(2t-1)}\norm{g}_{H^t}^2.
	\end{split}
	\end{align}
\end{proposition}

\begin{proof}
	From \cite{grohs1}, we know that
	\begin{align}
		&\norm{f}_{L^2} \sim \norm{\inpr{\Phi,f}}_{\ell^2}=\norm{\alpha_\Lambda}_{\ell^2} \label{eq:ridge_frame}\\
		&\norm{\CS[f]}_{\Hs} \sim \norm{\VW \inpr{\Phi,\CS[f]}}_{\ell^2}=\norm{\VW \beta_\Lambda}_{\ell^2}, \label{eq:ridge_Hs_stable}
	\end{align}
	where $u:=\CS[f]\in\Hs$ for $f\in L^2$ (cf. \eqref{eq:Hs_elliptic}). We define
	\begin{align}
		P^j_r(\vec n):=\bigcup_{\ell \in \Lambda^j_r(\vec n)} P_\jl,
	\end{align}
	and let $\hat\chi$ be the indicator function of this set. Obviously, $\hat \chi$ has an inverse Fourier transform, which we name $\chi$. Since $\curly{j}\times \Lambda^j_r\times \bbZ^d$ is a subset of all ridgelets intersecting $P^j_r$, we can apply \eqref{eq:ridge_frame} to $f*\chi$ (which, by looking at the Fourier side, is obviously in $L^2$) we arrive at
	\begin{align}\label{eq:alpha_loc_fourier}
		\sum_{\ell\in\Lambda^j_r}\sum_{\vec k} \abs*{\alpha_\lambda}^2
		\lesssim \int \abs*{\hat f \hat \chi}^2 \d \vec \xi
		=\int_{P^j_r(\vec n)} \abs*{\hat f}^2 \d \vec \xi.
	\end{align}
	To get the same (localised!) estimate in Fourier space for second sequence, we need to do a bit more. First, we split the absorption into its parts
	\begin{align}
		Au=\vec s \cdot \nabla u + (\gamma+\kappa_0)u = f \quad \Longleftrightarrow \quad \Ac u:=\vec s\cdot \nabla u+\gamma u=f-\kappa_0 u =:\fc,
	\end{align}
	and observe that $u$ is also the solution of the differential equation $\Ac u=\fc$. From \autoref{prop:sol_smooth_except_hyp} we know that $u(\vec x)=u_0(\vec x)+H(\vec x\cdot\vec  n-v)u_1(\vec x)$ is of the same form as $f$,  with $\norm{u_i}_{H^t}\lesssim\norm{g}_{H^t}$, and again, the multiplication with $\kappa$ also does not pose a problem (see \cite[Prop. 3.3.5]{my_thesis}).
	
	Thus, $\fc(\vec x)= g_0(\vec x) + H(\vec x\cdot\vec  n-v)g_1(\vec x)$ with $\norm{g_i}_{H^t}\lesssim\norm{g}_{H^t}$. The gain of this change of differential equation is that $\Ac$ trivially satisfies the requirements for \eqref{eq:Hs_elliptic}, and therefore $\norm{u}_{\Hs}\sim\norm[]{\Ac u}_{L^2}$ holds for $\Ac$ as well. Now, crucially, due to the constant absorption term in $\Ac$, the operator commutes with convolution, i.e.~$\Ac(u*\chi)= (\Ac u)*\chi=\fc*\chi$ with $\chi$ as above. Thus, using the same argument as for \eqref{eq:alpha_loc_fourier} and applying \eqref{eq:ridge_Hs_stable} for $u*\chi$,
	\begin{align}
		\sum_{\ell\in\Lambda^j_r}\sum_{\vec k} \abs*{w_\lambda\beta_\lambda}^2
		&\lesssim \norm*{u*\chi}_{\Hs}^2 \sim \norm*{\Ac (u*\chi)}_{L^2}^2 = \norm*{\fc*\chi}_{L^2}^2 \\
		&= \int_{P^j_r} \abs*{\hat{\fc}}^2 \d \vec \xi
		\le \int_{P^j_r(\vec n)} \abs*{\hat g_0}^2 \!+ \abs*{\CF[H(\vec x\cdot\vec  n-v)g_1(\vec x)](\vec \xi)}^2\d \vec \xi.
	\end{align}
	The term involving $\hat g_0$ will contribute to the regular part $\VW \beta^\rmr$, and the required inequality \eqref{eq:gamma_est_reg} is easily proved since, for $t\in\bbR^+$, $\norm{h}_{H^t}\sim\norm*{\reg*{\vec \xi}^{t}\hat h(\vec \xi)}_{L^2}$, thus
	\begin{align}
		\int_{P^j_r(\vec n)} \abs*{\hat g_0}^2\d \vec \xi =\int_{P^j_r(\vec n)} \frac{\reg*{\vec \xi}^{2t}}{\reg*{\vec \xi}^{2t}} \abs*{\hat g_0}^2\d \vec \xi \lesssim 2^{-2jt}\int_{P^j_r(\vec n)} \reg*{\vec \xi}^{2t} \abs*{\hat g_0}^2\d \vec \xi  \lesssim 2^{-2jt} \norm{g_0}_{H^t}  \lesssim 2^{-2jt} \norm{g}_{H^t},
	\end{align}
	using the fact that $|\vec \xi|\gtrsim 2^j$ for $\vec \xi \in P^j_r(\vec n)$.
	
	Aside from the contribution above, both sequences have now been dominated by the same kind of integral, and so we can deal with both cases together from now on. We will work with $g$ (and thus the estimate for $\alpha$), but by replacing any occurrence of it with $g_1$ (and recalling $\norm{g_1}_{H^{t}}\lesssim \norm{g}_{H^{t}}$) will automatically yield the estimate for $\VW\beta$.
	
	Setting
	\begin{align}
		f(\vec x)=\tilde f(\Rn \vec x) \quad \text{where} \quad \tilde f(\vec x) := H(x_1-v)\tilde g(\vec x) \quad \text{and} \quad \tilde g(\vec x):=g(\Rn^{-1}\vec x)
	\end{align}
	like in the proof of \autoref{lem:fourier_decomp}, we continue from \eqref{eq:alpha_loc_fourier} as follows. Transforming with $\vec \zeta=\Rn\vec \xi$, we use \eqref{eq:sing_fourier_rot} and the fact that $(a+b)^2\le 2(a^2+b^2)$,
	\begin{align}
		\int_{P^j_r(\vec n)} \abs*{\hat f(\vec \xi)}^2 \d \vec \xi
		&= \int_{\Rn P^j_r(\vec n)} \abs*{\hat f(\Rn^{-1}\vec \zeta)}^2 \d \vec \zeta\\
		&\lesssim \int_{\Rn P^j_r(\vec n)}  \abs*{\vec\zeta}^{-2} \abs*{\CF[H(x_1-v)\nabla \tilde g(\vec x)](\vec \zeta)}^2 + \abs*{\vec\zeta}^{-2} \abs*{\wh{g\normalr|_{h}}(\CP_{\vec e_1}\vec \zeta)}^2 \d \vec \zeta.
	\end{align}
	Each $P_\jl$ making up $P^j_r(\vec n)$ is a polar rectangle (compare \eqref{eq:supp_Pjl}), and with the help of the transform $\CR(r,\vec s):=r\vec s$, we can easily describe effect of $\Rn$,
	\begin{align}
		\Rn P_\jl = \Rn\CR\parens*{[2^{j-1},2^{j+1}]\times B_{\bbSd}(\sjl, 2^{-j+1})} = \CR\parens*{[2^{j-1},2^{j+1}]\times B_{\bbSd}(\Rn\sjl, 2^{-j+1})}.
	\end{align}
	Thus,
	\begin{align}
		\sin\theta_{\Rn\sjl}(\vec e_1) = \arccos (\Rn\sjl\cdot \vec e_1) = \sin\theta_\jl(\vec n) \quad \text{and therefore} \quad \Rn P^j_r(\vec n) = P^j_r(\vec e_1).
	\end{align}
	
	We begin the induction with $t=0$, where the claim follows trivially with $\alpha^\rms_\lambda\equiv0$ resp.~$\beta^\rms_\lambda\equiv0$. Suppose now that the claim holds for $t-1\in \bbN$, i.e.~$g\in H^{t-1}$ and thus also $\tilde g\in H^{t-1}$. Then, by the above, the induction hypothesis implies
	\begin{align}
		\int_{P^j_r(\vec n)} \abs*{\hat f}^2 \d \vec \xi
		&\lesssim \int_{P^j_r(\vec e_1)}  \abs*{\vec\xi}^{-2} \abs*{\CF[H(x_1-v)\nabla \tilde g(\vec x)](\vec \xi)}^2 \d \vec \xi + \! \int_{P^j_r(\vec e_1)} \abs*{\vec\xi}^{-2} \abs*{\wh{g\normalr|_{h}}(\CP_{\vec e_1}\vec \xi)}^2 \d \vec \xi \\
		&\lesssim  \eps_j^2 (\delta^j_r)^2 2^{-j(2t-1)-2j}+ 2^{-3j} 2^{-(j-r)(2(t-1)-1)} + 2^{-2j} \!\! \int_{P^j_r(\vec e_1)} \abs*{ \wh{g\normalr|_{h}} (\CP_{\vec e_1} \vec \xi)}^2 \d \vec \xi, \label{eq:gamma_est_ind_hyp}
	\end{align}
	because, again, $\abs*{\vec \xi}\gtrsim 2^j$ on $P^j_r(\vec e_1)$. The factors $\eps_j$ and $\delta^j_r$ satisfy
	\begin{align}
		\sum_{r=1}^j (\delta^j_r)^2 =1 \qquad \text{and} \qquad \sum_{j=1}^\infty \eps^2_j =1.
	\end{align}
	The last integral can be treated (for $1\le r<j$) as follows
	\begin{align}
		\int_{P^j_r(\vec e_1)} \abs*{ \wh{g\normalr|_{h}} (\CP_{\vec e_1} \vec \xi)}^2 \d \vec \xi
		&= \int_{P^j_r(\vec e_1)} \abs*{ \wh{g\normalr|_{h}} (\CP_{\vec e_1} T(r \cos \theta, r \sin\theta, \phi))}^2 r^{d-1} (\sin\theta)^{d-2}\d r \d \phi \d \theta,
		\intertext{
	where, after having transformed into spherical coordinates, we split off the first component $r\cos \theta$  ($\theta$ being the angle to the $\xi_1$-axis) and treat the rest as $d-1$-dimensional spherical coordinates with radius $r\sin\theta$, i.e.~$T(x,s,\phi)=(x,s\cos\phi_1,s\sin\phi_1\cos\phi_2,\ldots)^\top$. Consequently, we transform with $r'=r\sin\theta$,
		}
		&\lesssim \int_{\CP_\theta P^j_r(\vec e_1)} \frac{2^j}{\abs{\sin \theta}} \int_{\bbS^{d-2}} \int_{2^{j-1}\abs{\sin\theta}}^{2^{j+1}\abs{\sin\theta}} \abs*{\wh{g\normalr|_{h}}(r',\phi)}^2 (r')^{d-2} \d r' \d \phi \d \theta,
	\end{align}
	where one power of $r\lesssim 2^j$ has been estimated, $\abs{\sin \theta}$ is bounded from below due to $r<j$ and $\CP_\theta$ denotes the projection onto the relevant interval $[0,2\pi)$ for $\theta$ (which results in two intervals symmetric around $\pi$ due to the way $\Lambda^j_r$ is defined; we can treat both cases as one since only the absolute value of $\sin\theta$ is involved).
	
	The integrals over $r'$ and $\phi$ can again be interpreted as a $d-1$-dimensional integral, allowing us to continue the estimation due to $g\normalr|_{h}\in H^{t-\frac 12}(\bbR^{d-1})$, for the same reason as above (i.e.~the identification of the $H^t$-norm with a weighted $L^2$-norm in Fourier space),
	\begin{align}
		\int_{P^j_r(\vec e_1)} \abs*{\wh{g\normalr|_{h}} (\CP_{\vec e_1} \vec \xi)}^2 \d \vec \xi
		&\lesssim \int_{\CP_\theta P^j_r(\vec e_1)} \frac{2^j}{\abs{\sin \theta}} \int_{\bbS^{d-2}}\int_{2^{j-1}\abs{\sin\theta}}^{2^{j+1}\abs{\sin\theta}} \abs*{\wh{g\normalr|_{h}}(r',\phi)}^2 (r')^{d-2} \frac{\reg{r'}^{2t-1}}{\reg{r'}^{2t-1}} \d r' \d \phi \d \theta \\
		&\lesssim \int_{\CP_\theta P^j_r(\vec e_1)} \frac{2^j}{\abs{\sin \theta}} \frac{1}{\reg{2^j\sin\theta}^{2t-1}} \int_{\bbR^{d-1}} \abs*{\wh{g\normalr|_{h}}(\xip)}^2 \reg*{\xip}^{2t-1} \d \xip \d \theta \\
		&\lesssim \int_{\CP_\theta P^j_r(\vec e_1)} \frac{2^j}{\abs{\sin \theta}} \frac{1}{\reg{2^j\sin\theta}^{2t-1}} \norm{g\normalr|_{h}}_{H^{t-\frac 12}(\bbR^{d-1})}^2 \d \theta\\
		&\le 2^{-j(2t-2)} \int_{\CP_\theta P^j_r(\vec e_1)} \abs{\sin\theta}^{-2t} \d \theta \,\norm{g}_{H^{t}(\bbR^{d})}^2. \label{eq:branch_r1}
	\end{align}
	For $1<r<j$, we continue by transforming with $y=\sin\theta$ (the functional determinant $(\cos \theta)^{-1}$ is bounded from above by a constant due to the restriction that $r<j$)
	\begin{align}
		\int_{\CP_\theta P^j_r} \abs{\sin\theta}^{-2t} \d \theta \lesssim \bracket*{-y^{-2t+1}}_{2^{-r}-2^{-j}}^{2^{-r+1}+2^{-j}} = 2^{(-2t+1)(-r-1)}-2^{(-2t+1)(-r+2)}\lesssim 2^{r(2t-1)},
	\end{align}
	because $r>1$. For $r=1$, we have $\abs{\sin\theta}\gtrsim1$ and thus, the integral in \eqref{eq:branch_r1} can be estimated by a constant, and is trivially less than (a constant independent of $j$ times) $2^{r(2t-1)}$.
	
	Finally, for $r=j$, the support $P^j_j$ is contained in a cylinder aligned with the $\xi_1$-axis (cf. \cite[Prop. A.6]{compress}), and therefore
	\begin{align}
		\int_{P^j_j} \abs*{\wh{g\normalr|_{h}} (\CP_{\vec e_1} \vec \xi)}^2 \d \vec \xi \le \int_{[-2^{j+1},2^{j+1}]\times B_{\bbR^{d-1}}(0,8)} \abs*{\wh{g\normalr|_{h}} (\CP_{\vec e_1} \vec \xi)}^2 \d \vec \xi \lesssim 2^j \norm*{\wh{g\normalr|_{h}}}_{L^2(\bbR^{d-1})}^2 \le 2^j \norm{g}_{H^{t}(\bbR^{d})}^2.
	\end{align}
	Returning to \eqref{eq:gamma_est_ind_hyp} and collecting all the factors, we see that the latter two terms (combined into $\alpha^\rms_\lambda$ resp.~$\beta^\rms_\lambda$) satisfy \eqref{eq:gamma_est_sum_Lambda}, as claimed.
\end{proof}

\subsection{Localisation in Space}\label{sec:loc_space}

Similarly to the localisation in angle (i.e.~$\ell$), we require a localisation in space (i.e.~$\vec k$) for a function cut off across a hyperplane, which we will establish in this section. Before we are able to tackle the proof, we need to investigate the spatial decay of the ridgelets (and a modified variant) in physical space.

Finally, after having dealt with the singular functions in \autoref{prop:loc_space}, we investigate how spatial decay of a general $L^2$-function transfers to the ridgelet coefficients in \autoref{lem:ridge_coeff_tail_decay}.

\begin{lemma}\label{lem:decay_ridgelet}
	For arbitrary $n\in\bbN$ and an arbitrary rotation $R\in \mathrm{SO}(d)$, we have that,
	\begin{align}
		\abs{\varphi_\lambda(\vec x)} &\lesssim \frac{2^{\frac j2}}{\reg*{U_\jl^{-1}\vec x-\vec k}^{2m}}, \label{eq:ridge_est} \\
		\abs[\Big]{\CF^{-1}\bracket[\Big]{\frac{\xi_1}{\abs*{\vec \xi}^2}\hat\varphi_\lambda (R \vec \xi)}(\vec x)} &\lesssim \frac{2^{-\frac j2}}{\reg*{U_\jl^{-1}R\vec x-\vec k}^{2n}},\label{eq:mod_ridge_est}
	\end{align}
	where the implicit constants depends only on $n$ (and the choice of window function used to construct $\hat\psi$).
\end{lemma}

\begin{proof}
	We start by proving \eqref{eq:mod_ridge_est}, inserting the definition and transforming by $\vec \xi = R^{-1}U_\jl^{-\top}\vec\eta$, which yields (because, for rotations, $R^{-1}=R^{\top}$)
	\begin{align}
		\CF^{-1}\bracket[\Big]{\frac{\xi_1}{\abs*{\vec \xi}^2}\hat\varphi_\lambda (R\vec \xi)}(\vec x)
		&= 2^{-\frac{j}{2}}\int \frac{\xi_1}{\abs*{\vec \xi}^2} \hat \psi_\jl(R\vec \xi) \exp\parens*{2\pi\ii (\vec \xi\cdot \vec x-R\vec \xi \cdot U_\jl \vec k)} \d \vec \xi\\
		&= 2^{-\frac{j}{2}}\int \frac{\xi_1}{\abs*{\vec \xi}^2} \hat \psi_\jl(R\vec \xi) \exp\parens*{2\pi\ii R\vec \xi  \cdot U_\jl( U_\jl^{-1}R\vec x- \vec k)} \d \vec \xi\\
		&=2^{\frac{j}{2}}\int \frac{\parens*{R^{-1}U_\jl^{-\top} \vec \eta}_1}{\abs*{U_\jl^{-\top} \vec \eta}^2} \hat \psi_{(\jl)}(\vec \eta) \exp\parens*{2\pi\ii\vec \eta \cdot (U_\jl^{-1}R\vec x-\vec k)} \d \vec \eta. \label{eq:mod_ridge_est_trafo}
	\end{align}
	due to the representation from \autoref{assump:psi_smooth}. The idea now is to use integration by parts and
	\begin{align}\label{eq:laplace_exp}
		\Delta_{\vec \eta} \exp\parens*{2\pi\ii \vec \eta \cdot (U_\jl^{-1}R\vec x-\vec k)} = - (2\pi)^2 \abs*{U_\jl^{-1}R\vec x-\vec k}^2 \exp\parens*{2\pi\ii \vec \eta \cdot  (U_\jl^{-1}R\vec x-\vec k)}.
	\end{align}
	to generate sufficiently high powers of $\abs*{U_\jl^{-1}R\vec x-\vec k}$ in the denominator.
	
	The support $U_\jl^{\top}P_\jl \subseteq \bracket*{\frac 14,2}\times B_{\bbR^{d-1}}(0,4)$ is bounded from above and below (compare again \cite[Prop. A.6]{compress}) and so we just have to bound the derivatives independently of $j$. Due to \autoref{assump:psi_smooth}, this is not a problem for the $\psi_{(\jl)}(\vec \eta)$. If we can show the same for the fraction in front, we will have the desired estimate due to the product rule.
	\begin{align}
		d_\jl(\vec \eta):=\frac{\parens*{R^{-1}U_\jl^{-\top} \vec \eta}_1}{\abs*{U_\jl^{-\top} \vec \eta}^2} =\frac{\parens*{R^{-1}\Rjl^{-1} D_{2^{j}} \vec \eta}_1}{\abs*{D_{2^{j}} \vec \eta}^2} = \frac{2^jc_1\eta_1+c_2\eta_2+\ldots +c_d\eta_d}{2^{2j}\eta_1^2+\eta_2^2+\ldots +\eta_d^2}.
	\end{align}
	Due to the fact that the factor $2^j$ only appears in connection with $\eta_1$, we only need to consider what happens when deriving by $\eta_1$, as all other derivatives just produce higher powers of the denominator without adding powers of $2^j$ in the numerator. By splitting the term into $\eta_1$ and the rest, respectively, and adding
	\begin{align}
		0=-\ii c_1\sqrt{\eta_2^2+\ldots+\eta_d^2}+\ii c_1\sqrt{\eta_2^2+\ldots+\eta_d^2}
	\end{align}
	in the numerator of the first term, we see that (since the quantity we're interested is clearly real),
	\begin{align}
		\Dpn{k}{\eta_1} d_\jl = \Re\parens[\bigg]{\Dpn{k}{\eta_1}\frac{c_1}{2^j\eta_1+\ii\sqrt{\eta_2^2+\ldots+\eta_d^2}}} + \Dpn{k}{\eta_1}\frac{c_2\eta_2+\ldots +c_d\eta_d}{2^{2j}\eta_1^2+\eta_2^2+\ldots +\eta_d^2}.
	\end{align}
	The first term resolves to
	\begin{align}
		\Re\parens[\Bigg]{\frac{c_1 (-1)^k k! 2^{jk}}{\parens[\Big]{2^j\eta_1+\ii\sqrt{\eta_2^2+\ldots+\eta_d^2}}^{k+1}}},
	\end{align}
	where, clearly, the power of $2^j$ is always at least one higher in the denominator than in the numerator. For the second term, we anticipate Fa\`{a} di Bruno's formula \eqref{app:eq:faa_di_bruno} --- the generalisation of the chain rule to arbitrary order of differentiation --- and the generating function machinery of \autoref{ssec:genfunc} to deal with the Bell polynomials that are involved, i.e.~\eqref{app:eq:genfunc_bell}. Then, with ``inner'' function $g(x):=2^{2j}x^2$,
	\begin{align}
		B_{n,k}(2^{2j+1}x,2^{2j+1},0,\ldots)= 2^{2jk}\frac{n!}{k!} \parens[\Big]{\Dn{n}{t} t^k(2x+t)^k}\Bigr|_{t=0}, \quad \text{for} \quad \ceil{\frac n2}\le k\le n,
	\end{align}
	and thus
	\begin{align}
		\Dpn{k}{\eta_1} \frac{c_2\eta_2+\ldots +c_d\eta_d}{2^{2j}\eta_1^2+\eta_2^2+\ldots +\eta_d^2} = \sum_{k=\ceil{\frac n2}}^n \frac{2^{2jk} p_k(\eta_1) (c_2\eta_2+\ldots +c_d\eta_d)}{(2^{2j}\eta_1^2+\eta_2^2+\ldots +\eta_d^2)^{k+1}},
	\end{align}
	where $p_k(\eta_1)=(-1)^k n!\parens{\Dn{n}{t} t^k(2\eta_1+t)^k}\bigr|_{t=0}$ is a polynomial in $\eta_1$ independent of $2^j$. In particular, in terms of powers of $2^j$, the power in the denominator is always at least two larger than in the numerator.
	
	Taken together, we can see that because  $\abs{\vec \eta}\sim 1$ for $\vec \eta \in U_\jl^{\top}P_\jl$, for any multi-index $\alpha$ with $\abs{\alpha}\ge 0$,
	\begin{align}\label{eq:mod_ridge_djl}
		\Dpi{\alpha}[d_\jl]{\vec\eta}(\vec \eta) \lesssim 2^{-j}.
	\end{align}
	
	This allows us to apply \eqref{eq:laplace_exp} to \eqref{eq:mod_ridge_est_trafo} --- boundary terms vanish because of compact support. Consequently, using \cite[Cor. C.3]{compress}, we achieve
	\begin{align}
	\MoveEqLeft
		\abs[\Big]{\CF^{-1}\bracket[\Big]{\frac{\xi_1}{\abs*{\vec \xi}^2}\hat\varphi_\lambda (R\vec \xi)}(\vec x)}\\
		&= \abs[\bigg]{2^{\frac j2} \frac{(-4\pi^2)^n}{\abs*{U_\jl^{-1}R\vec x-\vec k}^{2n}} \int_{U_\jl^{\top}P_\jl} \Delta^n\parens{d_\jl(\vec\eta) \hat \psi_{(\jl)}(\vec \eta)} \exp\parens*{2\pi\ii\vec \eta \cdot (U_\jl^{-1}R\vec x-\vec k)} \d \vec \eta}\\
		&\lesssim \frac{2^{\frac j2}}{\abs*{U_\jl^{-1}R\vec x-\vec k}^{2n}} \int_{U_\jl^{\top}P_\jl} \abs*{d_\jl(\vec \eta)}_{\CC^{2n}} \, \abs*{\hat \psi_{(\jl)}(\vec \eta)}_{\CC^{2n}} \d \vec \eta \lesssim  \frac{2^{-\frac j2}}{\abs*{U_\jl^{-1}R\vec x-\vec k}^{2n}}.
	\end{align}
	Finally, this implies
	\begin{align}
		\abs[\Big]{\CF^{-1}\bracket[\Big]{\frac{\xi_1}{\abs*{\vec \xi}^2}\hat\varphi_\lambda (R\vec \xi)}(\vec x)}
		&\lesssim  \frac{2^{-\frac j2}}{\reg*{U_\jl^{-1}R\vec x-\vec k}^{2n}},
	\end{align}
	because the compact support of $\hat\varphi_\lambda$ implies that the function on the left-hand side belongs to $\CC^\infty(\bbR^d)$ --- i.e.~has no singularities --- with respect to $\vec x$ (and is bounded independently of $\vec k$) and therefore, we may replace the absolute value with its regularised version (up to a constant).
	
	The proof of \eqref{eq:ridge_est} proceeds along the same lines, with the simplification that we do not have to consider $d_\jl$ or its derivatives --- which also removes the factor $2^{-j}$ that we had gained from \eqref{eq:mod_ridge_djl}. This finishes the proof.
\end{proof}

With these tools in hand, we proceed to the promised localisation in space for a function cut off across a hyper plane.

\begin{proposition}\label{prop:loc_space}
	Let $f(\vec x)=H(\vec x \cdot \vec n-v)g(\vec x)$ with $g\in H^{\frac 12}(\bbR^d)$, cut off across the hyperplane $h=\set{\vec x}{\vec x\cdot \vec n=v}$, such that the restriction $g\normalr|_{h}\in L^2(\bbR^{d-1})$ satisfies
	\begin{align}
		\abs*{g(\vec x)\delta_{\{\vec x\cdot \vec n=v\}}}=\abs*{g\normalr|_{h}(\xp_h)} &\lesssim \reg{\xp_h}^{-2m} = \reg{\CP_{\vec n}\vec x}^{-2m} \quad \text{almost everywhere}, \label{eq:decay_sing_f}
		\intertext{
	where $\xp_h\in\bbR^{d-1}$ is identified with $\CP_{\vec n}\vec x=\vec x - (\vec x\cdot \vec n)\vec n$.
	For the solution $u=\CS[f]$, the mutilated part (cf. \autoref{prop:sol_smooth_except_hyp}) is of the same form, i.e.~$u(\vec x)=u_0(\vec x)+H(\vec x \cdot \vec n-v)u_1(\vec x)$, and we require
	}
		\abs*{u_1(\vec x)\delta_{\{\vec x\cdot \vec n=v\}}}=\abs*{u_1\normalr|_{h}(\xp_h)} &\lesssim \reg{\xp_h}^{-2m} = \reg{\CP_{\vec n}\vec x}^{-2m}\quad \text{almost everywhere}.\label{eq:decay_sing_u}
	\end{align}
	Condition \eqref{eq:decay_sing_f} does not necessarily imply \eqref{eq:decay_sing_u}, but the latter follows for example from the stronger assumption that $\abs{g(\vec x)} \lesssim \reg{\vec x}^{-2m}$, compare \autoref{lem:decay_sol}.
	
	If the ridgelets are constructed with a window function such that \eqref{eq:mod_ridge_est} holds for $m$ sufficiently large, we have the following \emph{localisation in space} in terms of a modified translation parameter $\vec t(\vec k)$,
	\begin{align}\label{eq:gamma_est_w}
		\abs*{\alpha^\rms_\lambda}&\lesssim \frac{1}{a} \frac{2^{-\frac j2}}{(\abs[]{\tpp} + \rho_1^2+\rho_2^2+1)^m}, &
		\abs*{w_\lambda\beta^\rms_\lambda}&\lesssim \frac{1}{a} \frac{2^{\frac j2}}{(\abs[]{\tpp} + \rho_1^2+\rho_2^2+1)^m},
		\intertext{
	where, with $\theta_{\vec n}:=\arccos(\vec n\cdot \vec e_1)$ and $\phi:=\theta_\jl-\theta_{\vec n}$,
		}
		\vec t(\vec k)&=\vec k-vU_\jl^{-1}\vec n, & \rho_1(\vec t)&:= \lcopywidth{\frac 1a}{\frac{1}{a^2}} (\cos \phi\, t_1-2^j\sin\phi\, t_2),\\
		a&:=\sqrt{2^{2j} \sin^2 \phi + \cos^2\phi}, &
		\rho_2(\vec t)&:= \frac{1}{a^2}\parens*{2^j\sin\phi\, t_1 + \cos\phi\, t_2}.
	\end{align}
	The essence \eqref{eq:gamma_est_w} is, in slightly obscured form, that the magnitude of the coefficients of the modified translation parameter $\vec t$ appear unchanged --- for $d-1$ dimensions --- in the decay of the ridgelet coefficients, while in one dimension (basically, along $t_1$), the magnitude is attenuated by a factor $\sim 2^{-j}$, as soon as $\sin\phi$ is not negligible. This behaviour corresponds directly to the way the grid of translations in \autoref{def:phi_jlk} is refined in one direction (resp.~the shape of the ridgelets themselves).
\end{proposition}

\begin{proof}
	Since $g$ and $u_1$ satisfy exactly the same condition (i.e.~\eqref{eq:decay_sing_f} and \eqref{eq:decay_sing_u}), we can deal with both cases at the same time (here exemplarily for $\alpha^\rms_\lambda$).
	Transforming with $\vec\zeta=\Rn\vec \xi$, using \eqref{eq:sing_fourier_rot} and applying the Plancherel identity, this results in (because $\zeta_1$ is real)
	\begin{align}
		\alpha^\rms_\lambda
		&=\frac{-\ii}{2\pi}\int \overline{\hat \varphi_\lambda (\vec \xi)} \frac{(\Rn\vec \xi)_1}{\abs*{\vec\xi}^2}  \wh{g\normalr|_{h}}(\CP_{\vec n}\vec \xi) \d \vec \xi
		=\frac{ -\ii}{2\pi}\int \overline{\hat \varphi_\lambda (\Rn^{-1}\vec \zeta) \frac{\zeta_1}{\abs*{\vec \zeta}^2}} \CF\bracket*{g(\Rn^{-1}\vec x)\delta_{\{\vec x\cdot \vec e_1=v\}}}(\vec \zeta) \d \vec \zeta \\
		&= \frac{-\ii}{2\pi} \int \overline{\CF^{-1}\bracket[\Big]{\frac{\xi_1}{\abs*{\vec \xi}^2}\hat\varphi_\lambda (\Rn^{-1}\vec \xi)}\!(\vec x)} \, g(\Rn^{-1}\vec x)\delta_{\{x_1=v\}} \d \vec x\\
		&= \frac{-\ii}{2\pi} \int_{[\vec e_1]^\bot} \overline{\CF^{-1}\bracket[\Big]{\frac{\xi_1}{\abs*{\vec \xi}^2}\hat\varphi_\lambda (\Rn^{-1}\vec \xi)}\!\binom{v}{\xp}} \, g\parens{\Rn^{-1}\binom{v}{\xp}} \d \xp,
	\end{align}
	where $[\vec e_1]$ denotes the span of $\vec e_1$, and $[\vec e_1]^\bot$ its orthogonal complement. Furthermore, it is easy to check that $\Rn^{-1}\binom{v}{\xp}= \Rn^{-1}\binom{0}{\xp}+v\vec n\in h$, as well as $\CP_{\vec n} \Rn^{-1}\binom{v}{\xp} = \Rn^{-1}\binom{0}{\xp}$. Therefore, using \eqref{eq:mod_ridge_est} and \eqref{eq:decay_sing_f},
	\begin{align}
		\abs{\alpha^\rms_\lambda}
		&\lesssim 2^{-\frac j2}\int_{[\vec e_1]^\bot} \frac{1}{\reg*{U_\jl^{-1}\Rn^{-1}\binom{v}{\xp}-\vec k}^{2n}} \frac{1}{\reg*{\Rn^{-1} \binom{0}{\xp}}^{2m}} \d \xp\\
		&= 2^{-\frac j2}\int_{[\vec e_1]^\bot} \frac{1}{\reg*{U_\jl^{-1}\Rn^{-1}\binom{0}{\xp}-\vec t}^{2n}} \frac{1}{\reg*{\binom{0}{\xp}}^{2m}} \d \xp,
	\end{align}
	where we have split off the component involving $v$ in the first denominator, i.e.
	\begin{align}
		U_\jl^{-1}\binom{v}{\xp}-\vec k=U_\jl^{-1}\binom{0}{\xp}-\vec t, \quad \text{using} \quad \vec t:=\vec k-vU_\jl^{-1}\Rn^{-1}\vec e_1=\vec k-vU_\jl^{-1}\vec n.
	\end{align}
	Since $U_\jl^{-1}\Rn^{-1}[\vec e_1]^\bot$ is still a hyperplane, the first denominator attains its minimum when the left term corresponds to the orthogonal projection of $\vec t$ onto that plane. The normal vector transforms with the transpose of the inverse of the transformation, i.e.
	\begin{align}
		U_\jl^{-1}\Rn^{-1}[\vec e_1]^\bot \bot (U_\jl^{-1}\Rn^{-1})^{-\top}\vec e_1 = U_\jl^\top \vec n,
		\quad \text{and we set} \quad
		\sst:= \frac{U_\jl^\top \vec n}{\abs*{U_\jl^\top \vec n}}.
	\end{align}
	Thus, the minimiser of the first denominator is  $\tau(\vec t):= \Rn U_\jl(\vec t-\sst(\sst\cdot \vec t)) \in [\vec e_1]^\bot$. By the transformation $\binom{0}{\yp}=\binom{0}{\xp}+\tau(\vec t)$, we have
	\begin{align}
		\abs{\alpha^\rms_\lambda}&\lesssim 2^{-\frac j2}\int_{[\vec e_1]^\bot} \frac{1}{\reg*{U_\jl^{-1}\Rn^{-1}\binom{0}{\yp} - \sst (\sst\cdot \vec t)}^{2n}} \frac{1}{\reg*{\binom{0}{\yp} - \tau(\vec t)}^{2m}} \d \yp\\
		&= 2^{-\frac j2}\int_{[\vec e_1]^\bot} \frac{1}{\parens*{\abs*{U_\jl^{-1}\Rn^{-1}\binom{0}{\yp}}^2 + (\sst\cdot \vec t)^2+1}^{n}} \frac{1}{\reg*{\binom{0}{\yp} - \tau(\vec t)}^{2m}} \d \yp, \label{eq:gamma_est_y_trafo}
	\end{align}
	where the two summands in the first denominator are now orthogonal to each other, and the last equality follow by Pythagoras' theorem.
	
	We will now try to construct a rotation, which will transform the terms in the integral in a way that it can be dealt with. For reasons of clarity, this will be a two-stop process, first constructing the main rotation $Q$, and then a small correction $T$. The condition on the rotation (that $\Rjl \sjl=\vec e_1$) fixes the first component of $\Rjl\vec e_1$ to be $\cos \theta_\jl$. Otherwise, we have the freedom to choose $\Rjl$ such that
	\begin{align}
		\Rjl \vec e_1=\begin{pmatrix}
		\phantom{-}\cos \theta_\jl \\ -\sin\theta_\jl\\0\\ \vdots
		\end{pmatrix},
	\qquad \text{and in the same way,} \qquad
		\Rn \vec e_1=\begin{pmatrix}
		\phantom{-}\cos \theta_{\vec n} \\ -\sin\theta_{\vec n}\\0\\ \vdots
		\end{pmatrix}.
	\end{align}
	We may choose further images for our rotations, as long as we maintain the property that $R\vec s\cdot R\spr=\vec s\cdot\spr$ for any pair of images $R\vec s,R\spr$ we choose for vectors $\vec s, \spr\in\bbSd$, because rotations maintain the intrinsic (geodesic) distance between points on the sphere. Defining $\sjl':=P_{\vec e_1}\sjl$ (with norm $\abs*{\sin\theta_\jl}$), we may therefore choose
	\begin{align}
		\Rn \frac{\sjl'}{\sin\theta_\jl}=\begin{pmatrix}
		\sin \theta_{\vec n} \\ \cos\theta_{\vec n}\\0\\ \vdots
		\end{pmatrix},
	\end{align}
	since, by our choice for $\Rn \vec e_1$ above, $\Rn \vec e_1\cdot \Rn \frac{\sjl'}{\sin\theta_\jl} = 0 =  \vec e_1 \cdot \frac{\sjl'}{\sin\theta_\jl}$. In case that $\sjl=\vec e_1$ (and thus $\sin\theta_\jl=0$), we take $\Rn e_2= (\sin \theta_{\vec n}, \cos\theta_{\vec n}, 0,\ldots)^\top$, which is a permissible choice for the same reason.
	
	Once more in the same fashion, we are able to determine a rotation $Q$ in the following way
	\begin{align}
	Q^{-1} \vec e_1=\Rn \vec e_1 \qquad \text{and} \qquad Q^{-1}\vec e_2 = \Rn\frac{\sjl'}{\sin\theta_\jl},
	\end{align}
	whereby we achieve
	\begin{align}
		Q^{-1} &= \left(\begin{array}{cc|c}
			\phantom{-}\cos \theta_{\vec n} & \sin \theta_{\vec n} & 0 \\
			-\sin\theta_{\vec n} & \cos\theta_{\vec n} & 0 \\ \hline
			0 & 0 &  (Q'')^{-1}
		\end{array}\right),\\
		\Rjl \Rn^{-1}Q^{-1} &= \left(\begin{array}{cc|c}
			\phantom{-}\cos \theta_\jl & \sin \theta_\jl & 0 \\
			-\sin\theta_\jl & \cos\theta_\jl & 0  \\ \hline
			0 & 0 &  (V'')^{-1}
		\end{array}\right). \label{eq:RjlQ}
	\end{align}
	The form of the second matrix can be seen since, due to the choice of $Q^{-1}\vec e_1$, the first column is just $\Rjl \vec e_1$, while the second column is calculated as follows. The first entry evaluates to
	\begin{align}
		\vec e_1 \cdot \Rjl \Rn^{-1} Q^{-1} \vec e_2
		&= \Rjl^{-1} \vec e_1 \cdot \Rn^{-1} \Rn \frac{\sjl'}{\sin\theta_\jl} = \sjl \cdot \frac{\CP_{\vec e_1} \sjl}{\sin\theta_\jl}  = \frac{\sum_{k=2}^d(\sjl)_k^2}{\sin\theta_\jl} = \frac{1-(\sjl)_1^2}{\sin\theta_\jl} \\
		&=\frac{1-\cos^2 \theta_\jl}{\sin\theta_\jl} = \sin \theta_\jl,
	\end{align}
	since $(\sjl)_1=\sjl \cdot \vec e_1= \cos \theta_\jl$. As the norm of the first row is thus one already, the other entries must be zero. Similarly,
	\begin{align}
		\vec e_2 \cdot \Rjl \Rn^{-1}Q^{-1} \vec e_2 = \vec e_2 \cdot  \Rjl \parens{\!\sjl - \begin{pmatrix}
		(\sjl)_1\\0\\ \vdots
		\end{pmatrix}\!} \frac{1}{\sin\theta_\jl} = \frac{\vec e_2}{\sin\theta_\jl} \!\cdot\! (\vec e_1 - \cos\theta_\jl \Rjl \vec e_1) = \cos\theta_\jl
	\end{align}
	implies the zeros in the second row and column. The matrices $Q''$ resp.~$V''$ (restricted to have determinant one) represent the degrees of freedom we still have in choosing the rotation $Q$.
	
	Finally, we define the two matrices
	\begin{align}
		S_{\theta} = \begin{pmatrix}
			\cos \theta &- \sin\theta \\ \sin\theta & \phantom{-}\cos \theta
		\end{pmatrix}, \qquad T:= \begin{pmatrix}
			S_{\theta_{\vec n}} & 0 \\ 0 & \bbI_{d-2}
		\end{pmatrix},
	\end{align}
	where $S_{\theta}$ is the standard rotation matrix in two dimensions satisfying $S_{\theta}S_{\phi}=S_{\theta+\phi}$. Inverting \eqref{eq:RjlQ}, we see that
	\begin{align}
		Q&= \begin{pmatrix}
			S_{\theta_{\vec n}} & 0 \\ 0 & Q''
		\end{pmatrix},
		& Q\Rn\Rjl^{-1} &= \begin{pmatrix}
			S_{\theta_\jl} & 0 \\ 0 & V''
		\end{pmatrix},
		& T^{-1}Q &= \begin{pmatrix}
			\bbI_{2} & 0 \\ 0 & Q''
		\end{pmatrix},
	\end{align}
	and consequently, setting $\phi:=\theta_\jl-\theta_{\vec n}$,
	\begin{align}\label{eq:RjlQ_inv}
		T^{-1}Q\Rn\Rjl^{-1} &= \begin{pmatrix}
			S_{\phi} & 0 \\ 0 & V''
		\end{pmatrix},
		&\Rjl \Rn^{-1} &=\Rjl \Rn^{-1}Q^{-1}Q
		=\begin{pmatrix}
			S_{-\phi} & 0 \\ 0 & (V'')^{-1}Q''
		\end{pmatrix}.
	\end{align}
	
	Coming back to \eqref{eq:gamma_est_y_trafo}, for reasons that will become apparent shortly, we transform with $\binom{0}{\zp} = T^{-1}Q\binom{0}{\yp}$ (since $T^{-1}Q$ leaves the first component invariant!) to arrive at
	\begin{align}
		\abs{\alpha^\rms_\lambda}
		&\!\begin{multlined}[t][\textwidth-\mathindent-\widthof{$\abs{\alpha^\rms_\lambda}$}-\multlinegap]
			\lesssim 2^{-\frac j2}\int_{[\vec e_1]^\bot} \! \parens{\abs*{D_{2^j}\Rjl \Rn^{-1}Q^{-1}T\zp}^2 + (\sst\cdot \vec t)^2+1}^{-n} \cdot \ldots \\
			\ldots \cdot \reg*{Q^{-1}T\parens*{(0,\zp)^\top - T^{-1}Q\tau(\vec t)}}^{-2m} \d \zp
		\end{multlined}\\
		&\!\begin{multlined}[t][\textwidth-\mathindent-\widthof{$\abs{\alpha^\rms_\lambda}$}-\multlinegap]
			=2^{-\frac j2}\int_{[\vec e_1]^\bot} \! \parens{\abs*{\parens*{-2^j \sin \phi\, z_2, \cos \phi\, z_2, (V'')^{-1}\zpp}^\top}^2 + (\sst\cdot \vec t)^2+1}^{-n} \cdot \ldots \label{eq:gamma_est_z_trafo} \\
			\ldots \cdot \reg*{(0,\zp)^\top - T^{-1}Q\tau(\vec t)}^{-2m} \d \zp,
		\end{multlined}
	\end{align}
	where we used the (inverse of the) first equality from \eqref{eq:RjlQ_inv} to evaluate the first denominator. The next step is to calculate $T^{-1}Q\tau(\vec t)$, which we tackle as follows, this time with the second equality in \eqref{eq:RjlQ_inv},
	\begin{align}
		U_\jl^\top \vec n &= D_{2^{-j}}\Rjl\Rn^{-1}\vec e_1 = \parens*{2^{-j}\cos \phi,-\sin \phi,0,\ldots}^\top,\\
		\abs*{U_\jl^\top \vec n}&=2^{-j}\sqrt{2^{2j} \sin^2\phi +\cos\phi}=:2^{-j} a.
	\end{align}
	From there on, we compute
	\begin{align}
		\sst = \frac{1}{a} \parens*{\cos \phi,-2^j\sin \phi,0,\ldots}^\top, \qquad \sst \cdot \vec t =\frac{1}{a}(\cos\phi \,t_1 - 2^j\sin\phi \,t_2) = \rho_1(\vec t),
	\end{align}
	and therefore, again with \eqref{eq:RjlQ_inv},
	\begin{align}
		T^{-1}Q\tau(\vec t)
		&=T^{-1}Q\Rn \Rjl^{-1}D_{2^{-j}}\parens*{\vec t - (\sst \cdot \vec t)\sst}\\
		&=T^{-1}Q\Rn \Rjl^{-1} \parens{
		\begin{pmatrix}
			2^{-j}t_1\\t_2\\ \tpp
		\end{pmatrix}
		-\frac{\rho_1(\vec t)}{a}
		\begin{pmatrix}
			2^{-j}\cos\phi\\-2^j\sin\phi\\ 0
		\end{pmatrix}} \\
		&= \begin{pmatrix}
			2^{-j}\cos \phi \, t_1 - \sin \phi \, t_2\\2^{-j}\sin\phi \, t_1 + \cos\phi \, t_2\\ V''\tpp
		\end{pmatrix}
		-\frac{\rho_1(\vec t)}{a}
		\begin{pmatrix}
			2^{-j}\cos^2\phi+2^j\sin^2\phi\\2^{-j}\sin\phi\cos\phi -2^j\sin\phi\cos\phi\\ 0
		\end{pmatrix} \\
		&= \begin{pmatrix}
			0 \\ \frac{1}{a^2}\parens*{2^{j} \sin \phi \, t_1 + \cos \phi \, t_2} \\ V''\tpp
		\end{pmatrix}
		= \begin{pmatrix}
			0 \\ \rho_2(\vec t) \\ V''\tpp
		\end{pmatrix},
	\end{align}
	where the first two components simplify substantially, after inserting $\rho_1(\vec t)$, and expanding the first term with $\frac{a^2}{a^2}$. We also note that $\rho_1^2+a^2\rho_2^2=t_1^2+t_2^2$.
	
	Continuing from \eqref{eq:gamma_est_z_trafo}, we can now bring this in the following form (for convenience, we will not further denote the dependence of $\rho_1,\,\rho_2$ on $\vec t$)
	\begin{align}
		\abs{\alpha^\rms_\lambda}
		&\lesssim 2^{-\frac j2} \int_{-\infty}^{\infty} \int_{[\vec e_1,\vec e_2]^\bot} \parens*{|\zpp|^2 + {\underbrace{a^2z_2^2+\rho_1^2+1}_{=:p^2}}}^{-n} \parens*{\abs[]{\zpp-V''\tpp}^2+{\underbrace{(z_2 - \rho_2)^2+1}_{=:q^2}}}^{-m}  \d \zpp \d z_2.
	\end{align}
	Using \autoref{cor:Imn_higher_dim} for $\zpp$ (basically applying \autoref{th:Imn_est} in direction $V''\tpp$, transforming to polar coordinates and then once more \autoref{th:Imn_est} with respect to the radius), this yields
	\begin{align}
		\abs{\alpha^\rms_\lambda}
		&\lesssim 2^{-\frac j2} \int_{-\infty}^{\infty} p^{-2n} \parens{\abs[]{\tpp} + p^2+q^2}^{-m} \d z_2\\
		&\le 2^{-\frac j2} \int_{-\infty}^{\infty} \parens{a^2z_2^2+ \rho_1^2+1}^{-m} \parens{(z_2 - \rho_2)^2+\abs[]{\tpp} + \rho_1^2+1}^{-m}  \d z_2
		\intertext{
	and again with the help of \eqref{eq:Imn_est},
		}
		&\lesssim  \frac{2^{-\frac j2}}{(a^2\rho_2^2+a^2(\abs[]{\tpp} + \rho_1^2+1)+\rho_1^2+1)^m} \parens[\bigg]{\frac{a^{2m-1}}{(\rho_1^2+1)^{\frac{2m-1}2}} + \frac{1}{(\abs[]{\tpp} + \rho_1^2+1)^{\frac{2m-1}2}}}\\
		&\le \frac{2^{-\frac j2}}{(\abs[]{\tpp} + \rho_1^2+\rho_2^2+1)^m} \parens{\frac{1}{a} + \frac{1}{a^{2m}}} \le \frac{1}{a} \frac{2^{-\frac j2}}{(\abs[]{\tpp} + \rho_1^2+\rho_2^2+1)^m},
	\end{align}
	since $1\le a\le 2^{j}$, which is what we wanted to show for $\alpha^\rms_\lambda$.
	
	Repeating the whole procedure for $u_1$ to estimate $w_\lambda\beta_\lambda^\rms$ increases the right-hand side by $2^j$ at worst, due to the weight (it is not apparent if or how the additional smoothness in the direction of $\vec s$ can be utilised, since we are only integrating over a hyperplane), and the proof is complete.
\end{proof}

As the last result in this section, we consider how general decay of the function being tested against the ridgelet frame transfers to the coefficients.

\begin{lemma}\label{lem:ridge_coeff_tail_decay}
	For a function $f\in L^2$ satisfying the decay condition
	\begin{align}
		\abs{f(\vec x)} \lesssim \reg{\vec x}^{-2m}
	\end{align}
	for some $n\in\bbN$, the ridgelet coefficients satisfy
	\begin{align}\label{eq:ridge_coeff_tail_decay}
		\abs*{\inpr{\varphi_\lambda,f}}\lesssim \frac{2^{-\frac j2}}{\parens*{\parens*{\frac{k_1}{2^j}}^2+\abs[]{\kp}^2+1}^m}.
	\end{align}
\end{lemma}

\begin{proof}
	Using \eqref{eq:ridge_est} and the required decay of $f$, we estimate and then transform by $\vec y = \Rjl \vec x$, yielding
	\begin{align}
		\abs*{\inpr{\varphi_\lambda,f}}
		&=\abs[\bigg]{\int  \overline{\varphi_\lambda(\vec x)} f(\vec x) \d \vec x}
		\le \int  \frac{2^{\frac j2}}{\reg*{U_\jl^{-1}\vec x-\vec k}^{2n}} \frac{1}{\reg*{\vec x}^{2m}} \d \vec x
		= \int \frac{2^{\frac j2}}{\reg*{D_{2^j}\vec y-\vec k}^{2n}}  \frac{1}{\reg*{\vec y}^{2m}} \d \vec y\\
		&=2^{\frac j2} \int_{-\infty}^{\infty}\int_{\bbR^{d-1}} \frac{1}{\parens*{\abs[]{\yp-\kp}^2+(2^jy_1-k_1)^2+1}^{n}}  \frac{1}{\parens*{\abs[]{\yp}^2+(y_1)^2+1}^{m}} \d \yp \d y_1.
	\end{align}
	By \autoref{cor:Imn_higher_dim}, this can be estimated in the following way,
	\begin{align}
		\abs*{\inpr{\varphi_\lambda,f}}
		&\lesssim 2^{\frac j2} \int_{-\infty}^{\infty} \frac{1}{\parens*{\abs[]{\kp}^2+(2^jy_1-k_1)^2+y_1^2+1}^{m}}  \frac{1}{\parens*{(2^jy_1-k_1)^2+1}^{m}} \d y_1\\
		&\le 2^{\frac j2} \int_{-\infty}^{\infty} \frac{1}{\parens*{2^{2j}\parens*{y_1-\frac{k_1}{2^j}}^2+1}^{m}} \frac{1}{\parens*{y_1^2+\abs[]{\kp}^2+1}^{m}} \d y_1,
		\intertext{
	which, by \eqref{eq:Imn_est}, is less than a multiple of
		}
		&\lesssim  \frac{2^{\frac j2}}{\parens*{2^{2j}\parens*{\frac{k_1}{2^j}}^2+2^{2j}(\abs[]{\kp}^2+1)+1}^{m}} \parens[\bigg]{2^{j(2m-1)}  + \frac{1}{\parens*{\abs[]{\kp}^2+1}^{\frac{2m-1}{2}}}} \le \frac{2^{-\frac j2}}{\parens*{\parens*{\frac{k_1}{2^j}}^2+\abs[]{\kp}^2+1}^m}.
	\end{align}
\end{proof}

\subsection{Proof of \autoref{th:approx}}\label{ssec:proof_approx}

With the results of Sections \ref{sec:loc_angle} and \ref{sec:loc_space}, we are now in a position to prove \autoref{th:approx}.

\begin{proof}[Proof of \autoref{th:approx}]
	Recall that $\alpha_\lambda=\inpr*{\varphi_\lambda,f}$ and $\beta_\lambda=w_\lambda\inpr*{\varphi_\lambda,\CS[f]}$ from \autoref{prop:loc_angle}. Thus, the claim of the theorem will follow if we are able to prove
	\begin{align}
		\norm{\alpha_\Lambda}_{\ell^{p^* + \frac dm}_w} \lesssim \sum_{i=0}^N\norm{f_i}_{H^t} \qquad \text{and} \qquad
		\norm{\VW \beta_\Lambda}_{\ell^{p^* + \frac dm}_w} \lesssim \sum_{i=0}^N\norm{f_i}_{H^t},
	\end{align}
	where the deviation $\frac dn$ from the best possible value $p^*=\parens*{\frac td + \frac 12}^{-1}$ (for $f_i\in H^t$), decays with $n$.
	
	\step{Preparations}
	First off, we note that due to the linearity of the ridgelet coefficients, resp.~the differential equation \eqref{eq:LinTrans}, it suffices to treat a function $f(\vec x)=H(\vec x\cdot \vec n-v)g(\vec x)$ and the rest follows through superposition. Also, we will be able to prove the results for $\alpha_\Lambda$ and $\beta_\Lambda$ almost completely simultaneously, and will mostly deal with $\beta_\Lambda$ (noting, where appropriate, the mitigating factors in the easier $\alpha_\Lambda$-case). The only additional thing we have to check for $u=\CS[f]$ is that the decay conditions for \autoref{lem:ridge_coeff_tail_decay} holds, but this follows directly from \autoref{lem:decay_sol}.
	
	Furthermore, we recall from \autoref{prop:sol_smooth_except_hyp} that $u(\vec x)=u_0(\vec x)+H(\vec x\cdot \vec n-v)u_1(\vec x)$ with $\norm{u_i}_{H^t}\lesssim \norm{g}_{H^t}$. Due to \eqref{eq:f_glob_decay}, we also have the necessary conditions to apply \autoref{prop:loc_space},
	\begin{align}
		\abs*{g\normalr|_{h}(\CP_{\vec n}\vec x)} &\lesssim \reg{\CP_{\vec n}\vec x}^{-2m}, \qquad \abs*{u_1\normalr|_{h}(\CP_{\vec n}\vec x)}\lesssim \reg{\CP_{\vec n}\vec x}^{-2m}.
	\end{align}
	
	We begin by choosing $\delta>0$, and will show membership of $\VW\beta_\Lambda$ in $\ell^{p'}_w$, where $p'=p^* + \delta^{*}$ and $\delta^{*}=\delta\parens{t+\frac d2}^{-1}$. The modified $\delta^{*}$ will have to satisfy a lower bound to ensure that all the arguments hold, and we will then that $\frac dm$ is enough to achieve this bound.
	Following the decomposition $\beta_\lambda=\beta^\rmr_\lambda+\beta^\rms_\lambda$, we will show this for both subsequences separately.
	
	\step{Singular part, large coefficients}
	We begin with $\VW\beta^\rms_\Lambda$, splitting the sequence into two parts, namely the \emph{tail}
	\begin{align}\label{eq:approx_main_sing_tail}
		T_\jl&:=\set[\Big]{\vec k\in\bbZ^d}{\abs*{V_\jl\vec t(\vec k)} > \parens*{2^{j+1}\abs*{\sin\theta_\jl}}^{\frac \delta d}+\frac{\sqrt{d}}{2}},
	\end{align}
	and its complement $T_\jl^\compl:=\bbZ^d\setminus T_\jl$. Here, $\vec t(\vec k)$ is defined as in \autoref{prop:loc_space}, and similarly, $a$ and $\phi$ are reused to define  $V_\jl$ as the following block-diagonal matrix (with determinant $\frac 1a$),
	\begin{align}
	V_\jl= \left(\begin{array}{cc|c}
		\frac{\cos\phi}{a} & \frac{2^j\sin\phi}{a} & 0 \\
		\frac{2^j\sin\phi}{a^2} & \frac{\cos\phi}{a^2} & 0 \\ \hline
		0 & 0 &\bbI_{d-2}
	\end{array}\right).
	\end{align}
	
	As the first step, we will show membership of the sequence $\VW\beta^\rms_\Lambda\normalr|_{\{\lambda:\vec k\not\in T_\jl\}}$ in $\ell^{p'}_w$. 
	Taking the number of elements whose absolute value exceeds $\eps$, i.e.
	\begin{align}
		N(\eps):=\card*{\lambda}{k \not\in T_\jl \land \abs{w_\lambda\beta^\rms_\lambda}\ge \eps},
	\end{align}
	an equivalent definition of the $\ell^{p'}_w$-norm is
	\begin{align}
		\norm*{\VW\beta^\rms_\Lambda\normalr|_{\{\lambda:\vec k\not\in T_\jl\}}}_{\ell^{p'}_w}\sim\sup_{\eps>0} \eps N^{\frac 1{p'}}(\eps) \lesssim \norm{g}_{H^t},
	\end{align}
	which is what we will show in the following.
	
	To do so, we define a further subset of $N(\eps)$,
	\begin{align}
		N_{j,r}(\eps):=\card*{(\ell,\vec k)}{\ell\in\Lambda^j_r \land k \not\in T_\jl \land \abs{w_\lambda\beta^\rms_\lambda}\ge \eps}.
	\end{align}
	Clearly, \eqref{eq:gamma_est_sum_Lambda} implies that $\abs{w_\lambda\beta^\rms_\lambda}^2 \le C 2^{-j} \norm{g}^2_{H^t}$, and therefore $N_{j,r}(\eps)=0$ if $2^j\ge C\eps^{-2}\norm{g}^2_{H^t}$.
	
	Next, we need to determine the cardinality of $T_\jl^\compl$ for $\ell \in \Lambda^j_r$. First we estimate the sum by an integral (where the maximum of each cell of size $[0,1]^d$ can be at most $\frac{\sqrt{d}}2$ away from the value of the function at $\vec k \in \bbZ^d$), and then transform with $\vec x=V_\jl\vec t(\vec k)=V_\jl (\vec k-vU_\jl^{-1}\vec n)$ --- introducing a factor $a$ from the determinant --- to estimate
	\begin{align}
		\# T_\jl^\compl
		&\le \int \ind_{\curly*{\abs[]{V_\jl \vec t(\vec k)}\le \parens{2^{j+1}\abs{\sin\theta_\jl}}^{\frac \delta d}+\sqrt{d}}} \d \vec k = \int a \,\ind_{\curly*{\abs[]{\vec x}\le \parens{2^{j+1}\abs{\sin\theta_\jl}}^{\frac \delta d}+\sqrt{d}}}\d \vec x\\
		&= a \,\mu\parens[\Big]{B_{\bbR^{d}}\parens*{0, \parens*{2^{j+1}\abs*{\sin\theta_\jl}}^{\frac \delta d}+\sqrt{d}}} \lesssim 2^{(j-r)(1+\delta)}. \label{eq:card_Sjl}
	\end{align}
	Here, $\mu$ is the $d$-dimensional Lebesgue measure.
	The restriction defining $T_\jl$ may now seem somewhat arbitrary, but the result is that the number of translations in the singular set is essentially one-dimensional (even though the set of translations is $d$-dimensional!), which gives rise to the following bound for the cardinality $N_{j,r}$ (regardless of the condition that the absolute value be larger than $\eps$),
	\begin{align}\label{eq:card_Njr}
		N_{j,r}(\eps)\lesssim \underbrace{2^{(j-r)(d-1)}}_{\mathllap{\#\Lambda^j_r}\lesssim} \cdot \underbrace{2^{(j-r)(1+\delta)}}_{\mathllap{\#T_\jl^\compl}\lesssim} = 2^{(j-r)(d+\delta)}.
	\end{align}
	As we shall see, the factor $d+\delta$ will directly influence the best $p$ possible for the $\ell^p_w$-norm of the subsequence involving $\vec k \in T_\jl$, in the sense that we will achieve
	\begin{align}\label{eq:est_p4sing}
		p=\frac{2(d+\delta)}{2t+(d+\delta)} \le \frac{2(d+\delta)}{2t+d} = \parens[\Big]{\frac td + \frac 12}^{-1} + \frac{2\delta}{2t+d} = p^*+\delta^{*} =p',
	\end{align}
	where we recall $\delta^{*}=\delta\parens*{t+\frac d2}^{-1}$.
	
	Of course, restricting the translations to achieve this bound is only half the battle --- everything we exclude now needs to be bounded afterwards --- but this shall work out in our favour (and make sense of the definition of $T_\jl$), because the tail is precisely defined to contain only the small coefficients.
	
	Let $\eta:=\eps/\norm{g}_{H^t}$.  Then, by \eqref{eq:gamma_est_sum_Lambda} and \eqref{eq:card_Njr},
	\begin{align}
		N_{j,r}(\eps)\lesssim \min\parens*{2^{(j-r)(d+\delta)},\eta^{-2} 2^{-j} 2^{-(j-r)(2t-1)}}.
	\end{align}
	Calculating the maximal $r$ such that the second term is the minimum, we find
	\begin{align}
		\eta^{-2} 2^{-j} 2^{-(j-r)(2t-1)} \le 2^{(j-r)(d+\delta)} \qquad \Longleftrightarrow \qquad r\le j- \log_2(\eta^{-\frac 2\sigma} 2^{-\frac{j}\sigma}),
	\end{align}
	where $\sigma:=d+\delta+2t-1$. Therefore,
	\begin{align}
		N_j(\eps)=\sum_{r=1}^j N_{j,r}(\eps) \lesssim \min\parens*{2^{j(d+\delta)},\eta^{-2} 2^{-j} (\eta^{-\frac 2\sigma} 2^{-\frac{j}\sigma})^{2t-1}}=\min\parens*{2^{j(d+\delta)},\eta^{-\frac{2(d+\delta)}\sigma} 2^{-\frac{j(d+\delta)}\sigma}}
	\end{align}
	Again, we determine the critical $j$ where the minimum switches from one term to the other, and see that
	\begin{align}
		\eta^{-\frac{2(d+\delta)}\sigma} 2^{-\frac{j(d+\delta)}\sigma} \le 2^{j(d+\delta)} \qquad \Longleftrightarrow \qquad \eta^{-\frac{2}{\sigma+1}} \le 2^j,
	\end{align}
	which implies
	\begin{align}
		N_j(\eps)\lesssim\begin{cases}
			2^{j(d+\delta)}, & 2^j\le  \eta^{-\frac{2}{\sigma+1}}, \\
			\eta^{-\frac{2(d+\delta)}\sigma} 2^{-\frac{j(d+\delta)}\sigma}, & \eta^{-\frac{2}{\sigma+1}} \le 2^j \le C\eta^{-2},\\ 
			0, & 2^j\ge C\eta^{-2}.
		\end{cases}
	\end{align}
	Finally, we have
	\begin{align}
		N(\eps)
		&=\sum_{j=0}^\infty N_j(\eps)\lesssim \sum_{j\colon 2^j \le \eta^{-2/(\sigma+1)}} 2^{j(d+\delta)} + \sum_{j\colon\eta^{-2/(\sigma+1)}\le 2^j} \eta^{-\frac{2(d+\delta)}\sigma} 2^{-\frac{j(d+\delta)}\sigma} \\
		&\lesssim \eta^{-\frac{2(d+\delta)}{\sigma+1}}+ \eta^{-\frac{2(d+\delta)}\sigma + \frac{d+\delta}\sigma \frac{2}{\sigma+1}} \lesssim \eta^{-\frac{2(d+\delta)}{2t+d+\delta}} \le \eta^{-p'},
	\end{align}
	by \eqref{eq:est_p4sing} since $\eta<1$, which finishes the argument for the first subsequence.
	
	\step{Singular part, tail}
	Next we will show membership of the sequence $\VW\beta^\rms_\Lambda\normalr|_{\{\lambda:\vec k\in T_\jl\}}$ in $\ell^{p'}\subseteq \ell^{p'}_w$. We begin by taking $q<1$, and --- after applying \eqref{eq:gamma_est_w} --- again estimate the sum by an integral (recalling that the maximum per cell can be at most $\frac{\sqrt{d}}2$ away from the value at $\vec k$, which cancels with the shift in the right-hand side of the defining inequality of $T_\jl$),
	\begin{align}
		\sum_{\vec k \in T_\jl} \abs{w_\lambda \beta^\rms_\lambda}^q
		&\lesssim \sum_{\vec k \in T_\jl} \frac{2^{\frac {jq}2} a^{-q}}{\parens*{|\tpp|^2+\rho_1^2+\rho_2^2+1}^{mq}} \\
		&\le \int \ind_{\curly*{\abs[]{V_\jl \vec t}> \sqrt[d/\delta]{2^{j+1}\abs{\sin\theta_\jl}}}} \frac{2^{\frac {jq}2} a^{-q}}{\parens*{|\tpp|^2+\rho_1^2+\rho_2^2+1}^{mq}} \d \vec k.
	\end{align}
	Calculating $V_\jl \parens*{t_1, t_2, t_3,\ldots}^\top=\parens*{\rho_1, \rho_2, t_3,\ldots}^\top$,
	we transform the first term by $\vec x=V_\jl^{-1} (\vec k-U_\jl^{-1}\vec n)$ to yield
	\begin{align}
		\sum_{\vec k \in T_\jl} \abs{w_\lambda \beta^\rms_\lambda}^q 
		&\lesssim \int \ind_{\curly*{\abs[]{\vec x} > \parens[]{2^{j+1}\abs{\sin\theta_\jl}}^{\frac \delta d}}} \frac{2^{\frac {jq}2}}{(|\vec x|^2+1)^{mq}} a^{1-q}\d \vec x \\
		&\lesssim  \int_{\parens[]{2^{j+1}\abs{\sin\theta_\jl}}^{\frac \delta d}}^\infty \frac{2^{\frac {jq}2}}{(r^2+1)^{mq}} r^{d-1} a^{1-q} \d r \\
		&\lesssim  2^{\frac {jq}2} 2^{(j-r)(-\frac{2mq-d}{d}\delta+1-q)}
		\lesssim 2^{\frac {jq}2}2^{(j-r)(-\frac{2mq\delta}{d}+1+\delta)}. \label{eq:est_transl_sing}
	\end{align}
	%
	Summing over $\ell \in \Lambda^j_r$, which has cardinality $\# \Lambda^j_r\lesssim 2^{(j-r)(d-1)}$, this implies
	\begin{align}
		\sum_{\ell\in\Lambda^j_r}\sum_{\vec k \in T_\jl} \abs{w_\lambda \beta^\rms_\lambda}^q \lesssim 2^{\frac {jq}2}2^{-(j-r)(\frac{2mq\delta}{d}-d-\delta)}.
	\end{align}
	
	At first glance this estimate might seem of questionable benefit, since the first term has a positive power of $j$. However, \eqref{eq:gamma_est_sum_Lambda} and the following interpolation inequality\footnote{Log-convexity of $L^p$-norms; follows from applying H\"older's inequality to $\abs{f}=\abs{f}^\theta \abs{f}^{1-\theta}$.} for a sequence $(c_i)_{i\in I}$ will save the day,
	\begin{align}
		\norm{c}_{\ell^p} \le \norm{c}_{\ell^q}^\theta \norm{c}_{\ell^2}^{1-\theta} \qquad \text{where} \qquad \frac 1p = \frac{\theta}q + \frac{1-\theta}{2}.
	\end{align}
	In particular, we have that
	\begin{align}
		\parens[\bigg]{\sum_{\ell\in\Lambda^j_r}\sum_{\vec k \in T_\jl} \abs{w_\lambda \beta^\rms_\lambda}^p}^{\frac 1p} 
		&\lesssim \parens{2^{\frac {j}2}2^{-(j-r)(\frac{2m\delta}{d}-\frac{d+\delta}{q})}}^\theta \parens{2^{-\frac j2} 2^{-(j-r)(t-\frac12)}}^{1-\theta}\\
		&=2^{-j(\frac 12 -\theta)} 2^{-(j-r)\parens{\theta(\frac{2m\delta}{d}-\frac{d+\delta}{q})+(1-\theta)(t-\frac12)}}, \label{eq:sum_gamma_p_theta}
	\end{align}
	which implies that for $\theta<\frac 12$ (i.e.~$q$ sufficiently small) and $n$ sufficiently large, we have achieved a finite $\ell^p$-norm for the tail of the singular part, as we can now simply sum over $r$:
	\begin{align}
		\sum_{\ell}\sum_{\vec k \in T_\jl} \abs{w_\lambda \beta^\rms_\lambda}^p
		&\lesssim 2^{-jp(\frac 12 -\theta)}.
	\end{align}
	The condition $\theta<\frac 12$ prescribes an upper bound for $q$ depending on the desired $p$, namely that $\theta<\frac 12 \Longleftrightarrow \frac 1q > \frac 2p-\frac 12$.
	
	In the same way as above, we get that
	\begin{align}
		\sum_{\ell\in\Lambda^j_r}\sum_{\vec k \not\in S_\jl} \abs{\alpha^\rms_\lambda}^q \lesssim 2^{-\frac {jq}2}2^{-(j-r)(\frac{2nq\delta}{d}-d-\delta)},
	\end{align}
	which lets us set $q=p'$ directly without having to interpolate (which leads to a slightly improved estimate for $\delta^*$ than the one below, compare also Step 7).
	
	\step{Singular part, estimating $\delta^{*}$}[\label{stp:approx:sing_delta}]
	To determine a lower bound for $\delta^{*}$ in dependence of $n$, we observe that $\frac{2m\delta}{d}>\frac{d+\delta}{q}$ has to hold for the second exponent in \eqref{eq:sum_gamma_p_theta} to be negative (actually this is a source for potential improvement for $\delta$, since, due to $\theta<\frac 12$, we could afford to let the first term become slightly negative; we will not deal with this, though). Together with the condition for $\theta$ and for the $p$ we want to achieve, this implies
	\begin{align}
		\frac{2m\delta}{d(d+\delta)}>\frac{1}{q}>\frac{2}{p^*+\delta^{*}}-\frac 12 = \frac{2t+d}{d+\delta}-\frac 12,
	\end{align}
	where we recalled $p^*=(\frac td+\frac 12)^{-1}$ and $\delta^{*}=\delta(t+\frac d2)^{-1}$. From this, we can use this to deduce a worst-case lower bound for the deviation from the optimal $p^*$ to make our argument work, namely
	\begin{align}
		\delta\parens[\Big]{\frac{2m}{d}+\frac 12}> 2t+\frac d2 \qquad \Longrightarrow \qquad \delta^{*}>2 \smash{\underbrace{\frac{t+\frac d4}{t+\frac d2}}_{<1}} \frac{1}{\frac {2m}d+\frac 12},
	\end{align}
	which is satisfied in particular if
	\begin{align}
	\delta^{*}\ge \frac{4d}{4n+d}. 
	\end{align}
	Said otherwise, everything we showed is true if we choose $\delta^{*}=\frac dm$, which, conversely, means that the coefficient sequence is \emph{at least} in the space $\ell^{p^*+\frac dm}_w$, and possibly in an even smaller space $\ell^{p^*+\bar{\delta}}_w$, i.e.~with $\bar{\delta}\le\frac dm$.
	
	\step{Regular part, large coefficients}[\label{stp:approx:reg_large}]
	Like for the singular part, we split the sequence into two parts, again defining a \emph{tail}
	\begin{align}\label{eq:tail_reg}
		T&:=\set[\Big]{\vec k\in\bbZ^d}{\abs*{D_{2^{-j}}\vec k} > 2^{j\frac{\delta}{d}}+\frac{\sqrt{d}}2},
	\end{align}
	only this time, both subsequences will be summable in $\ell^{p'}\subseteq \ell^{p'}_w$. Like before, we estimate the cardinality of the complement $T^\compl:=\bbZ^d\setminus T$,
	\begin{align}
		\#T^\compl \le \int \ind_{\{D_{2^{-j}}\vec k<2^{j\frac{\delta}{d}}+\sqrt{d}\}} \d \vec k = 2^j \int_{0}^{2^{j\frac{\delta}{d}}+\sqrt{d}} r^{d-1} \d r \lesssim 2^{j(1+\delta)}.
	\end{align}
	
	The following estimate will be the key to finish this step. For sequences $f,g$ with $g>0$ almost everywhere and some $q>1$, we apply H\"older's inequality,
	\begin{align}
		\norm*{\abs{f}^{\frac 1q}}_{\ell^1} = \norm*{\abs{fg}^{\frac 1q}\abs{g}^{-\frac 1q}}_{\ell^1} \le \norm*{\abs{fg}^{\frac 1q}}_{\ell^q} \norm*{\abs{g}^{-\frac 1q}}_{\ell^{q'}} = \norm*{fg}_{\ell^1}^{\frac 1q} \norm*{\abs{g}^{-\frac 1{q-1}}}_{\ell^{1}}^{\frac{q-1}q}
	\end{align}
	where $q'=\frac{q}{q-1}$, which is sometimes called the \emph{reverse H\"older inequality}, because then, by taking $q$\nth powers and bringing the last factor to the left-hand side, ``$\norm{fg}_{\ell^1}\ge \norm{f}_{\ell^{\frac 1q}} \norm{g}_{\ell^{-\frac{1}{q-1}}}$'', if we allowed negative exponents in the $\ell^p$-quasi-norms.
	
	In our case, we chose $q=\frac 2{p'}$, $f:=\abs{\VW\beta^\rmr_\Lambda}^2$ and $g=\ind_{\{j,\vec k\in T^\compl\}}$, allowing us to apply the above inequality,
	\begin{align}
		\norm*{\VW\beta^\rmr_\Lambda\normalr|_{j,T^\compl}}_{\ell^{p'}}^2 = \parens[\bigg]{\sum_{\ell=1}^{L_j} \sum_{\vec k\in T^\compl}\abs{w_\lambda \beta^\rms_\lambda}^{p'}}^{\frac 2{p'}} = \norm*{\VW\beta^\rmr_\Lambda\normalr|_{j,T^\compl}}_{\ell^{\frac 1q}} \le \norm*{\VW\beta^\rmr_\Lambda\normalr|_{j,T^\compl}}_{\ell^2}^2 \norm*{\ind_{\{j,\vec k\in T^\compl\}}}_{\ell^{1}}^{q-1}.
	\end{align}
	Since $g$ is constant, the last norm simply evaluates to $L_j\cdot\#T^\compl\lesssim 2^{j(d+\delta)}$ and therefore, using \eqref{eq:gamma_est_reg}, we have
	\begin{align}
		\norm*{\VW\beta^\rmr_\Lambda\normalr|_{j,T^\compl}}_{\ell^{p'}}^2 \lesssim \eps_j^2 2^{-2jt} 2^{j(d+\delta)(\frac 2{p'}-1)}.
	\end{align}
	Computing
	\begin{align}
		\frac{1}{p'}=\parens[\Big]{p^*+\frac{\delta}{t+\frac d2}}^{-1} =\parens[\Big]{p^*+\frac{\delta}{t+\frac d2}}^{-1}=\frac{\frac td +\frac 12}{1+ \frac \delta d}
	\end{align}
	makes it clear that the second exponent simplifies to
	\begin{align}
		2jd\parens[\Big]{1+\frac \delta d}\parens[\bigg]{\frac{\frac td +\frac 12}{1+ \frac \delta d}-\frac 12}
		=2jd\parens[\Big]{\frac td - \frac{\delta}{2d}}=2j\parens[\Big]{t - \frac{\delta}{2}},
	\end{align}
	and thus
	\begin{align}
		\norm*{\VW\beta^\rmr_\Lambda\normalr|_{j,T^\compl}}_{\ell^{p'}} \lesssim \eps_j 2^{-j\frac{\delta}{2}},
	\end{align}
	which is summable in $j$.
	
	\step{Regular part, tail}
	For the tail of $\VW\beta^\rmr_\Lambda$, we need to use the decay of $f$ resp.~$\CS[f]$ to apply \autoref{lem:ridge_coeff_tail_decay},
	\begin{align}
		\sum_{\vec k\in T} \abs{w_\lambda \beta^\rmr_\lambda}^{p'}
		&\lesssim \sum_{\vec k\in T}  \frac{2^{-\frac {jp'}2}}{\parens*{\parens*{\frac{k_1}{2^j}}^2+\abs[]{\kp}^2+1}^{mp'}}
		\le 2^{-\frac {jp'}2} \int \ind_{\{\abs[]{D_{2^{-j}}\vec k} > 2^{j\frac{\delta}{d}}\}} \frac{1}{\parens*{\parens*{\frac{k_1}{2^j}}^2+\abs[]{\kp}^2+1}^{mp'}} \d \vec k.
		\intertext{
	Transforming with $\vec t=D_{2^{-j}}\vec k$, we are able to continue,
		}
		&= 2^{-\frac {jp'}2} \int_{\abs[]{\vec k}>2^{j\frac{\delta}{d}}}  \frac{2^j}{\parens*{\abs[]{\vec t}^2+1}^{mp'}} \d \vec t
		= 2^{j\parens[]{1-\frac{p'}2}} \!\! \int_{2^{j\frac{\delta}{d}}}^\infty  r^{d-1-2mp'} \d r = 2^{-j\parens[]{(2mp'-d)\frac{\delta}{d}+\frac{p'}2-1}}.
	\end{align}
	Since this estimate is independent of $\ell$, we can sum over the $L_j\sim 2^{j(d-1)}$ directions on scale $j$ to arrive at
	\begin{align}\label{eq:sum_reg_tail}
		\sum_{\ell=1}^{L_j} \sum_{\vec k\in T} \abs{w_\lambda \beta^\rmr_\lambda}^q \lesssim 2^{-j\parens{\frac{2mp'\delta}{d}+\frac{p'}2-d-\delta}}.
	\end{align}
	For $m$ sufficiently large, resp.~$\delta$ not too small (see below), the exponent is negative, and we can thus sum in $j$ and have shown that the tail of the regular part is in $\ell^{p'}\subseteq \ell^{p'}_w$.
	
	\step{Regular part, estimating $\delta^{*}$}
	Similarly to \autoref{stp:approx:sing_delta}, we check which condition $\delta$ (resp.~$\delta^{*}$) needs to satisfy so that the results from above hold (i.e.~that the exponent \eqref{eq:sum_reg_tail} is negative). Apart from the negligible term $\frac{p'}2$ in the exponent, this is the same condition we had for the tail of the singular part, only that now, we don't need to interpolate,
	\begin{align}
		\frac{2m\delta}{d(d+\delta)}>\frac{1}{p'}=\frac{1}{p^*+\delta^{*}} = \frac{t+\frac d2}{d+\delta}.
	\end{align}
	Therefore, the condition for $\delta^{*}$ becomes
	\begin{align}
		\delta>\frac{d}{2m}\parens[\Big]{t+\frac d2} \qquad \Longrightarrow \qquad \delta^{*}> \frac{d}{2m},
	\end{align}
	which is weaker than the condition from \autoref{stp:approx:sing_delta}, and satisfied in particular for $\frac dm$.
	
	This allows us to round up the results thus far, and we observe that, taken together, Steps 1--\thestep\ prove the claim of the theorem for $t\in\bbN$.
	
	\step{Interpolating in $t$}
	What remains to show is to extend this to the half-line $t>0$ via interpolation theory. Taking the operators defined by
	\begin{align}
		T_1 f:=\inpr{\Phi,f}, \qquad T_2 f:=\VW\inpr{\Phi,\CS[f]},
	\end{align}
	we know that, by the frame property, resp.~by \eqref{eq:ridge_Hs_stable},
	\begin{align}
		\norm{T_i}_{L^2\to\ell^2}&<\infty, \quad i=1,2.
	\end{align}
	On the other hand, as we have just proved above for any $t\in\bbN$, the $T_i$ also satisfy
	\begin{align}
		\norm{T_i}_{H^t\to\ell^{p^*+\delta^*}_w}&<\infty, \quad i=1,2,
	\end{align}
	where $\delta^*\le \frac dm$ and $\frac{1}{p^*}=\frac{t}{d}+\frac 12$. Using the results from \autoref{thm:interp}, we see that
	\begin{align}
		\norm{T_i}_{H^{t\theta}\to\ell^{\bar{p}}_2}&<\infty,
	\end{align}
	where $\frac{1}{\bar{p}}=\frac{1-\theta}{2} + \frac{\theta}{p^*+\delta^*}$. Letting $\bbR^+\ni \bar{t}=t\theta$ (regardless of the choice of $t\in\bbN$ and $0<\theta<1$), we compute
	\begin{align}
		\frac{1}{\bar{p}}=\frac{1-\theta}{2} + \frac{\theta(t+\frac d2)}{d+\delta}=\frac{\bar{t}}{d+\delta}+\frac{\theta}{2}\parens[\Big]{\frac{d}{d+\delta}-1} +\frac 12,
	\end{align}
	and thus, since $\theta<\frac 12$, we can deduce
	\begin{align}
		\bar{p}=\frac{d+\delta}{\bar{t}+\frac d2 + \frac{\delta}{2}(1-\theta)} \le \frac{d+\delta}{\bar{t}+\frac d2} = \parens[\Big]{\frac {\bar{t}}d+\frac 12}^{-1} +\delta^{*}.
	\end{align}
	Finally, we remark that $\ell^{\bar{p}}_2\subseteq \ell^{\bar{p}}_w$, and thus the proof is finished, since for arbitrary $t\in\bbR^+$, we have shown that for $f\in H^{t}(\bbR^d)$ with solution $\CS[f]$, it holds that
	\begin{align}
	\inpr{\Phi,f}\in\ell^{p^*+\frac dm}_w \quad \text{as well as} \quad \VW\inpr{\Phi,\CS[f]}\in\ell^{p^*+\frac dm}_w \quad \text{where} \quad p^*=\parens[\Big]{\frac td+\frac 12}^{-1}.
	\end{align}
\end{proof}

\subsection{Implications}

To conclude this section, we briefly discuss differences to the proof and results in \cite{mutilated}.

\begin{remark}
	An obvious question that presents itself with regard to \cite{mutilated} is why Cand\`es was able to achieve $p=p^*$, while we only achieve this value up to an arbitrarily small $\delta^*$. Aside from the fact that the functions we treat do not have to have compact support, one major source for this loss is the fact that we have so many more translations to deal with ($d$ dimensions vs. one in \cite{mutilated}).
	
	To achieve $p=p^*$, we would have to:
	\begin{enumerate}[(A)]
	\item
		Bound \eqref{eq:card_Sjl} in a way that depends purely linearly from $2^{j-r}$ --- since we can see that the deviation $\delta$ directly impacts \eqref{eq:est_p4sing} through \eqref{eq:card_Njr}.
	\item
		On the other hand, the $\ell^q$ norm of the tail in \eqref{eq:est_transl_sing} still has to have an exponent of $2^{j-r}$ that is negative for sufficiently large $m$.
	\end{enumerate}
	The transformation $V_\jl$ in $T_\jl$ is already chosen in a way to be able to optimally estimate \eqref{eq:est_transl_sing} by achieving radial symmetry, and due to its functional determinant $a\sim 2^{j-r}$, (A) becomes almost impossible. One may squeeze a slight improvement out of \eqref{eq:gamma_est_w} by pulling a factor $a^{\frac 1m}$ in the denominator, which would reduce the determinant of transforming with $V_\jl$ to $a^{1-\frac{d}{2m}}$. Still, this improvement is not enough, as it would force $\delta\le\frac{d}{2m}$ in (A), making (B) impossible, since $m$ would disappear from the exponent. Also, the $\ell^p$-interpolation does not save us, due to the fact that the condition on $\theta$ would necessitate $q<0$.
	
	One solution would be if the fraction $\frac 1a$ in \eqref{eq:gamma_est_w} had any power that grew with $m$, even if it were as slow as $\log(m)$ --- as this would save (B) in the above scenario. But in general, due to the sharpness of \eqref{eq:Imn_est}, this seems unlikely to be possible. Nevertheless, we do not rule out that $\delta^*$ could be eliminated through other proof techniques for functions with compact support.
\end{remark}

\begin{remark}
	Finally, while we were able to follow the general approach of \cite{mutilated} --- i.e.~localisation in angle, localisation in space, split off tails etc. --- we had very different issue to deal with. On the one hand, we were able to avoid the sampling estimates in Fourier space (which is essentially due to the fact that Cand\`es construction is supported on $[\sjl]$ in frequency, while our construction diffuses this to the sets $P_\jl$ where we are able to integrate), but on the other hand, issues like the localisation in space became much more intricate and required ``heavy machinery'' (in the form of \autoref{th:Imn_est}) to deal with. The removal of the condition of compact support --- resp.~being also able to quantify the coefficient decay for functions with only polynomial decay --- is a welcome bonus, as is the fact that we were able to avoid the numerous case distinctions that are often necessary in other proofs of optimality (compare e.g.~\cite{curvelet_cartoon}).
\end{remark}

\section{Conclusion}\label{sec:conclusion}

Wrapping up, the fact that ridgelets approximate mutilated Sobolev functions with the benchmark rate --- up to arbitrarily small $\delta>0$ --- for this class (as set out in \autoref{sec:benchmark}) allows us to combine \autoref{th:approx} with \cite{compress} to achieve \autoref{th:main_result}. To the best of our knowledge,
this is the first construction of an optimally adapted PDE solver with non-standard frames and for non-elliptic
problems.

As indicated in \autoref{ssec:impact}, this can be utilised for solving more involved transport equations \eqref{eq:RTE}, based on the fact that the ridgelet frame $\Phi$ covers all direction $\vec s$ simultaneously, while the multiscale structure makes it possible to alleviate the curse of dimensionality, for example by the ``sparse discrete ordinates method'' (see e.g. \cite{grella}).

Finally, as a numerical check of the claimed results, we return to \autoref{fig:sing_sol}, which is the solution of \eqref{eq:LinTrans} with a source function that is a box function times a Gaussian, rotated so that the edges of the box are aligned with the transport direction $\vec s$ (which is chosen to not coincide with any of the $\sjl$ in the frame). Since --- apart from the singularities --- the function is $\CC^\infty$, we even observe super-polynomial decay, in the sense that in the doubly logarithmic plot \autoref{fig:sing_sol:nterm}, the curve has non-zero curvature (as far as we were able to calculate) and overtakes any straight line (which would correspond to a fixed power $N^{-\sigma}$).

Furthermore, the localisation in angle discussed in \autoref{prop:loc_angle}, is also confirmed theoretically. Indeed, the results are again very convincing, in that, for at least the first $100'000$ largest coefficients (on scales $1$--$10$), the corresponding rotational parameters $\ell$ are all either the $\sjl$ closest to the direction $\vec s$ or its immediate left/right neighbours --- across all scales! To put this into perspective, there are $2^{j+1}$ different $\sjl$ on each scale, and we can discard \emph{all but three} of them without discernible loss in accuracy. For further numerical results in this context see \cite[Sec. 8.2]{my_thesis}.

\begin{figure}
	\setlength{\plotsize}{0.5\linewidth}
	{
	\subcaptionbox{Solution for singular $f$\label{fig:sing_sol:rhs}}{%
		\tikzsetnextfilename{box_sol_sml}%
%
%
\definecolor{myblue}{rgb}{0,0,0.56078}%
\begin{tikzpicture}[trim axis left,trim axis right]
\begin{axis}[%
	width=\plotsize,
	xlabel={\footnotesize$x_1$},
	ylabel={\footnotesize$x_2$},
	xlabel shift={-0.75em},
	ylabel shift={-0.75em},
	scale only axis,
	tick align=outside,
	clip=false,
	xmin=-1,
	xmax=1,
	xmajorgrids,
	ymin=-1,
	ymax=1,
	ymajorgrids,
	zmin=0,
	zmax=2,
	zmajorgrids,
	every tick label/.append style={font=\footnotesize},
	axis x line*=bottom,
	axis y line*=left,
	axis z line*=left
	]
	
	\addplot3 graphics[points={
	(-0.99665,-0.98692,9.7769e-11) => (55.7081,105.3937)
	(0.97341,0.98941,1.7626e-10) => (377.9119,156.585)
	(0.023055,0.054493,1.0839) => (223.5853,168.63)
	(0.96699,-0.96354,2.7661e-09) => (260.4731,37.6406)
	}]{figures/source/box_sol.png};
	
	\def\myxtick{0.055}
	\def\myytick{0.07}
	\def\myztick{0.19}
	\def\mylw{0.005}
	
	\pgfplotsinvokeforeach{0,1,2}
	{
	\draw[white] (-1,-1,#1) -- (-1-\myxtick,-1,#1);
	\draw[white] (-1,-1,#1) -- (-1,-1-\myytick,#1);
	\draw[very thin,gray] (-1,-1,#1) -- (-1,-1-\myytick,#1);
	}
	\pgfplotsinvokeforeach{-1,-0.5,0,0.5,1}
	{
	\draw[white] (#1,-1,0) -- (#1,-1,-\myztick);		
	\draw[white] (#1,-1,0) -- (#1,-1-\myytick,0);		
	\draw[very thin,gray] (#1,-1,0) -- (#1,-1-\myytick,0);
	}
	\pgfplotsinvokeforeach{-1,-0.5,0,0.5,1}
	{
	\draw[white] ( 1,#1,0) -- ( 1,#1,-\myztick);		
	\draw[very thin,gray] (1,#1,0) -- (1+\myxtick,#1,0);
	}
	\draw[very thin,gray] (1,-0.99,0) -- (1,-1-\myytick,0);
	\draw[myblue,thick] (-1+\mylw,-1+\mylw,0) -- (1-\mylw,-1+\mylw,0) -- (1-\mylw,1-\mylw,0) -- (-1+\mylw,1-\mylw,0) -- cycle;
	\draw[white,thick] (-1,1,0.02) -- (1,1,0.02);
	\draw[gray!50] (-1,-1,0) -- (-1,1,0) -- (1,1,0);
	\pgfplotsinvokeforeach{-1,-0.5,0,0.5,1}
	{
	\draw[gray!50] (#1,1,0) -- (#1,1,1);		
	}
	\draw (-1,-1,2+\myztick) -- (-1,-1,0) -- (1,-1,0) -- (1,1,0);
\end{axis}
\end{tikzpicture}%
}}%
	\hfill\setlength{\plotsize}{0.35\linewidth}%
	\subcaptionbox{$N$-term approximation rate\label{fig:sing_sol:nterm}}{%
		\tikzsetnextfilename{Nterm_sing}%
%
%
\definecolor{mycolor1}{rgb}{0.00000,0.44700,0.74100}%
\definecolor{mycolor2}{rgb}{0.85000,0.32500,0.09800}%
\definecolor{mycolor3}{rgb}{0.92900,0.69400,0.12500}%
\definecolor{mycolor4}{rgb}{0.49400,0.18400,0.55600}%
\definecolor{mycolor5}{rgb}{0.46600,0.67400,0.18800}%
\definecolor{mycolor6}{rgb}{0.30100,0.74500,0.93300}%
\definecolor{mycolor7}{rgb}{0.63500,0.07800,0.18400}%
\begin{tikzpicture}[trim axis left,trim axis right]
\begin{axis}[%
width=\plotsize,
height=\plotsize,
scale only axis,
unbounded coords=jump,
xmode=log,
xmin=10,
xmax=100000,
xminorticks=true,
xlabel={\footnotesize Number $N$ of ridgelets used for reconstruction},
ylabel={\footnotesize Relative $L^2$-error},
ymode=log,
ymin=0.0001,
ymax=1,
yminorticks=true,
every tick label/.append style={font=\footnotesize},
axis background/.style={fill=white}
]
\addplot [color=mycolor1,solid]
  table[row sep=crcr]{%
20	0.400253717065241\\
40	0.366785504405462\\
60	0.338509829701719\\
80	0.316387129535475\\
100	0.298352573420227\\
120	0.280634959648918\\
140	0.267123514012344\\
160	0.257420908415868\\
180	0.246439015727762\\
200	0.235438273250285\\
220	0.224129991955433\\
240	0.214240368735786\\
260	0.206878160115163\\
280	0.200572998868481\\
300	0.19436767238386\\
320	0.186646207582423\\
340	0.181035630558091\\
360	0.172262317933199\\
380	0.166070580759316\\
400	0.162442187751984\\
800	0.108580436982768\\
1200	0.0843392508718165\\
1600	0.0679698273867626\\
2000	0.0576735604547602\\
2400	0.0492195014500669\\
2800	0.0433007324274309\\
3200	0.0385861094737698\\
3600	0.0349202409971326\\
4000	0.0316176713327697\\
4400	0.0287405442230186\\
4800	0.0264194175659328\\
5200	0.0240999953836969\\
5600	0.0225593169424807\\
6000	0.0210117661654486\\
6400	0.0195951034171552\\
6800	0.0182806139710217\\
7200	0.0171503153004934\\
7600	0.0159448866974774\\
8000	0.0150906198045825\\
8400	0.0140564489142905\\
8800	0.0132138752063825\\
9200	0.0122219395625939\\
9600	0.0115101689122792\\
10000	0.0109216705466815\\
10400	0.0103085232003954\\
10800	0.00976979387441293\\
11200	0.00924494300372729\\
11600	0.00883383117411731\\
12000	0.00839110968598781\\
12400	0.00796460378357287\\
12800	0.00762300093971332\\
13200	0.00722006890584702\\
13600	0.00687734649969845\\
14000	0.00651938988139802\\
14400	0.00620588737989716\\
14800	0.00598469378887903\\
15200	0.00574531492491916\\
15600	0.00548596570592125\\
16000	0.00523961520802978\\
16400	0.00502165242048476\\
16800	0.00482798692859305\\
17200	0.00466677309842402\\
17600	0.00451050142666545\\
18000	0.00438855641449497\\
18400	0.004221055358441\\
18800	0.00405914161472572\\
19200	0.00393174112448324\\
19600	0.00378151802125968\\
20000	0.00365345307986607\\
20400	0.00353070779612971\\
20800	0.00339701722126172\\
21200	0.00327899434893389\\
21600	0.0031526015892518\\
22000	0.00303667657880326\\
22400	0.00294452374252543\\
22800	0.00284946762908533\\
23200	0.00277847333102539\\
23600	0.00269896689218675\\
24000	0.0026277154401427\\
24400	0.00254922964864998\\
24800	0.0024748771892525\\
25200	0.00241529983111379\\
25600	0.00236178581916897\\
26000	0.0022964832724088\\
26400	0.00223757907358718\\
26800	0.00217395019344824\\
27200	0.00212648563584064\\
27600	0.00207077951127554\\
28000	0.00200920196004147\\
28400	0.00195448142533384\\
28800	0.00191967790726867\\
29200	0.00188008318821874\\
29600	0.00183895817498489\\
30000	0.00179534029018666\\
30400	0.00175361562623122\\
30800	0.00169336844801087\\
31200	0.0016533593948564\\
31600	0.00161606290646892\\
32000	0.0015779240651339\\
32400	0.00154184076086434\\
32800	0.00151380817109019\\
33200	0.0014810376168379\\
33600	0.00145121566181024\\
34000	0.00142878018424134\\
34400	0.00140017358476354\\
34800	0.00138069597404368\\
35200	0.00135142132568477\\
35600	0.00132248550887982\\
36000	0.00129223086797639\\
36400	0.00127344965135168\\
36800	0.00124470223314735\\
37200	0.00122379648552762\\
37600	0.00119355868423493\\
38000	0.00116059038551262\\
38400	0.0011374959440566\\
38800	0.00111776376168619\\
39200	0.00109904841167898\\
39600	0.0010809249556525\\
40000	0.00106285023117039\\
40400	0.00104047588901495\\
40800	0.00102105565726185\\
41200	0.00100294922641198\\
41600	0.000985722612921455\\
42000	0.000967676505179647\\
42400	0.000948932975021804\\
42800	0.000935229438152228\\
43200	0.000922382849881713\\
43600	0.000905494936270761\\
44000	0.000888831842153098\\
44400	0.000876943305193061\\
44800	0.000863315800418887\\
45200	0.000850960614230924\\
45600	0.000835967917428397\\
46000	0.000821830003966122\\
46400	0.000809061485684631\\
46800	0.000795350211189422\\
47200	0.000782110551465521\\
47600	0.000768457986853257\\
48000	0.000753186598026425\\
48400	0.000740340935229143\\
48800	0.000727796560565427\\
49200	0.00071308792151991\\
49600	0.00069845331374364\\
50000	0.000688474087192801\\
50400	0.000679067961363729\\
50800	0.000670538542828765\\
51200	0.00065887648085558\\
51600	0.000647298567210067\\
52000	0.000632983007731799\\
52400	0.000623564027527762\\
52800	0.000614181668982787\\
53200	0.000603394125595843\\
53600	0.00059364788088969\\
54000	0.000585661734952078\\
54400	0.000579103366229247\\
54800	0.000571304128517272\\
55200	0.000563410587440862\\
55600	0.00055512359262748\\
56000	0.000546905226205307\\
56400	0.00054015688425051\\
56800	0.00053283670393377\\
57200	0.000526909095219483\\
57600	0.00052002033934293\\
58000	0.000514331853675557\\
58400	0.000505697985153354\\
58800	0.000500099631121234\\
59200	0.000493476624603522\\
59600	0.000487775559589444\\
60000	0.000481569027871918\\
60400	0.000475900483311146\\
60800	0.000469602845969755\\
61200	0.00046461611445901\\
61600	0.000458956247964244\\
62000	0.000451970168187945\\
62400	0.000448013577242851\\
62800	0.000441510664180219\\
63200	0.000438061138160377\\
63600	0.000434066993498719\\
64000	0.000429367272453699\\
64400	0.000425115691160745\\
64800	0.000418216821850554\\
65200	0.00041246371566326\\
65600	0.000406137112244872\\
66000	0.000400770216458566\\
66400	0.000396930542467272\\
66800	0.000392800633085096\\
67200	0.000388776503544264\\
67600	0.000383139037339925\\
68000	0.000377877249708062\\
68400	0.000373894953447406\\
68800	0.000370050828652463\\
69200	0.000365585935834494\\
69600	0.0003610132336397\\
70000	0.000356545819844313\\
70400	0.000351711836799892\\
70800	0.000346929612170912\\
71200	0.00034339962166354\\
71600	0.000339418865941148\\
72000	0.000335835915393634\\
72400	0.000331547640088736\\
72800	0.000327987697482023\\
73200	0.000325193918131843\\
73600	0.000321568752370638\\
74000	0.000317717806265301\\
74400	0.000313646350621568\\
74800	0.000309484515597228\\
75200	0.000306249968934425\\
75600	0.000303620055126596\\
76000	0.000300807050142987\\
76400	0.000297555314091918\\
76800	0.000294070538141204\\
77200	0.000291419389271804\\
77600	0.000288889817890459\\
78000	0.000285349145129653\\
78400	0.000282489171646123\\
78800	0.00028070618930509\\
79200	0.00027745895804107\\
79600	0.000274091012512273\\
80000	0.000271313016121651\\
80400	0.000268285165636292\\
80800	0.000265319647105734\\
81200	0.000262181950524269\\
81600	0.000260003112804151\\
82000	0.000258360563712762\\
82400	0.000255901189380579\\
82800	0.000252929456095915\\
83200	0.000249949999516145\\
83600	0.000247467112081717\\
84000	0.000245688273276501\\
84400	0.000244225013762706\\
84800	0.000241065312292236\\
85200	0.000239088896276795\\
85600	0.000236419334484655\\
86000	0.000233313033108251\\
86400	0.00023046836685325\\
86800	0.000228304308198342\\
87200	0.000226922279115262\\
87600	0.000225096548423486\\
88000	0.000223082756242354\\
88400	0.000220452230144545\\
88800	0.000217944605917957\\
89200	0.000215769225289692\\
89600	0.000213658267773207\\
90000	0.000211339127243784\\
90400	0.000209065381045004\\
90800	0.000207084934007732\\
91200	0.000204662628784685\\
91600	0.000202828621227884\\
92000	0.000200830545299225\\
92400	0.000199024157373629\\
92800	0.000196987236491504\\
93200	0.000194913443942539\\
93600	0.000192820120215648\\
94000	0.000190542174216383\\
94400	0.000188606317975546\\
94800	0.000186705999381065\\
95200	0.000184708794412508\\
95600	0.000183088389434769\\
96000	0.000180740078300988\\
96400	0.0001789685277761\\
96800	0.000176992368307468\\
97200	0.0001743346636375\\
97600	0.000173192759950084\\
98000	0.000171538963282594\\
98400	0.000170385207228866\\
98800	0.00016855922242743\\
99200	0.000167654459535419\\
99600	0.000165925496380674\\
100000	0.000164509483176166\\
};

\end{axis}
\end{tikzpicture}%
%
	}
	\caption{Solution to the advection equation for a singular right-hand side, as well the $N$-term approximation rate of the ridgelet frame}\label{fig:box_sol}
\end{figure}

\section{An Integral (In)Equality}\label{sec:int_est}

The following result turned out to be necessary for the proof of \autoref{th:approx}, but has proven useful in other contexts as well (e.g.~\autoref{lem:decay_sol}). It is very well-suited for quantifying interactions between (offset) decaying functions --- particularly for convolutions (see \autoref{app:cor:Imn_higher_dim}) --- and substantially stronger than other results of this type we are aware of (see \autoref{rem:grafakos}).

\subsection{Main Theorem \texorpdfstring{$\&$}{and} Consequences}

The following theorem is formulated not in its most general form, but in a form that any one-dimensional problem of this type can be transformed into.

\begin{theorem}\label{app:th:Imn_est}
	For $m,n\in\bbN$, $a\in\bbR^+_0$, $b\in\bbR$, $c,d\in\bbR^+$, we have
	\begin{align}
		I_{m,n}&:=\int_{-\infty}^{\infty} \frac{1}{\parens{a^2(x-b)^2+c^2}^m} \frac{1}{\parens{x^2+d^2}^n} \d x \notag \\
		&\phantom{:}= \frac{\pi}{\parens{(c+ad)^2+a^2 b^2}^{m+n-1}} \frac{1}{c^{2m-1}} \frac{1}{d^{2n-1}} \sum_{\substack{i+j+2k=2(m+n)-3\\i\ge 2m-1 \, \lor\, j\ge 2n-1}} c^{m,n}_{i,j} c^{i} (ad)^{j} (ab)^{2k}\notag \\
		&\phantom{:}\lesssim  \frac{a^{2n-1}}{\parens{(c+ad)^2+a^2 b^2}^n} \frac{1}{c^{2m-1}} + \frac{1}{\parens{(c+ad)^2+a^2 b^2}^m} \frac{1}{d^{2n-1}} \label{app:eq:Imn_est_best}\\
		&\phantom{:}\le \frac{a^{2n-1}}{\parens{a^2b^2+a^2d^2+c^2}^n} \frac{1}{c^{2m-1}} + \frac{1}{\parens{a^2b^2+a^2d^2+c^2}^m} \frac{1}{d^{2n-1}}. \label{app:eq:Imn_est}
	\end{align}
	Furthermore, the generating function for $I_{m,n}$ is
	\begin{align}\label{app:eq:genfunc_Imn_intro}
		\sum_{m,n\ge 1} I_{m,n} y^m z^n = \frac{\pi yz}{\sqrt{c^2-\smash{y}}\sqrt{d^2-z}} \frac{\sqrt{c^2-\smash{y}}+a\sqrt{d^2- z}}{(\sqrt{c^2-\smash{y}}+a\sqrt{d^2- z})^2+a^2b^2},
	\end{align}
	and with $h(v,w):=(v+w)^2+1$, we also have a generating function for the coefficients,
	\begin{multline}\label{app:eq:genfunc_coeff}
		\sum_{\substack{i,j\ge 0\\m,n\ge 1}} c^{m,n}_{i,j} v^i w^j y^m z^n \\
		=\frac{h(v,w)yz}{\sqrt{1-h(v,w)y}\sqrt{1-h(v,w)z}} \frac{v\sqrt{1-h(v,w)y}+w\sqrt{1-h(v,w)z}}{\parens*{v\sqrt{1-h(v,w)y}+w\sqrt{1-h(v,w)z}}^2+1}.\qquad
	\end{multline}
	The coefficients are zero \emph{unless} the following conditions are satisfied,
	\begin{align}
	i+j&\equiv 1\bmod{2}, & i+j&\le 2(m+n)-3, & (i&\ge 2m-1 \lor j\ge 2n-1).
	\end{align}
	If these are satisfied, the coefficients can be calulated as follows,
	\begin{align}
		c^{m,n}_{i,j}
		\begin{multlined}[t][\linewidth-\mathindent-\widthof{$c^{m,n}_{i,j}$}-\multlinegap]
			=\sum_{r=0}^{i}\sum_{s=0}^{j} \delta_{\{r+s\equiv1\bmod{2}\}} (-1)^{m+n+\frac{r+s-1}{2}} \binom{r+s}{s}  \cdot\ldots \\
			\ldots \cdot  \binom{\frac{r-1}{2}}{m-1}\binom{\frac{s-1}{2}}{n-1} \binom{m+n-1}{m+n-1-\frac{i+j-r-s}{2}} \binom{i+j-r-s}{i-r}.
		\end{multlined}
	\end{align}
	Another representation of the coefficients can be found in \autoref{app:prop:coeff_expl_long}.
\end{theorem}

\begin{remark}\label{rem:grafakos}
	In \cite[App. B.1]{grafakos}, it is shown that for dimension $k\ge 1$, powers $m,n>k$, factors $p,q>0$ and vectors $\vec r, \vec s\in \bbR^k$, the following inequality holds
	\begin{align}
		\int_{\bbR^k} \frac{p^k}{\parens{1+p\abs{\vec x-\vec r}}^m} \frac{q^k}{\parens{1+q\abs{\vec x-\vec s}}^n} \dd \vec x \lesssim \frac{\min(p,q)^k}{\parens{1+\min(p,q)\abs{\vec r-\vec s}}^{\min(m,n)}}.
	\end{align}
	Applied to our context (using the fact that $1+x^2\ge \frac 12 (1+\abs{x})^2\ge \frac 12 (1+x^2)$, resp.~$k=1$), implies
	\begin{align}
		I_{m,n} \lesssim \frac{1}{a} \frac{1}{(b^2+d^2)^{\min(m,n)}} \frac{1}{c^{2m-1}} \frac{1}{d^{2n-2\min(m,n)}} + \frac{1}{(a^2b^2+c^2)^{\min(m,n)}} \frac{1}{c^{2m-2\min(m,n)}} \frac{1}{d^{2n-1}}.
	\end{align}
	Even though we know that in our case, we can take $m$ to be arbitrarily large (but fixed), this only yields
	\begin{align}\label{app:eq:grafakos}
		I_{m,n} \lesssim \frac{1}{a} \frac{1}{(b^2+d^2)^{n}} \frac{1}{c^{2m-1}} + \frac{1}{\parens{a^2b^2+c^2}^{n}} \frac{1}{c^{2(m-n)}} \frac{1}{d^{2n-1}}.
	\end{align}
	In one of the key estimates that we need for \autoref{th:approx} (see the proof of \autoref{prop:loc_space}), $c$ and $d$ will contain variables in other dimensions to be integrated over, and the fact that one term now has three factors renders an second application of \eqref{app:eq:grafakos} impossible in this case. We would need to accept considerable slack in the estimates (by dropping one factor in the second term of the right-hand side of \eqref{app:eq:grafakos}), which would make achieving sufficiently strong estimates much more difficult (if not impossible).
	
	Compare also with \autoref{app:cor:Imn_higher_dim}, where we are able to leverage \eqref{app:eq:Imn_est} into higher dimensions as well --- in our case obtaining a denominator that contains both $p$ and $q$ (in the notation of this remark), see \eqref{eq:Imn_higher_dim}.
\end{remark}

\begin{remark}\label{rem:Imn_est_decay_a}
	The estimate has one deficiency in terms of the behaviour of $a$ --- namely, that $I_{m,n}$ always decreases with increasing $a$ (albeit much slower than might be expected; the decay with $b,c,d$ is much more pronounced), whereas the first term in the estimates increases with $a$ until around
	\begin{align}
		a\sim \sqrt{\frac{n c^2}{b^2+d^2}} \quad \text{for \eqref{app:eq:Imn_est}}
		\quad \text{resp.} \quad
		a\sim \frac{ncd+\sqrt{n(b^2+d^2)c^2}}{b^2+d^2} \quad \text{for \eqref{app:eq:Imn_est_best}}
	\end{align}
	and only then starts to decrease with $a$.
	
	This is not avoidable, as the first term of \eqref{app:eq:Imn_est_best} actually appears as such (modulo a constant) in the explicit representation of $I_{m,n}$, but there, its growth is eliminated by the decay with $a$ of terms like the second one in \eqref{app:eq:Imn_est_best}, which have a higher weight in practice.
	
	Consequently, barring a more precise analysis of the constant's dependence on $m$ and $n$, the estimate can be made more efficient in some cases by directly estimating away $a\ge1$ within $I_{m,n}$ and then applying \eqref{app:eq:Imn_est_best} --- namely when
	\begin{align}
		1& \ll a \lesssim \frac{(b^2+(c+d)^2)^n}{(b^2+d^2)^n} 
	\end{align}
	if only the first term should be minimised, which is the case if $a\sim c \gg b,d$, for example. Considering both terms simultaneously, estimating $a\ge1$ is still beneficial in the following regime,
	\begin{align}
		1& \ll a \lesssim \frac{(b^2+(c+d)^2)^n}{(b^2+d^2)^n}\frac{(b^2+(c+d)^2)^{m}d^{2n-1}}{(b^2+(c+d)^2)^m d^{2n-1}+(b^2+(c+d)^2)^n c^{2m-1}}.\tag*{\qedhere}
	\end{align}
\end{remark}

\begin{proof}[Proof of \autoref{app:th:Imn_est}]
	The proof is split into several parts. First, we need to determine the partial fraction decomposition (PFD) of the integrand, which we do in \autoref{app:prop:pfd}. Since $m$ and $n$ are arbitrary, we will only be able to formulate a recursion at first. However, we can leverage this recursion into explicit generating functions for the terms appearing in the PFD, which we do in \autoref{app:prop:genfunc}. This machinery is necessary, unfortunately, since mere induction is hopelessly inadequate for the task at hand.
	
	With the help of these two tools, we are able to calculate the generating function \eqref{app:eq:genfunc_Imn_intro}
	which we prove in \autoref{app:prop:genfunc_Imn}. 
	
	In the form \eqref{app:eq:genfunc_Imn_intro}, we have already achieved (essentially) the crucial cancellation (compared to the PFD) that eliminates the ``bad'' factors from the denominator. However, what remains to be shown compared to \autoref{app:th:Imn_est} is that $c^{m,n}_{i,j}=0$ if $i\le 2(m-1) \land j\le 2(n-1)$. To gain explicit control over these coefficients, we first ``disassemble'' the function \eqref{app:eq:genfunc_Imn_intro} into its parts (by differentiation) in \autoref{app:prop:coeff_expl_long}, which yields another formula for $c^{m,n}_{i,j}$ (which is more complex, but without binomial coefficients of non-integers).
	
	Then, inserting the ``indicators'' we need, we put it back together to arrive at the formula \eqref{app:eq:genfunc_coeff} in \autoref{app:prop:genfunc_coeff}. Finally, we take apart \eqref{app:eq:genfunc_coeff} one last time in a different way that allows us to conclude that the required coefficients are actually zero in \autoref{app:prop:coeff_zero}. This will finish the proof. Finally, in \autoref{rem:conject_coeff_pos}, we mention the conjecture that, always, $c^{m,n}_{i,j}\ge0$, which, however, we have not (seriously) attempted to prove.
\end{proof}

Before we continue, we record an corollary of \autoref{app:th:Imn_est} for higher dimensions.

\begin{corollary}\label{app:cor:Imn_higher_dim}
	For $m,n\in\bbN$  and $c,d>0$, as well as vectors $\vec r,\vec s \in\bbR^k$ and invertible matrices $A,B$ such that $AB^{-1}$ is diagonalisable\footnote{The restriction that $AB^{-1}$ has to be diagonalisable is obviously artificial and can be removed in principle (although the formula would become much more complicated).}, we assume that two functions satisfy
	\begin{align}
		\abs{f(\vec x)}\lesssim \parens*{\abs*{A(\vec x+\vec r)}^2+c^2}^{-m}
		\quad \text{and} \quad
		\abs{g(\vec x)}\lesssim \parens*{\abs*{B(\vec x+\vec s)}^2+d^2}^{-n}.
	\end{align}
	Then, if $m,n> \ceil*{\frac{k}{2}}$, we have the following estimate for the convolution of $f$ and $g$,
	\begin{multline}
		\abs{[f*g](\vec t)}
		\lesssim\frac{1}{\abs{\det B}} \biggl( \frac{\norm[]{BA^{-1}}^{k}}{\parens*{|B(\vec r+\vec s +\vec t)|^2+d^2+\|BA^{-1}\|^2 c^2}^{n}} \frac{1}{c^{2m-k-1}} + \ldots \\
					\ldots+\frac{\norm[]{BA^{-1}}^{2m}}{\parens*{|B(\vec r+\vec s +\vec t)|^2+d^2+\|BA^{-1}\|^2 c^2}^{m}} \frac{1}{d^{2n-k-1}} \biggr)
	\end{multline}
	
	and --- in a dual way --- the same estimate holds after concurrently switching $\vec r\leftrightarrow\vec s$, $A\leftrightarrow B$, $c\leftrightarrow d$ and $m\leftrightarrow n$, i.e.~we can choose the minimum of the two.
	
	Specialising to $A=B=\bbI$ and $\vec r=\vec s =0$, we see that
	\begin{align}\label{eq:Imn_higher_dim}
		\abs{[f*g](\vec t)} \lesssim \frac{1}{\parens*{|\vec t|^2+c^2+d^2}^n}\frac{1}{c^{2m-k-1}} + \frac{1}{\parens*{|\vec t|^2+c^2+d^2}^m}\frac{1}{d^{2n-k-1}}.
	\end{align}
\end{corollary}

\begin{proof}
	We begin by inserting the definition, using the assumed estimates and transforming by $\vec y=B(\vec x+\vec s)$.
	\begin{align}
		\abs{[f*g](\vec t)}
		&= \abs[\bigg]{\int f(\vec t-\vec x) g(\vec x) \d \vec x}
		\lesssim \int \parens*{\abs*{A(\vec t-\vec x+\vec r)}^2+c^2}^{-m} \parens*{\abs*{B(\vec x+\vec s)}^2+d^2}^{-n} \d \vec x \label{app:eq:est_conv_def}\\
		&=\frac{1}{\abs{\det B}} \int \parens*{\abs*{A(\vec r+\vec t-B^{-1}\vec y +\vec s)}^2+c^2}^{-m} \parens*{\abs{\vec y}^2+d^2}^{-n} \d \vec y.
	\end{align}
	Decomposing $AB^{-1}=VDV^{-1}$ with $D$ being a diagonal matrix with eigenvalues sorted by descending absolute value and $V$ being the (unitary) matrix of corresponding eigenvectors, we set $\vec u:=V^{-1}A(\vec r+\vec s+ \vec t)$ and continue by transforming with $\vec z= V^{-1}\vec y$,
	\begin{align}
		\abs{[f*g](\vec t)}
		&\lesssim \frac{1}{\abs{\det B}} \int \parens*{\abs*{-VDV^{-1}\vec y+A(\vec r+\vec s+ \vec t)}^2+c^2}^{-m} \parens*{\abs{\vec y}^2+d^2}^{-n} \d \vec y\\
		&= \frac{1}{\abs{\det B}} \int \parens*{\abs*{D\vec z-\vec u}^2+c^2}^{-m} \parens*{\abs{\vec z}^2+d^2}^{-n} \d \vec z.
	\end{align}
	The problem we face now is that the above integral behaves ``elliptically'' in some sense --- and in a way we can't remove by suitable stretching --- because each component appears \emph{both} as $z_i$ and as $\lambda_i z_i$. With \eqref{app:eq:Imn_est} in mind, there are two ways out of this. On the one hand, we could shave off dimension after dimension (since the $z_i$ are decoupled, we can apply \eqref{app:eq:Imn_est} in each dimension sequentially), but this blows up the number of terms to (potentially) $2^k$, and we would ``lose'' at least half a power of either denominator in each step (or every second step, if one is careful).
	
	The second way --- which we will choose --- is to (effectively) make the matrix $D$ a multiple of the identity (thus removing the ``elliptic'' influences), by estimating it with its smallest eigenvalue (by magnitude) as follows below. In view of \autoref{rem:Imn_est_decay_a}, it is not unlikely that this might even be the more efficient estimate in many cases. A combination of the two methods is of course also possible, in fact, if $AB^{-1}$ is not diagonalisable, it is necessary to ``cut apart'' the Jordan blocks in the way described above.
	
	Each entry of $D\vec z-\vec u$ contributes a term $(\lambda_i z_i- u_i)^2=\abs{\lambda_i}^2\parens*{z_i- \frac{u_i}{\lambda_i}}^2$ to the absolute value, and we can estimate this from below (thus estimating the integral from above) by $|\lambda_d|^2(z_i-\frac{u_i}{\lambda_i})^2$. For notational ease, we set $a:=\abs{\lambda_d}$ as well as $\vec v:=D^{-1}\vec u$, and continue from above,
	\begin{align}
		\int \parens*{\abs*{D(\vec z-D^{-1}\vec u)}^2+c^2}^{-m} \parens*{\abs{\vec z}^2+d^2}^{-n} \d \vec z
		&\le \int \parens*{a^2\abs{\vec z-\vec v}^2+c^2}^{-m} \parens*{\abs{\vec z}^2+d^2}^{-n} \d \vec z.
	\end{align}
	Now we set $R:=R_{\vec v/ \abs{\vec v}}$ (compare \autoref{def:rot_Rs}) --- satisfying $R\vec v=\parens*{|\vec v|,0,\ldots}^\top$ --- and transform with $\vec w:=R\vec z$ using the invariance of the Euclidian norm under rotations,
	\begin{multline}
		\int \parens*{a^2\abs{\vec z-\vec v}^2+c^2}^{-m} \parens*{\abs{\vec z}^2+d^2}^{-n} \d \vec z\\
		=\int_{\bbR^k} \frac{1}{\parens*{a^2(w_1 - |\vec v|)^2 + a^2|\vec w'|^2 +c^2}^m} \frac{1}{\parens*{w_1^2+|\vec w'|^2+d^2}^n} \d \vec w,
	\end{multline}
	where $\vec w=\binom{w_1}{\vec w'}$, i.e.~$\vec w'$ represents the $k-1$ lower components of $\vec w$.
	
	We split off the integration in $w_1$ and apply \eqref{app:eq:Imn_est},
	\begin{align}
		\MoveEqLeft
		\int_{\bbR^{k-1}} \int_{-\infty}^{\infty } \frac{1}{\parens*{a^2(w_1 - |\vec v|)^2 + a^2|\vec w'|^2 +c^2}^m} \frac{1}{\parens*{w_1^2+|\vec w'|^2+d^2}^n} \d w_1 \d \vec w'\\
		&\!\begin{multlined}[t][\linewidth-\mathindent-2em-\multlinegap]
			\lesssim \int_{\bbR^{k-1}}  \frac{a^{2n-1}}{\parens*{a^2|\vec v|^2 + 2a^2|\vec w'|^2+c^2+a^2d^2}^n} \frac{1}{\parens*{a^2|\vec w'|^2+c^2}^{\frac{2m-1}{2}}} + \ldots\\
			\ldots + \frac{1}{\parens*{a^2|\vec v|^2 + 2a^2|\vec w'|^2+c^2+a^2d^2}^m} \frac{1}{\parens*{|\vec w'|^2+d^2}^{\frac{2n-1}{2}}} \d \vec w'
		\end{multlined}\\
		&\!\begin{multlined}[t][\linewidth-\mathindent-2em-\multlinegap]
			\le \int_{0}^\infty  \frac{a^{2n-k}}{\parens*{r^2+a^2|\vec v|^2+c^2+a^2d^2}^n} \frac{r^{k-1}}{\parens*{r^2+c^2}^{\frac{2m-1}{2}}} \d r + \ldots\\
			\ldots + \int_{0}^\infty \frac{1}{\parens*{a^2r^2+a^2|\vec v|^2+c^2+a^2d^2}^m} \frac{r^{k-1}}{\parens*{r^2+d^2}^{\frac{2n-1}{2}}} \d r
		\end{multlined}\\
		&\!\begin{multlined}[t][\linewidth-\mathindent-2em-\multlinegap]
			\le \frac{1}{2} \int_{-\infty}^\infty  \frac{a^{2n-k}}{\parens*{r^2+a^2|\vec v|^2+c^2+a^2d^2}^n} \frac{1}{\parens{r^2+c^2}^{m-\ceil{\frac{k}{2}}}} \frac{1}{(c^2)^{\ceil{\frac{k}{2}}-\frac{k}{2}}} \d r + \ldots\\
			\ldots + \frac{1}{2}\int_{-\infty}^\infty \frac{1}{\parens*{a^2r^2+a^2|\vec v|^2+c^2+a^2d^2}^m} \frac{1}{\parens{r^2+d^2}^{n-\ceil{\frac{k}{2}}}} \frac{1}{(d^2)^{\ceil{\frac{k}{2}}-\frac{k}{2}}} \d r,
		\end{multlined}
	\end{align}
	where we split the integrals and transformed the first term by $\vec x'=\frac{1}{a}\vec w'$ before changing to polar coordinates. We then extended the integral over $r$ to $-\infty$ in order to be able to apply \eqref{app:eq:Imn_est} once more,
	\begin{align}
		&\!\begin{multlined}[t][\linewidth-\mathindent-\multlinegap]
			\lesssim \frac{a^{2n-k}}{\parens*{a^2|\vec v|^2 +c^2+a^2d^2}^{m-\ceil{\frac k2}+n-\frac 12}}\frac{1}{(c^2)^{\ceil{\frac{k}{2}}-\frac{k}{2}}}  + \frac{a^{2n-k}}{\parens*{a^2|\vec v|^2 +c^2+a^2d^2}^{n}} \frac{1}{c^{2m-k-1}} + \ldots \\
			\ldots + \frac{a^{2n-2\ceil{\frac{k}{2}}-1}}{\parens*{a^2|\vec v|^2 +c^2+a^2d^2}^{n-\ceil{\frac k2}+m-\frac 12}} \frac{1}{(d^2)^{\ceil{\frac{k}{2}}-\frac{k}{2}}}  + \frac{1}{\parens*{a^2|\vec v|^2 +c^2+a^2d^2}^{m}} \frac{1}{d^{2n-k-1}}
		\end{multlined}\\
		&\lesssim \frac{a^{2n-k}}{\parens*{a^2|\vec v|^2 +c^2+a^2d^2}^{n}} \frac{1}{c^{2m-k-1}} + \frac{1}{\parens*{a^2|\vec v|^2 +c^2+a^2d^2}^{m}} \frac{1}{d^{2n-k-1}}, 
	\end{align}
	because, obviously, $a^2|\vec v|^2 +c^2+a^2d^2>c^2,a^2d^2$. Now, $a=\abs{\lambda_d}$, the smallest eigenvalue of $AB^{-1}$, corresponds to the inverse of the largest eigenvalue of $BA^{-1}$, which itself is equal to the matrix norm $\norm{BA^{-1}}$. Furthermore, $|D^{-1}\vec u|=|VD^{-1}\vec u|=|B(\vec r+\vec s +\vec t)|$. Putting everything together, we arrive at
	\begin{align}
		\abs{[f*g](\vec t)}
		&\!\begin{multlined}[t][\linewidth-\mathindent-\widthof{$\abs{[f*g](\vec t)}$}-\multlinegap]
			\lesssim \frac{1}{\abs{\det B}} \biggl( \frac{\norm[]{BA^{-1}}^{k}}{\parens*{|B(\vec r+\vec s +\vec t)|^2+d^2+\|BA^{-1}\|^2 c^2}^{n}} \frac{1}{c^{2m-k-1}} + \ldots \\
			\ldots+\frac{\norm[]{BA^{-1}}^{2m}}{\parens*{|B(\vec r+\vec s +\vec t)|^2+d^2+\|BA^{-1}\|^2 c^2}^{m}} \frac{1}{d^{2n-k-1}} \biggr).
		\end{multlined}
	\end{align}
	Depending on the quantities in question (but certainly in the case that $\norm{AB^{-1}}\ll\norm{BA^{-1}}$), we transform differently from \eqref{app:eq:est_conv_def},
	\begin{align}
		\abs{[f*g](\vec t)}
		&\lesssim \int \parens*{\abs*{A(\vec t-\vec x+\vec r)}^2+c^2}^{-m} \parens*{\abs*{B(\vec x+\vec s)}^2+d^2}^{-n} \d \vec x \\
		&=\frac{1}{\abs{\det A}} \int \parens*{\abs{\vec y}^2+c^2}^{-m} \parens*{\abs*{B(\vec t+\vec s-A^{-1}\vec y+\vec r}^2+d^2}^{-n} \d \vec y \\
		&=\frac{1}{\abs{\det A}} \int \parens*{\abs{\vec y}^2+c^2}^{-m} \parens*{\abs*{BA^{-1}\vec y - A(\vec r+ \vec s + \vec t)}^2+d^2}^{-n} \d \vec y.
	\end{align}
	Proceeding like before, this means that the convolution \emph{also} satisfies
	\begin{align}
		\abs{[f*g](\vec t)}
		&\!\begin{multlined}[t][\linewidth-\mathindent-\widthof{$\abs{[f*g](\vec t)}$}-\multlinegap]
			\lesssim \frac{1}{\abs{\det A}} \biggl( \frac{\norm[]{AB^{-1}}^{2n}}{\parens*{|A(\vec r+ \vec s + \vec t)|^2+c^2+\|AB^{-1}\|^2 d^2}^{n}} \frac{1}{c^{2m-k-1}} + \ldots \\
			\ldots+\frac{\norm[]{AB^{-1}}^{k}}{\parens*{|A(\vec r+ \vec s + \vec t)|^2+c^2+\|AB^{-1}\|^2 d^2}^{m}} \frac{1}{d^{2n-k-1}} \biggr).
		\end{multlined}
	\end{align}
	and we can choose the one that is smaller.
\end{proof}

\subsection{Some Basic Generating Function Theory}\label{ssec:genfunc}

The main idea of the approach of generating functions --- see e.g.~\cite{wilf} --- can be described as follows: take a sequence $g_n$ (recursively defined, for example) and calculate
\begin{align}
	G(x):=\ops{g_n}{x^n}:=\sum_{n\ge 0} g_n x^n,
\end{align}
where ``ops'' stands for \emph{ordinary power series} (as opposed to exponential power series, which we will not need).
If $G$ can be identified with a known function, $g_n$ can be recovered as
\begin{align}\label{app:eq:def_coeff_op}
	g_n=\frac{1}{n!}\Dn{n}[G(x)]{x}\biggr|_{x=0}=: \coeff{x^n} G(x),
\end{align}
which is often possible, even if the recursion for $g_n$ cannot be resolved by induction. The notation $[x^n]$ will henceforth denote the $n$\nth coefficient of $G$ with respect to $x$. Two properties follow immediately from the definition,
\begin{align}
	\coeff{x^n}(x^k G(x)) &= \coeff*{x^{n-k}} G(x), \label{app:eq:coeff_shift} \\
	\coeff{x^n}G(\beta x) &=\beta^n\coeff{x^n}G(x), \label{app:eq:coeff_factor}
\end{align}
where $\beta\in\bbR$, as a simple consequence of the chain rule.

The main tool to calculate $G$ are basic identities from the theory of power series, with the advantage that we can do all calculations purely formally at first, while ultimately, if the resulting $G$ turns out to be convergent in a ball of radius $R>0$ then all our formal calculations are actually justified analytically as well. This kind of freedom is especially useful if $g_n$ is itself a partial sum of the sequence $f_{n,k}$, since we can freely interchange the order of summation, i.e.
\begin{align}
	G(x):=\sum_{n\ge 0} \sum_{k=0}^n f_{n,k} x^n = \sum_{n\ge 0} \sum_{k\ge 0} \ind_{\{k\le n\}} f_{n,k} x^n = \sum_{k\ge 0} \sum_{n\ge k} f_{n,k} x^n = \sum_{k\ge 0} x^k \smash{\overbrace{\sum_{n\ge 0} f_{n-k,k} x^n}^{=F_k(x)}}.
\end{align}
If --- as will often be the case --- we can calculate $F_k(x)$ and consequently $G(x)$, we may then find $g_n=\sum_{k=0}^n f_{n,k}$ as $\coeff{x^n} G$.

We consider --- again only formally --- a function $F(x)=\ops{f_n}{x^n}$ generated by $\{f_n\}$, as well as another one generated by $\{g_n\}$, $G(x)=\ops{g_n}{x^n}$; then
\begin{align}
	\ops{f_{n+1}}{x^n}&=\frac{F(x)-F(0)}{x}, \label{app:eq:ops_shift}\\
	F(x)G(x)&=\ops[\bigg]{\sum_{k=0}^n f_n g_{n-k}}{x^n}. \label{app:eq:ops_prod}
\end{align}

Furthermore, the last essential ingredient is having the identification of as many power series as possible with known functions. But before we define the binomial coefficient for $\alpha\in\bbC$ (we will only need $\alpha\in\bbR$) and $n\in\bbN_0$,
\begin{align}\label{eq:binom_alpha}
	\binom{\alpha}{n}:=\frac{\alpha(\alpha-1)\cdot\ldots\cdot(\alpha-n+1)}{n!}= \frac{\Gamma(\alpha+1)}{\Gamma(n+1)\Gamma(\alpha-n+1)},
\end{align}
where the last equality holds in the limit $\alpha'\to\alpha$ if one of the $\Gamma$-terms has a singularity at $\alpha$.
By reversing the signs and order of the factors in the numerator, we see that the following identity holds
\begin{align}\label{app:eq:binom_alpha}
	\binom{\alpha}{n}=(-1)^n\binom{n-\alpha-1}{n}.
\end{align}
We only list the power series identities we will need (\cite[Sec 2.5]{wilf}):
\begin{alignat}{2}
	\frac{1}{1-x} & =\sum_{n\ge 0} x^n, && \label{app:eq:ops_geom} \\
	\frac{1}{(1-x)^{k+1}} & =\sum_{n\ge 0} \binom{n+k}{n} x^n, && k\in\bbN_0, \label{app:eq:ops_geom_k} \\
	(1+x)^\alpha & =\sum_{n\ge 0} \binom{\alpha}{n} x^n, &\quad& \alpha\in\bbR, \label{app:eq:ops_geom_alpha} \\
	\frac{1}{\sqrt{1-4x}} & =\sum_{n\ge 0}\binom{2n}{n} x^n, && \label{app:eq:ops_sqrt} \\
	\parens[\bigg]{\frac{1-\sqrt{1-4x}}{2x}}^k & =\sum_{n\ge 0} \frac{k}{n+k}\binom{2n+k-1}{n} x^n, && k\in\bbN. \label{app:eq:ops_sqrt_k}
\end{alignat}
Note that, due to \eqref{app:eq:binom_alpha}, \eqref{app:eq:ops_geom_alpha} generalises both \eqref{app:eq:ops_geom} and \eqref{app:eq:ops_geom_k}.

\subsection{Partial Fraction Decomposition}

Before we can think about the integration in $I_{m,n}$, we first need to figure out the partial fraction decomposition (PFD) of the integrand.

\begin{proposition}\label{app:prop:pfd}
	Using
	\begin{align}\label{app:eq:pfd_Delta}
		\Delta:=\delta_+\delta_-:=((c+ad)^2+a^2 b^2)((c-ad)^2+a^2 b^2)
	\end{align}
	it holds that
	\begin{align}
	\begin{split}\label{app:eq:pfd}
		P_{m,n}&:=\frac{1}{\parens{a^2(x-b)^2+c^2}^m} \frac{1}{\parens{x^2+d^2}^n} = \\
		&\phantom{:}= \sum_{k=1}^{m} \frac{a^{2n}}{\Delta^{n+k-1}} \frac{r^n_k + a^2b\,x\,s^n_k}{(a^2(x-b)^2+c^2)^{m-k+1}} + \sum_{\ell=1}^{n} \frac{a^{2(\ell-1)}}{\Delta^{m+\ell-1}} \frac{t^\ell_m +a^2b\,x\,u^\ell_m}{(x^2+d^2)^{n-\ell+1}},
	\end{split}
	\end{align}
	where, for $k,\ell\in\bbN$, the coefficients can be calculated recursively:
	\mathtoolsset{showonlyrefs=false}
	\begin{subequations}\label{app:eqs:pfd_recursion}
	\begin{align}
		r^\ell_k&=\sum_{k'=1}^k r^1_{k-k'+1}r^{\ell-1}_{k'} + a^2b^2 (\Delta u^1_{k-k'}- (a^2b^2+c^2) u^1_{k-k'+1}) u^{\ell-1}_{k'} \label{app:eq:pfd_recursion:r} \\
		t^\ell_k&=\sum_{k'=1}^k t^1_{k-k'+1}r^{\ell-1}_{k'} + a^4b^2d^2u^1_{k-k'+1} u^{\ell-1}_{k'} \label{app:eq:pfd_recursion:t} \\
		u^\ell_k&=\sum_{k'=1}^k  u^{1}_{k-k'+1} r^{\ell-1}_{k'}-t^1_{k-k'+1}u^{\ell-1}_{k'}  = -s^\ell_k \label{app:eq:pfd_recursion:u}
	\end{align}
	\end{subequations}
	The calculation first needs to resolve $\ell\to\ell-1\to\ldots\to 1$ back to initial values
	\begin{subequations}\label{app:eqs:pfd_init}
	\begin{align}
		r^1_k&=r^1_1 t^{1}_{k-1} + a^2b^2(a^2b^2+c^2) u^1_1 u^1_{k-1}, \label{app:eq:pfd_init:r}\\
		t^1_k&=t^1_1 t^{1}_{k-1} - a^4b^2d^2 u^1_1 u^1_{k-1}, \label{app:eq:pfd_init:t}\\
		u^1_k&=u^1_1 t^{1}_{k-1} + t^1_1 u^1_{k-1} = -s^1_k, \label{app:eq:pfd_init:u}
	\end{align}
	\end{subequations}
	which themselves can be resolved by recurring back (in $k$) to
	\begin{subequations}\label{app:eqs:pfd_init_init}
	\begin{align}
		r^1_1&=3a^2b^2+a^2d^2-c^2, & && r^1_0&=-1, \label{app:eq:pfd_init_init:r}\\
		t^1_1&=\phantom{3}a^2b^2-a^2d^2+c^2, &\text{resp.}&& t^1_0&=1, \label{app:eq:pfd_init_init:t}\\
		u^1_1&=2 = -s^1_1, & && u^1_0&=0=s^1_0. \label{app:eq:pfd_init_init:u}
	\end{align}
	\end{subequations}\mathtoolsset{showonlyrefs=true}%
	Furthermore, we have the following important relation between the coefficients
	\begin{align}\label{app:eq:r_plus_t_eq_u}
		r^\ell_k+t^\ell_k= 2a^2b^2 u^\ell_k.
	\end{align}
\end{proposition}

\begin{proof}
	First off, we calculate the PFD for the case $m,n=1$,
	\begin{align}
		\frac{1}{a^2(x-b)^2+c^2} \frac{1}{x^2+d^2} = \frac{a^{2}}{\Delta} \frac{3a^2 b^2+a^2 d^2-c^2 -2a^2 b\,x}{a^2(x-b)^2+c^2} + \frac{1}{\Delta} \frac{a^2 b^2-a^2 d^2+c^2 +2a^2 b\,x}{x^2+d^2}
	\end{align}
	from which we can read off the left half of \eqref{app:eqs:pfd_init_init}. Next, assuming \eqref{app:eq:pfd} for $m$ and $n=1$ we use induction in $m$,
	\begin{align}\label{app:eq:pfd_indm_hyp}
		P_{m+1,1}
		&= \sum_{k=1}^{m} \frac{a^{2}}{\Delta^{k}} \frac{r^1_k +a^2b\,x\,s^1_k}{(a^2 (x-b)^2+c^2)^{m-k+2}} +  \frac{1}{\Delta^{m}} \frac{t^1_m +a^2b\,x\,u^1_m}{(a^2y^2+c^2)(x^2+d^2)}.
	\end{align}
	We continue by splitting the second term with the help of \eqref{app:eqs:pfd_init_init},
	\begin{align}
	\MoveEqLeft
		\frac{1}{\Delta^{m}} \frac{t^1_m +a^2b\,x\,u^1_m}{(a^2(x-b)^2+c^2)(x^2+d^2)} \\
		&= \frac{a^{2}}{\Delta^{m+1}} \frac{(r^1_1 +a^2b\,x\,s^1_1)(t^1_m +a^2b\,x\,u^1_m)}{a^2 (x-b)^2+c^2} + \frac{1}{\Delta^{m+1}} \frac{(t^1_1 +a^2b\,x\,u^1_1)(t^1_m +a^2b\,x\,u^1_m)}{x^2+d^2}\\
		&\!\begin{multlined}[t][\linewidth-\mathindent-2em-\multlinegap]
			=\frac{a^{2}}{\Delta^{m+1}} \frac{r^1_1 t^1_m +a^2b\,x(r^1_1 u^1_m+s^1_1 t^1_m)+a^4 b^2 x^2 s^1_1 u^1_m}{a^2 (x-b)^2+c^2} +\ldots \\
			\ldots+ \frac{1}{\Delta^{m+1}} \frac{t^1_1 t^1_m +a^2b\,x(t^1_1 u^1_m+u^1_1 t^1_m)+a^4 b^2 x^2 u^1_1 u^1_m)}{x^2+d^2}
		\end{multlined}
	\end{align}
	Considering \eqref{app:eq:pfd_init_init:u}, we collect the $x^2$-terms in the numerators,
	\begin{align}
		\frac{2a^4 b^2 u^1_m}{\Delta^{m+1}} \parens[\bigg]{\frac{x^2}{x^2+d^2} - \frac{a^2 x^2}{a^2 (x-b)^2 + c^2}} =\frac{2a^4 b^2 u^1_m}{\Delta^{m+1}} \parens[\bigg]{1 - \frac{d^2}{x^2+d^2} - 1 + \frac{a^2b^2+c^2-2a^2b\,x}{a^2 (x-b)^2 + c^2}}.
	\end{align}
	Coming back to \eqref{app:eq:pfd_indm_hyp}, we set the abbreviation $f_1(x):= a^2 (x-b)^2 + c^2$ for the first denominator, and compute that
	\begin{multline}
		P_{m+1,1} = \sum_{k=1}^{m} \frac{a^{2n}}{\Delta^{k}} \frac{r^1_k +a^2b\,x\,s^1_k}{f_1(x)^{m-k+2}} +   \frac{\smash{\overbrace{t^1_1 t^1_m - a^4 b^2 d^2 u^1_1 u^1_m}^{t^1_{m+1}}}+ a^2 b\,x (\smash{\overbrace{t^1_1 u^1_m + u^1_1 t^1_m}^{u^1_{m+1}}})}{\Delta^{m+1}(x^2+d^2)}  + \ldots\\
		\ldots+\frac{a^{2}}{\Delta^{m+1}f_1(x)} \parens*{{\underbrace{r^1_1 t^1_m\! + a^2 b^2 (a^2b^2\!+c^2) u^1_1 u^1_m}_{r^1_{m+1}}}\!+ a^2 b\,x ({\underbrace{r^1_1 u^1_m\! + s^1_1 t^1_m-4a^2b^2 u^1_m}_{=-u^1_1 t^1_m-t^1_1 u^1_m=s^1_{m+1}}})},
	\end{multline}
	because $r^1_1-4a^2 b^2 =-t^1_1$ and $s^1_1=-u^1_1$. This coincides with the recurrences in \eqref{app:eqs:pfd_init}, as claimed.
	
	As claimed in \eqref{app:eq:r_plus_t_eq_u}, the identity $r^1_1+t^1_1= 4a^2 b^2$ can be extended (here first for $\ell=1$),
	\begin{align}
	\begin{split}\label{app:eq:r_plus_t_eq_u:one}
		r^1_k+t^1_k
		&=(r^1_1+t^1_1) t^1_{k-1} +2a^2 b^2 (a^2 b^2-a^2d^2+c^2) u^1_{k-1}\\
		&= 4a^2b^2 t^1_{k-1}+2a^2 b^2 t^1_1 u^1_{k-1}=2a^2b^2u^1_k,
	\end{split}
	\end{align}
	which will help cut short some computations below.
	
	Now we come to the induction in $n$. The case for arbitrary $m$ and $n=1$ has been proved above, which covers the base case. Under the induction hypothesis that \eqref{app:eq:pfd} holds for $m,n\in\bbN$, we have
	\begin{align}\label{app:eq:pfd_indn_hyp}
		P_{m,n+1}
		= \sum_{k=1}^{m} \frac{1}{\Delta^{n+k-1}} \frac{a^{2n}(r^n_k + a^2b\,x\,s^n_k)}{f_1(x)^{m-k+1}}\frac{1}{x^2+d^2} + \sum_{\ell=1}^{n} \frac{1}{\Delta^{m+\ell-1}} \frac{a^{2(\ell-1)}(t^\ell_m +a^2b\,x\,u^\ell_m)}{(x^2+d^2)^{n-\ell+2}}.
	\end{align}
	Clearly, we can now apply our previous knowledge (the case with $n=1$), i.e.
	\begin{multline}
		\frac{1}{(a^2(x-b)^2+c^2)^{m-k+1}}\frac{1}{x^2+d^2}\\
		=\sum_{k'=1}^{m-k+1} \frac{1}{\Delta^{k'}} \frac{a^{2}(r^1_{k'} +a^2b\,x\,s^1_{k'})}{(a^2 (x-b)^2+c^2)^{m-k+1-k'+1}} +  \frac{1}{\Delta^{m-k+1}} \frac{t^1_{m-k+1} +a^2b\,x\,u^1_{m-k+1}}{x^2+d^2}
	\end{multline}
	Inserting this into \eqref{app:eq:pfd_indn_hyp} yields
	\begin{align}
	\MoveEqLeft
		\sum_{k=1}^{m} \frac{a^{2n}}{\Delta^{n+k-1}} \frac{r^n_k + a^2b\,x\,s^n_k}{(a^2(x-b)^2+c^2)^{m-k+1}}\frac{1}{x^2+d^2} \\
		&\!\begin{multlined}[t][\linewidth-\mathindent-2em-\multlinegap]
			=\sum_{k=1}^{m}\sum_{k'=1}^{m-k+1} \frac{a^{2(n+1)}}{\Delta^{n+k+k'-1}} \frac{r^1_{k'} r^n_k + a^2b\,x(r^1_{k'} s^n_k+s^1_{k'}r^n_k)+a^4 b^2 x^2 s^1_{k'}s^n_k}{(a^2(x-b)^2+c^2)^{m-k-k'+2}} + \ldots \\
			\ldots + \sum_{k=1}^{m} \frac{a^{2n}}{\Delta^{m+n}} \frac{t^1_{m-k+1} r^n_k +a^2b\,x(t^1_{m-k+1}s^n_k + u^1_{m-k+1} r^n_k)+a^4 b^2 x^2 u^1_{m-k+1}s^n_k}{x^2+d^2}.
		\end{multlined}
	\end{align}
	We replace the undesired quadratic part $a^4b^2x^4$ depending on the denominator of the term we're dealing with,
	\begin{align}
		a^4 b^2 x^2&=a^2 b^2 \parens*{a^2 (x-b)^2 + c^2}\parens[\Big]{1-\frac{a^2b^2+c^2-2a^2b\,x}{a^2 (x-b)^2 +c^2}},\\
		a^4 b^2 x^2&=a^4 b^2 \parens*{x^2+ d^2}\parens[\Big]{1-\frac{d^2}{x^2+d^2}},
	\end{align}
	and change the summation indices as follows (demonstrated for only one term):
	\begin{align}
		\sum_{k=1}^{m}\sum_{k'=1}^{m-k+1} r^1_{k'} r^n_k
		&=\sum_{k=0}^{m-1}\sum_{k'=1}^{m-k} r^1_{k'} r^n_{k+1}
		=\sum_{k=0}^{m-1}\sum_{k'=0}^{m-k-1}  r^1_{k'+1} r^n_{k+1}
		=\sum_{k''=0}^{m-1}\sum_{k+k'=k''}  r^1_{k+1} r^n_{k'+1}\\
		&=\sum_{k''=0}^{m-1}\sum_{k'=0}^{k''} r^1_{k''-k'+1} r^n_{k'+1}
		=\sum_{k=1}^{m}\sum_{k'=0}^{k-1} r^1_{k-k'} r^n_{k'+1}
		=\sum_{k=1}^{m}\sum_{k'=1}^{k}  r^1_{k-k'+1} r^n_{k'}
	\end{align}
	This leads to
	\begin{align}
	\MoveEqLeft
		\sum_{k=1}^{m} \frac{a^{2n}}{\Delta^{n+k-1}} \frac{r^n_k + a^2b\,x\,s^n_k}{(a^2(x-b)^2+c^2)^{m-k+1}}\frac{1}{x^2+d^2} \\
		&= \sum_{k=1}^{m} \frac{a^{2(n+1)}}{\Delta^{n+k}} \biggl(\frac{\sum_{k'=1}^{k} r^1_{k-k'+1} r^n_{k'}-a^2 b^2 (a^2b^2+c^2) s^1_{k-k'+1}s^n_{k'}}{(a^2(x-b)^2+c^2)^{m-k+1}} + \ldots\\
		&\phantomrel\ldots + \frac{a^2b\,x\sum_{k'=1}^{k} s^1_{k-k'+1}r^n_{k'}+(r^1_{k-k'+1}-2a^2b^2 u^1_{k-k'+1})s^n_{k'}}{(a^2(x-b)^2+c^2)^{m-k+1}}  + \ldots \\
		&\!\begin{multlined}[t][\linewidth-\mathindent-2em-\multlinegap]
			\phantomrel\ldots + \frac{\sum_{k'=1}^{k} a^2 b^2 s^1_{k-k'+1}s^n_{k'}}{(a^2(x-b)^2+c^2)^{m-k}} \biggr)  + \sum_{k'=1}^{m} \frac{a^{2n}}{\Delta^{m+n}} \biggl(\frac{t^1_{m-k'+1} r^n_{k'} + a^4 b^2 d^2 s^1_{m-k'+1}s^n_{k'}}{x^2+d^2}+ \ldots \\
			\phantomrel\ldots + \frac{a^2b\,x(t^1_{m-k'+1}s^n_{k'} + u^1_{m-k'+1} r^n_{k'})}{x^2+d^2} +  a^4 b^2 u^1_{m-k'+1}s^n_{k'} \biggr).
		\end{multlined}
	\end{align}
	The terms with $x^2+d^2$ in the denominator exactly match the claimed $t$- and $u$-terms from \eqref{app:eqs:pfd_recursion} for $k=m$ and $\ell=n+1$, and thus we only have to deal with the remaining terms having powers of $a^2(x-b)^2+c^2$ in the denominator (as well as the very last term). To harmonise those powers, we perform an index shift for the third term (except the last summand, which will cancel with the very last term above), to arrive at
	\begin{align}
		\sum_{k=2}^{m} \frac{a^{2(n+1)}}{\Delta^{n+k-1}} \frac{\sum_{k'=1}^{k-1} a^2 b^2 s^1_{k-k'}s^n_{k'}}{(a^2(x-b)^2+c^2)^{m-k+1}} + \frac{a^{2(n+1)}}{\Delta^{m+n}} \sum_{k'=1}^{m} a^2 b^2 s^n_{k'} (\smash{\overbrace{s^1_{m-k'+1}+u^1_{m-k'+1}}^{=0}}).
	\end{align}
	Here, we can extend the summation to $\sum_{k=1}^m\sum_{k'=1}^k$, because all additional terms are zero.
	
	Furthermore, due to \eqref{app:eq:r_plus_t_eq_u:one}, we can simplify the second term (containing all terms with the factor $a^2 b\,x$) to
	\begin{align}
		\sum_{k=1}^{m} \sum_{k'=1}^{k} \frac{a^{2(n+1)}}{\Delta^{n+k}} \frac{a^2b\,x(t^1_{k-k'+1} u^n_{k'}-u^1_{k-k'+1}r^n_{k'})}{(a^2(x-b)^2+c^2)^{m-k+1}}.
	\end{align}
	Therefore,
	\begin{align}
		&\!\begin{multlined}[t][\linewidth-\mathindent-\multlinegap]
			P_{m,n+1}= \sum_{k=1}^{m} \frac{a^{2(n+1)}}{\Delta^{n+k}} \biggl( \frac{\sum_{k'=1}^{k} r^1_{k-k'+1} r^n_{k'}+a^2 b^2 u^n_{k'}(\Delta u^1_{k-k'} - (a^2b^2+c^2) u^1_{k-k'+1})}{(a^2(x-b)^2+c^2)^{m-k+1}} +\ldots \\
			\ldots+ \frac{a^2b\,x\sum_{k'=1}^{k} t^1_{k-k'+1} u^n_{k'}-u^1_{k-k'+1}r^n_{k'}}{(a^2(x-b)^2+c^2)^{m-k+1}} \biggr) + \sum_{\ell=1}^{n+1} \frac{a^{2(\ell-1)}}{\Delta^{m+\ell-1}} \frac{t^\ell_m +a^2b\,x\,u^\ell_m}{(x^2+d^2)^{n-\ell+2}},
		\end{multlined}
	\end{align}
	which now also matches the $r$- and $s$-terms in \eqref{app:eqs:pfd_recursion}, and in particular also $u^{n+1}_m=-s^{n+1}_m$.
	
	Finally, we want to extend \eqref{app:eq:r_plus_t_eq_u:one} to \eqref{app:eq:r_plus_t_eq_u}, which can be done as follows,
	\begin{align}
		r^\ell_k+t^\ell_k
		&=\sum_{k'=1}^k (r^1_{k-k'+1}+t^1_{k-k'+1})r^{\ell-1}_{k'} + a^2b^2 (\Delta u^1_{k-k'}- t^1_1 {\underbrace{u^1_{k-k'+1}}_{=\mathrlap{u^1_1 t^{1}_{k-k'} + t^1_1 u^1_{k-k'}}}}) u^{\ell-1}_{k'}\\[-0.245cm]
		&=\sum_{k'=1}^k 2a^2b^2 u^1_{k-k'+1} r^{\ell-1}_{k'} + a^2b^2 \parens*{\underbrace{(\overbrace{\Delta-(t^1_1)^2}^{=4a^4b^2d^2}) u^1_{k-k'}- 2t^1_1 t^1_{k-k'}}_{=-2t^1_{k-k'+1}}} u^{\ell-1}_{k'}=2a^2b^2 u^\ell_k.\tag*{\qedhere}
	\end{align}
\end{proof}

\begin{remark}
	We note that the recursion \eqref{app:eqs:pfd_init} is formulated in terms of $t$- and $u$-terms because we chose to induce in $m$ first --- inducting over $n$ first would lead to a different update rule \eqref{app:eqs:pfd_init}! 
\end{remark}

\subsection{Generating Functions for the PFD}

The following result about coupled recursions answers an obvious question to ask in this context, and as such, is certainly known already. However, since the solution is quite trivial, we have not searched for a reference, but rather proved it ourselves.

\begin{proposition}\label{app:prop:genfunc_2d_rec}
	For a recursion
	\begin{align}
		\binom{g_k}{h_k}=M \binom{g_{k-1}}{h_{k-1}} = M^{k-i} \binom{g_i}{h_i} \qquad \text{with} \qquad M=\begin{pmatrix}
			m_{1,1}& m_{1,2}\\ m_{2,1} & m_{2,2}
		\end{pmatrix},
	\end{align}
	with arbitrary $M$ and initial value $\binom{g_i}{h_i}$, where $i\ge 0$, the generating functions $G(y)=\ops*{g_k}{y^k}$ and $H(y)=\ops*{h_k}{y^k}$ are given by
	\begin{align}\label{app:eq:genfunc_2d_rec}
		G(y)=y^i \frac{(m_{1,2}h_i-m_{2,2}g_i)y+g_i}{\det(M) y^2 - \trace(M) y +1}, \qquad H(y)=y^i \frac{(m_{2,1}g_i- m_{1,1}h_i)y+h_i}{\det(M) y^2 - \trace(M) y +1}.
	\end{align}
\end{proposition}

\begin{proof}
	Let us begin by assuming that $M$ is diagonalisable. Then there exists an invertible matrix $Q$ such that
	\begin{align}
		M=Q\begin{pmatrix}
			\lambda_1 & 0 \\ 0 & \lambda_2
		\end{pmatrix} Q^{-1}.
	\end{align}
	Setting $\binom{g'_i}{h'_i}:=Q^{-1}\binom{g_i}{h_i}$, we can calculate --- because $g_k=0$ for $k<i$ --- that
	\begin{align}
		\binom{G(y)}{H(y)}
		&=\sum_{k\ge 0} \binom{g_{k}}{h_k} y^k 
		= \sum_{k\ge i} Q\begin{pmatrix}
			\lambda_1^{k-i} & 0 \\ 0 & \lambda_2^{k-i}
		\end{pmatrix} \binom{g'_i}{h'_i} y^k
		= y^i Q\sum_{k\ge 0} \binom{g'_i\lambda_1^k y^k}{h'_i\lambda_2^k y^k} = y^i Q \binom{\frac{g'_i}{1-\lambda_1 y}}{\frac{h'_i}{1-\lambda_2 y}}.
	\end{align}
	Since $\lambda_1 \lambda_2=\det(M)$ and $\lambda_1+\lambda_2=\trace(M)$, this implies
	\begin{align}
		\binom{G(y)}{H(y)}&
		= \frac{y^i}{\det(M) y^2 - \trace(M) y +1} Q \binom{g'_i(1-\lambda_2 y)}{h'_i(1-\lambda_1 y)}.
	\end{align}
	By explicitly calculating the eigenvalues and eigenvectors in terms of the $m_{i,j}$, one can check that (note the switched eigenvalues!)
	\begin{align}
		Q\begin{pmatrix}
			\lambda_2 & 0 \\ 0 & \lambda_1
		\end{pmatrix} Q^{-1} = \begin{pmatrix}
			\phantom{-}m_{2,2} & -m_{1,2} \\ -m_{2,1} & \phantom{-}m_{1,1}
		\end{pmatrix},
	\end{align}
	and therefore
	\begin{align}
		Q \binom{g'_i(1-\lambda_2 y)}{h'_i(1-\lambda_1 y)}=Q Q^{-1}\binom{g_i}{h_i} -
		y\begin{pmatrix}
			\phantom{-}m_{2,2} & -m_{1,2} \\ -m_{2,1} & \phantom{-}m_{1,1}
		\end{pmatrix} \binom{g_i}{h_i}
		= \binom{g_i+y(m_{1,2}h_i-m_{2,2}g_i)}{h_i+y(m_{2,1}g_i-m_{1,1}h_i)},
	\end{align}
	which finishes the proof for diagonalisable $M$.
	
	If $M$ is not diagonalisable, we can choose a sequence $M_n$ of diagonalisable matrices which converge to $M$ in an appropriate norm (e.g.~the Frobenius-norm), since, in the space $\bbC^{d\times d}$, non-diagonalisable matrices are a subset of the hypersurface where the discriminant of the characteristic polynomial is zero.
	
	For $G_n(y)$ and $H_n(y)$ corresponding to $M_n$, we see that they depend continuously on the coefficients $(m_{i,j})_n$ of $M_n$ --- taking an appropriate norm with respect to $y$ over a sufficiently small ball around the origin (such that the denominator is always positive) --- \emph{with no blow-up} for $M_n$ converging to non-diagonalisable $M$, and therefore $G_n(y)$ and $H_n(y)$ converge to the $G(y),\,H(y)$ formed with $M$. This finishes the proof.
\end{proof}

\begin{proposition}\label{app:prop:genfunc}
	Letting $q:=a^2b^2+a^2d^2-c^2$, we can calculate the generating functions for the coefficients of the PFD of $P_{m,n}$,
	\mathtoolsset{showonlyrefs=false}
	\begin{subequations}\label{app:eqs:pfd_genfunc}
	\begin{alignat}{3}
		R(y,z)&:=\sum_{\ell\ge 1} R^\ell(y) z^\ell&&:= \sum_{\ell\ge 1} \sum_{k\ge 1} r^\ell_k y^k z^\ell &&=yz\frac{\phantom{-}\Delta(y-z) +r^1_1}{\Delta (y-z)^2 - 2t^1_1 y -2q z+1},\hspace{-0.5cm} \label{app:eq:pfd_genfunc:r} \\
		T(y,z)&:=\sum_{\ell\ge 1} T^\ell(y) z^\ell&&:= \sum_{\ell\ge 1} \sum_{k\ge 1} t^\ell_k y^k z^\ell &&=yz\frac{-\Delta(y-z)+t^1_1}{\Delta (y-z)^2 - 2t^1_1 y -2q z+1},\hspace{-0.5cm} \label{app:eq:pfd_genfunc:t} \\
		U(y,z)&:=\sum_{\ell\ge 1} U^\ell(y) z^\ell&&:= \sum_{\ell\ge 1} \sum_{k\ge 1} u^\ell_k y^k z^\ell &&=yz\frac{2}{\Delta (y-z)^2 - 2t^1_1 y -2q z+1},\hspace{-0.5cm} \label{app:eq:pfd_genfunc:u}
	\end{alignat}
	\end{subequations}\mathtoolsset{showonlyrefs=true}%
	and $U(y,z)=-S(y,z)$. Then,
	\begin{align}
		r^\ell_k = \coeff*{y^k z^\ell} R(y,z), \qquad t^\ell_k = \coeff*{y^k z^\ell} T(y,z), \qquad u^\ell_k = \coeff*{y^k z^\ell} U(y,z).
	\end{align}
\end{proposition}

\begin{proof}
	We begin by considering $T^1$ and $U^1$ --- if we write \eqref{app:eq:pfd_init:t} and \eqref{app:eq:pfd_init:u} in matrix form, we see that
	\begin{align}
		\binom{t^1_{k}}{u^1_k}=M\binom{t^1_{k-1}}{u^1_{k-1}} = M^{k-1} \binom{t^1_1}{u^1_1}
		\qquad \text{with} \qquad
		M:= \begin{pmatrix}
			t^1_1 & -2a^4 b^2d^2 \\ 2 & t^1_1
		\end{pmatrix}.
	\end{align}
	Even though going back to initial values $\binom{r^1_0}{u^1_0}=\binom{1}{0}$ would be possible, we omit it on purpose, because the case for $m=0$ is not of interest to us (compare the definition of $I_{m,n}$).
	
	Fortunately, it turns out that we know the determinant of $M$ already from somewhere,
	\begin{align}
		\det(M)=\Delta, \qquad \trace(M)=2t^1_1,
	\end{align}
	and therefore, by \eqref{app:eq:genfunc_2d_rec}, we see that
	\begin{align}
		T^1(y)=y\frac{-\parens*{{\overbrace{4a^4b^2d^2+(t^1_1)^2}^{=\Delta}}} y+t^1_1}{\Delta y^2 -2t^1_1 y+1}, \qquad U^1(y)=\frac{2y}{\Delta y^2 -2t^1_1 y+1} 
	\end{align}
	Since \eqref{app:eq:r_plus_t_eq_u} implies that $R^1(y)+T^1(y)=2a^2b^2 U^1(y)$, we can easily calculate
	\begin{align}
		R^1(y)=y\frac{\Delta y+r^1_1}{\Delta y^2 -2t^1_1 y+1}.
	\end{align}
	
	Now we come to the full recursion, recalling that $R^{\ell}(y)=\sum_{k\ge 1} r^\ell_k y^k$ (and correspondingly for $U$, $T$) --- which effectively means that we have set $r^\ell_0$ to zero for the function $R^\ell$, resp.~$R$.
	Starting from \eqref{app:eq:pfd_recursion:r} and applying \eqref{app:eq:ops_shift} --- which is now easier due to the choice regarding $r^\ell_0$ --- it follows that
	\begin{align}
		\frac{R^\ell(y)}{y}
		&= \sum_{k\ge 0} r^\ell_{k+1} y^{k}
		=\sum_{k\ge 0}  y^{k} \sum_{k'=1}^{k+1} r^1_{k-k'+2}r^{\ell-1}_{k'} + a^2b^2 (\Delta u^1_{k-k'+1}- (a^2b^2+c^2) u^1_{k-k'+2}) u^{\ell-1}_{k'}.
	\end{align}
	We treat the terms separately, using \eqref{app:eq:ops_prod}, to compute
	\begin{align}
		\sum_{k\ge 0} y^k  \sum_{k'=1}^{k+1} r^1_{k-k'+2}r^{\ell-1}_{k'}
		&= \sum_{k\ge 0} y^k  \sum_{k'=0}^{k} r^1_{k-k'+1}r^{\ell-1}_{k'+1}
		= \parens[\bigg]{\sum_{k\ge 0}r^1_{k+1}y^k}\parens[\bigg]{\sum_{k\ge 0}r^{\ell-1}_{k+1}y^k}
		= \frac{R^1(y)}{y} \frac{R^{\ell-1}(y)}{y}.
	\end{align}
	Applying same procedure to the other terms, we see that, due to the choice to omit $r^\ell_0$ (and because $u^\ell_0$ is zero anyway),
	\begin{align}
		R^\ell(y)=\frac{\Delta y+r^1_1}{\Delta y^2 -2t^1_1 y+1} R^{\ell-1}(y)+2a^2b^2\frac{\Delta y-a^2b^2-c^2}{\Delta y^2 -2t^1_1 y+1} U^{\ell-1}(y).
	\end{align}
	is a well-defined recursion for $R^\ell$ with initial value $\ell=1$. For the $U$-recursion, we proceed in the same way to arrive at
	\begin{align}
		U^\ell(y)&=\frac{2}{\Delta y^2 -2t^1_1 y+1} R^{\ell-1}(y)+\frac{\Delta y - t^1_1}{\Delta y^2 -2t^1_1 y+1}U^{\ell-1}(y).
	\end{align}
	
	This allows us to again use \autoref{app:prop:genfunc_2d_rec} with $i=1$ and $z$ instead of $y$, by writing
	\begin{align}
		\binom{R^\ell(y)}{U^\ell(y)}&=\parens[\Bigg]{{\underbrace{\frac{1}{\Delta y^2 -2t^1_1 y+1}\begin{pmatrix} \Delta y+r^1_1 & 2a^2b^2(\Delta y-a^2b^2-c^2) \\ 2 & \Delta y - t^1_1 \end{pmatrix}}_{=:N}}}^{\ell-1} \binom{R^1(y)}{U^1(y)}.
	\end{align}
	Again, some factors in the determinant cancel,
	\begin{align}
		\det(N)=\frac{\Delta}{\Delta y^2 -2t^1_1 y+1}, \qquad \trace(N)=2\frac{\Delta y+\smash{\overbrace{a^2b^2+a^2d^2-c^2}^{=q}}}{\Delta y^2 -2t^1_1 y+1},
	\end{align}
	and we see that
	\begin{align}
		U(y,z)
		&=z\frac{\Delta y^2 -2t^1_1 y+1}{\Delta z^2-2(\Delta y+q)z+\Delta y^2 -2t^1_1 y+1} \frac{(\smash{\overbrace{N_{2,1}(\Delta y+r^1_1) -2N_{1,1}}^{=0}})yz+2y}{\Delta y^2 -2t^1_1 y+1}\\
		&=\frac{2yz}{\Delta (y-z)^2 - 2t^1_1 y -2q z+1}, \label{app:eq:pfd_genfunc:U} \\
		R(y,z)
		&=z\frac{(2 N_{1,2} -N_{2,2}(\Delta y+r^1_1)) yz+(\Delta y+r^1_1)y}{\Delta z^2-2(\Delta y+q)z+\Delta y^2 -2t^1_1 y+1}
		=yz\frac{\Delta(y-z) +r^1_1}{\Delta (y-z)^2 - 2t^1_1 y -2q z+1}, \label{app:eq:pfd_genfunc:R}
	\end{align}
	because $2 N_{1,2} -N_{2,2}(\Delta y+r^1_1)=-\Delta$.
	
	To conclude, we can use the fact that \eqref{app:eq:r_plus_t_eq_u} implies $R(y,z)+T(y,z)=2a^2b^2 U(y,z)$, which implies
	\begin{align}
		T(y,z)=yz\frac{-\Delta(y-z)+t^1_1}{\Delta (y-z)^2 - 2t^1_1 y -2q z+1}.\tag*{\qedhere}
	\end{align}
\end{proof}

\subsection{Generating Function for \texorpdfstring{$I_{m,n}$}{Imn}}

We begin by noting that for values at half-integers, the Gamma function can be evaluated as follows ($m\in\bbN_0$)
\begin{align}\label{app:eq:gamma_half_int}
	\Gamma\parens[\Big]{\frac 12+m}=\frac{(2m)!}{4^m m!}\sqrt{\pi}, \qquad \Gamma\parens[\Big]{\frac 12-m}=\frac{(-4)^m m!}{(2m)!}\sqrt{\pi},
\end{align}
which can be shown by induction using the functional equation $\Gamma(x+1)=x\Gamma(x)$ and the well-known fact that $\Gamma\parens*{\frac 12}=\sqrt{\pi}$.

This allows us to explicitly calculate the constants of the integral over one of the terms in the PFD, transforming first with $y=\frac ac (x-b)$, using the fact that the resulting function is even, and finally transforming with $t=\frac{1}{y^2+1}$,
\begin{align}
	\int_{-\infty}^{\infty} \frac{1}{(a^2 (x-b)^2+c^2)^m} \d x
	&= \frac{2}{c^{2m}}\int_{0}^{\infty} \frac{1}{(y^2+1)^m} \frac{c}{a} \d y= \frac{2}{a\,c^{2m-1}}\int_{1}^{0} t^m \frac{-1}{2\sqrt{t}^3\sqrt{1-t}} \d t \\
	&=\frac{1}{a\,c^{2m-1}}\int_{0}^{1} t^{m-\frac 12 -1} (1-t)^{\frac 12-1} \d t=\frac{1}{a\,c^{2m-1}} B\parens[\Big]{m-\frac 12,\frac 12} \\
	&= \frac{1}{a\,c^{2m-1}} \frac{\Gamma(m-\frac 12)\Gamma(\frac 12)}{\Gamma(m)} = \frac{\pi}{a\,c^{2m-1}} \frac{\binom{2(m-1)}{m-1}}{4^{m-1}}, \label{app:eq:single_factor_expl_int}
\end{align}
using the definition of the Beta function as well as \eqref{app:eq:gamma_half_int}.

\begin{proposition}\label{app:prop:genfunc_Imn}
	The generating function for $I_{m,n}$ is
	\begin{align}
		\sum_{m\ge 1}\sum_{n\ge 1} I_{m,n} y^m z^n 
		&=\frac{\pi yz}{\sqrt{c^2-\smash{y}}\sqrt{d^2-z}} \frac{\sqrt{c^2-\smash{y}}+a\sqrt{d^2- z}}{(\sqrt{c^2-\smash{y}}+a\sqrt{d^2- z})^2+a^2b^2}.
	\end{align}
	In particular, for $m,n\in\bbN$, we have
	\begin{align}\label{app:eq:Imn_as_coeff}
		I_{m,n} =\coeff*{y^{m-1} z^{n-1}} \frac{\pi}{\sqrt{c^2-\smash{y}}\sqrt{d^2-z}} \frac{\sqrt{c^2-\smash{y}}+a\sqrt{d^2- z}}{(\sqrt{c^2-\smash{y}}+a\sqrt{d^2- z})^2+a^2b^2}.
	\end{align}
\end{proposition}

\begin{proof}
	We begin by inserting \eqref{app:eq:pfd},
	\begin{align}
		I_{m,n}&=\int_{-\infty}^\infty P_{m,n} \d x
		=\int_{-\infty}^\infty \sum_{k=1}^{m} \frac{a^{2n}}{\Delta^{n+k-1}} \frac{r^n_k + a^2b\,x\,s^n_k}{(a^2(x-b)^2+c^2)^{m-k+1}} + \sum_{\ell=1}^{n} \frac{a^{2(\ell-1)}}{\Delta^{m+\ell-1}} \frac{t^\ell_m +a^2b\,x\,u^\ell_m}{(x^2+d^2)^{n-\ell+1}} \d x.
	\end{align}
	The $u$-terms vanish because the integral is an odd function, and for the $s$-terms --- after shifting $x$ accordingly for the first sum --- the same happens. However, we need to pay attention to the terms with $k=m$ and $\ell=n$, because they are not integrable in $x$ by themselves, but we will see that together, they are.
	\begin{multline}\label{app:eq:int_PFD_split}
		I_{m,n}
		=\int_{-\infty}^\infty \sum_{k=1}^{m-1} \frac{a^{2n}}{\Delta^{n+k-1}} \frac{r^n_k + a^2b^2s^n_k}{(a^2x^2+c^2)^{m-k+1}} + \sum_{\ell=1}^{n-1} \frac{a^{2(\ell-1)}}{\Delta^{m+\ell-1}} \frac{t^\ell_m}{(x^2+d^2)^{n-\ell+1}} + \ldots\\
		\ldots+ \frac{a^{2(n-1)}}{\Delta^{n+m-1}} \parens[\bigg]{\frac{r^n_m}{a^2x^2+c^2}+\frac{t^n_m}{x^2+d^2}+a^2 s^n_m \parens[\bigg]{\frac{a^2 b(x+b)}{a^2x^2+c^2}-\frac{b\,x}{x^2+d^2}} } \d x.
	\end{multline}
	We begin by treating the last term, first calculating the definite integral for arbitrary $0<X_-,X_+<\infty$, splitting and transforming appropriately,
	\begin{align}
	\MoveEqLeft
		\int_{-X_-}^{X_+} \frac{a^2 b(x+b)}{a^2x^2+c^2}-\frac{b\,x}{x^2+d^2} \d x\\
		&=\frac{b}{2}\int_{-X_-}^{X_+} \frac{2a^2 x}{a^2x^2+c^2} \d x + \int_{-X_-}^{X_+}\frac{a^2 b^2}{a^2x^2+c^2} \d x-\frac{b}{2}\int_{-X_-}^{X_+} \frac{2x}{x^2+d^2} \d x \\
		&=\frac{b}{2}\parens[\bigg]{\int_{0}^{a^2X_+^2}-\int_{0}^{a^2X_-^2}} \frac{1}{u+c^2} \d u + \frac{a\, b^2}{c}  \int_{-\frac{a}{c}X_-}^{\frac{a}{c}X_+}  \frac{1}{u^2+1} \d u-\frac{b}{2}\parens[\bigg]{\int_{0}^{X_+^2}-\int_{0}^{X_-^2}} \frac{1}{u+d^2} \d u\\
		&=\frac{b}{2}\log\parens[\Big]{\frac{a^2X_+^2+c^2}{a^2X_-^2+c^2}} + \frac{a\,b^2}{c}\parens[\Big]{\arctan\parens[\Big]{\frac{a}{c}X_+}-\arctan\parens[\Big]{-\frac{a}{c}X_-}\!\?} -\frac{b}{2}\log\parens[\Big]{\frac{X_+^2+d^2}{X_-^2+d^2}} \\
		&=\frac{b}{2}\parens[\Big]{\log\parens[\Big]{\frac{a^2X_+^2+c^2}{X_+^2+d^2}}-\log\parens[\Big]{\frac{a^2X_- ^2+c^2}{X_-^2+d^2}}\!\?} + \frac{a\,b^2}{c}\parens[\Big]{\arctan\parens[\Big]{\frac{a}{c}X_+ }+\arctan\parens[\Big]{\frac{a}{c}X_-}\!\?}.
	\end{align}
	We can see that the integral does not only exist as a principal value, but converges to a finite number in $X_-$ and $X_+$ separately, and therefore,
	\begin{align}
		\lim_{X_-\to\infty} \lim_{X_+\to\infty}
		\int_{-X_-}^{X_+} \!\frac{a^2 b(x+b)}{a^2x^2+c^2}-\frac{b\,x}{x^2+d^2} \d x
		= \frac{b}{2}\parens*{\log(a^2)-\log(a^2)} +\frac{a\,b^2}{c}\parens[\Big]{\frac{\pi}{2}+\frac{\pi}{2}} = \frac{\pi a\,b^2}{c}.
	\end{align}
	Continuing from \eqref{app:eq:int_PFD_split}, we apply \eqref{app:eq:single_factor_expl_int} and see that the terms we have just treated separately actually conform to the pattern exhibited by the other terms --- i.e.~we can extend the summation to $k=m$ and $\ell=n$ again.
	\begin{align}
		I_{m,n}
		&\!\begin{multlined}[t][\linewidth-\mathindent-\widthof{$I_{m,n}$}-\multlinegap]
			=\sum_{k=1}^{m}  \frac{1}{\Delta^{n+k-1}} \frac{a^{2n-1}}{c^{2(m-k)+1}} \frac{\pi}{4^{m-k}} \binom{2(m-k)}{m-k} \parens*{r^n_k + a^2b^2s^n_k} + \ldots\\
			\ldots+ \sum_{\ell=1}^{n} \frac{1}{\Delta^{m+\ell-1}} \frac{a^{2(\ell-1)}}{d^{2(n-\ell)+1}} \frac{\pi}{4^{n-\ell}} \binom{2(n-\ell)}{n-\ell} t^\ell_m
		\end{multlined}\\
		&\!\begin{multlined}[t][\linewidth-\mathindent-\widthof{$I_{m,n}$}-\multlinegap]
			=\frac{\pi}{\Delta^{m+n-1}c^{2m-1}d^{2n-1}}\biggl(\sum_{k=1}^{m}  \frac{\Delta^{m-k}  c^{2(k-1)} (ad)^{2n-1}\binom{2(m-k)}{m-k}}{4^{m-k}} \parens*{r^n_k + a^2b^2s^n_k} + \ldots\\
			\ldots+\sum_{\ell=1}^{n}  \frac{\Delta^{n-\ell} c^{2m-1}  (ad)^{2(\ell-1)}}{4^{n-\ell}} \binom{2(n-\ell)}{n-\ell} t^\ell_m \biggr)
		\end{multlined}\\
		&=: \frac{\pi}{\Delta^{m+n-1}c^{2m-1}d^{2n-1}}Q_{m,n}.
	\end{align}
	We deal with the two sum separately (reversing the order of $k$), calculating the generating function that will have the term we're looking for as the $m\nth$ coefficient.
	\begin{align}
	\MoveEqLeft
		\sum_{m\ge 0} \sum_{k=1}^{m} \Delta^{m-k} c^{2(k-1)} (ad)^{2n-1} \frac{1}{4^{m-k}}\binom{2(m-k)}{m-k} \parens*{r^n_k + a^2b^2s^n_k} y^m\\
		&=\sum_{m\ge 0} \sum_{k\ge 0} \ind_{\{k\le m-1\}} \Delta^{k} c^{2(m-k-1)} (ad)^{2n-1} \frac{1}{4^{k}}\binom{2k}{k} \parens*{r^n_{m-k} - a^2b^2u^n_{m-k}} y^m  \\
		&=(ad)^{2n-1} \sum_{k\ge 0} \parens[\Big]{\frac{\Delta}{4c^2}}^k \binom{2k}{k} c^{-2} \sum_{m\ge k+1} \parens*{r^n_{m-k} - a^2b^2u^n_{m-k}} (c^2y)^m \\
		&=(ad)^{2n-1} \sum_{k\ge 0} \parens[\Big]{\frac{\Delta}{4}}^k \binom{2k}{k} y^{k+1} \sum_{m\ge 0} \parens*{r^n_{m+1} - a^2b^2u^n_{m+1}} (c^2y)^m\\
		&=(ad)^{2n-1} \sum_{k\ge 0} \parens[\Big]{\frac{\Delta}{4}}^k \binom{2k}{k} y^{k+1} \parens[\Big]{\frac{R^n(c^2y)-a^2b^2U^n(c^2y)}{c^2y}}\\
		&=\coeff{z^n}\frac{(ad)^{2n-1}y}{\sqrt{1-\Delta y}}\parens[\Big]{\frac{R(c^2y,z)-a^2b^2U(c^2y,z)}{c^2y}}
	\end{align}
	where we used \eqref{app:eq:ops_shift} --- recalling that $R^n$ and $U^n$ have no constant term --- and the generating functions of \autoref{app:prop:genfunc}, as well as \eqref{app:eq:ops_sqrt} for the remaining sum in $k$.
	
	Using the same procedure, we see that
	\begin{align}
		\sum_{\ell=1}^{n} \Delta^{n-\ell} c^{2m-1} (ad)^{2(\ell-1)} \frac{1}{4^{n-\ell}} \binom{2(n-\ell)}{n-\ell} t^\ell_m =\coeff{y^m z^n}\frac{c^{2m-1}z}{\sqrt{1-\Delta z}} \frac{T(y,a^2d^2z)}{a^2d^2z},
	\end{align}
	since $t^0_m=0$ for all $m\in\bbN$.
	
	Then, for $m,n\ge 1$, using \eqref{app:eq:coeff_factor},
	\begin{align}
		Q_{m,n}
		&=\coeff*{y^m z^n} \frac{y}{ad\sqrt{1-\Delta y}} \frac{R(c^2y,a^2d^2z)-a^2b^2U(c^2y,a^2d^2z)}{c^2y} + \frac{z}{c\sqrt{1-\Delta z}} \frac{T(c^2y,a^2d^2z)}{a^2d^2z}\\
		&\!\begin{multlined}[t][\linewidth-\mathindent-\widthof{$Q_{m,n}$}-\multlinegap]
			=\coeff*{y^m z^n} \frac{adyz}{\sqrt{1-\Delta y}} \frac{\Delta(c^2y-a^2d^2z) +r^1_1-2a^2b^2}{\Delta (c^2y-a^2d^2z)^2 - 2t^1_1 c^2y -2q \,a^2d^2z+1} + \ldots\\
			\ldots+\frac{cyz}{\sqrt{1-\Delta z}} \frac{-\Delta(c^2y-a^2d^2z)+t^1_1}{\Delta (c^2y-a^2d^2z)^2 - 2t^1_1 c^2y -2q\, a^2d^2z+1}
		\end{multlined}\\
		&=\coeff*{y^{m-1} z^{n-1}} \frac{ad\sqrt{1-\Delta z}(\Delta(c^2y-a^2d^2z) +q)-c\sqrt{1-\Delta y}(\Delta(c^2y-a^2d^2z) -t^1_1)}{\sqrt{1-\Delta y}\sqrt{1-\Delta z}(\Delta (c^2y-a^2d^2z)^2 - 2c^2 t^1_1 y -2a^2d^2 q z+1)},
	\end{align}
	where we have used \eqref{app:eq:coeff_shift} for the last equality.
	
	The next crucial realisation is to (try to) factor everything in terms of $v(y):=\sqrt{1-\Delta y}$ and $w(z):=\sqrt{1-\Delta z}$, which works out almost miraculously well. First for the last factor of the denominator (multiplied by $\Delta$),
	\begin{align}
	\MoveEqLeft
		\Delta\parens*{\Delta (c^2y-a^2d^2z)^2 - 2c^2 t^1_1 y -2a^2d^2q z+1}\\
		&=\parens*{c^2(1-v^2)-a^2d^2(1-w^2)}^2-2c^2 t^1_1(1-v^2)-2a^2d^2 q(1-w^2)+\Delta\\
		&\!\begin{multlined}[t][\linewidth-\mathindent-2em-\multlinegap]
			=c^4v^4-2a^2c^2d^2v^2w^2+ a^4d^4w^4+ 2c^2v^2(\underbrace{t^1_1+a^2d^2-c^2}_{a^2b^2}) +\ldots \\[-0.245cm]
			\ldots+2a^2d^2w^2({\overbrace{q-a^2d^2+c^2}^{\cramped{a^2b^2}}})+\smash{\overbrace{c^4-2a^2c^2d^2 +a^4d^4 -2c^2t^1_1-2a^2d^2q+\Delta}^{\cramped{a^4b^4}}}
		\end{multlined}\\
		&=(c^2v^2+a^2d^2w^2+a^2b^2)^2-4a^2c^2d^2v^2w^2\\
		&=(c^2v^2+a^2d^2w^2+a^2b^2+2acdvw)(c^2v^2+a^2d^2w^2+a^2b^2-2acdvw)\\
		&=\parens*{(cv+adw)^2+a^2b^2}\parens*{(cv-adw)^2+a^2b^2}
	\end{align}
	and then, in similar fashion for the numerator
	\begin{align}
		adw(\Delta(c^2y-a^2d^2z) +q)-cv(\Delta(c^2y-a^2d^2z) -t^1_1)
		= (cv+adw)\parens*{(cv-adw)^2+a^2b^2},
	\end{align}
	which finally achieves the cancellation that will turn out to be the key to get rid of the factor $\delta_-$ in the denominator of the PFD. Taken together, this yields
	\begin{align}
		Q_{m,n}=\coeff*{y^{m-1} z^{n-1}} \frac{\Delta}{v(y)w(z)}\frac{cv(y)+adw(z)}{(cv(y)+adw(z))^2+a^2b^2},
	\end{align}
	
	Finally, recalling $I_{m,n}=  \pi Q_{m,n} / (\Delta^{m+n-1}c^{2m-1}d^{2n-1})$, we apply \eqref{app:eq:coeff_factor} to arrive at
	\begin{align}
		I_{m,n} &= \coeff*{y^{m-1} z^{n-1}}  \frac{\pi}{c^{2m-1}d^{2n-1}} \frac{1}{\sqrt{1-\smash{y}}\sqrt{1-z}}\frac{c\sqrt{1-\smash{y}}+ad\sqrt{1- z}}{(c\sqrt{1-\smash{y}}+ad\sqrt{1- z})^2+a^2b^2},
	\end{align}
	and once more using \eqref{app:eq:coeff_factor}, to yield the representation \eqref{app:eq:Imn_as_coeff}. This proves the claimed generating function for $I_{m,n}$.
\end{proof}

\subsection{Determining the Coefficients}

Compared to \autoref{app:th:Imn_est}, \eqref{app:eq:Imn_as_coeff} already has the right structure, but what remains to be proved at this stage is that the coefficients $c^{m,n}_{i,j}$ are zero when $i\le 2(m-1) \land j\le2(n-1)$. Showing this will concern us for the remainder of this section.

The generalised chain rule is expressed through \emph{Fa\`{a} di Bruno's formula}
\begin{align}\label{app:eq:faa_di_bruno}
	\Dn{n}{x} f(g(x))=\sum_{k=1}^n f^{(k)}(g(x)) B_{n,k}\parens*{g'(x),g''(x),\ldots,g^{(n-k+1)}(x)},
\end{align}
where
\begin{align}
	B_{n,k}(y_1,\ldots,y_{n-k+1})=\sum_{\substack{\sum_i j_i=k\\ \sum_i i \cdot j_i=n}} \frac{n!}{j_1!j_2!\cdot \ldots \cdot j_{n-k+1}!} \parens[\Big]{\frac{y_1}{1!}}^{j_1}\parens[\Big]{\frac{y_2}{2!}}^{j_2}\cdot \ldots \cdot\parens[\Big]{\frac{y_{n-k+1}}{(n-k+1)!}}^{j_{n-k+1}}
\end{align}
are the so-called \emph{Bell polynomials}, which satisfy the following identity for their generating function,
\begin{align}\label{app:eq:genfunc_bell}
	\sum_{n=k}^\infty B_{n,k}(y_1,y_2,\ldots,y_{n-k+1}) \frac{t^n}{n!}=\frac{1}{k!}\parens[\Big]{\sum_{m=1}^\infty y_m \frac{t^m}{m!}}^k.
\end{align}

\begin{lemma}
	For $1\le k\le n$ and $c>0$, we have
	\begin{align}\label{app:eq:bell_sqrt}
		B_{n,k}\parens[\Big]{\D{y}\sqrt{c^2-y},\ldots,\Dn{n-k+1}{y}\sqrt{c^2-y}}\Bigr|_{y=0} = (-2c)^{k-2n} \frac{(2n-k-1)!}{(k-1)!(n-k)!}.
	\end{align}
\end{lemma}

\begin{proof}
	Using \eqref{app:eq:ops_geom_alpha} we see that
	\begin{align}
		\coeff{y^m}\sqrt{c^2-y}=\coeff{y^m} c\sqrt{1-\frac{y}{c^2}}=\binom{\frac 12}{m}(-1)^m c^{1-2m} 
	\end{align}
	Inserting this into \eqref{app:eq:genfunc_bell}, we obtain by virtue of \eqref{app:eq:ops_shift} and \eqref{app:eq:ops_sqrt_k} that
	\begin{align}
		B_{n,k}\parens[\Big]{\D{y}\sqrt{c^2-y},\ldots}\Bigr|_{y=0} 
		&= \coeff{t^n} \frac{n!}{k!} c^k \parens*{\sqrt{1-t/c^2}-1}^k = \coeff{t^n} \frac{n!}{k!} c^{k}  \parens[\Big]{\frac{1-\sqrt{1-t/c^2}}{t/(2c^2)}}^k \parens[\Big]{\frac{-t}{2c^2}}^k\\
		&= \coeff*{t^{n-k}} \frac{n!}{k!} (-2c)^{-k}  \parens[\Big]{\frac{1-\sqrt{1-t/c^2}}{t/(2c^2)}}^k \\
		&= \frac{n!}{k!} (-2c)^{-k} \frac{k}{n}\binom{2n-k-1}{n-k} (4c^2)^{k-n}, 
	\end{align}
	which yields the desired result after rearranging.
\end{proof}

\begin{proposition}\label{app:prop:coeff_expl_long}
	By deriving \eqref{app:eq:Imn_as_coeff} $m-1$ times in $y$ and $n-1$ times in $z$, the coefficients in \eqref{app:eq:Imn_est_best} can be calculated as follows
	\begin{align}
		c^{m,n}_{i,j}
		&\!\begin{multlined}[t][\linewidth-\mathindent-\widthof{$c^{m,n}_{i,j}$}-\multlinegap]
			= \frac{\delta_{\{i+j\equiv1\bmod{2}\}}}{4^{m+n-2}}  \sum_{k=1}^{m-1}\sum_{\ell=1}^{n-1}\sum_{p=0}^{\ell}\sum_{q=0}^{k}\sum_{r=0}^{k-q}\sum_{s,t\ge 0} \frac{2^k k}{m-1} \binom{2m-k-3}{\phantom{2}m-k-1} \frac{2^\ell \ell}{n-1} \binom{2n-\ell-3}{\phantom{2}n-\ell-1}  \cdot\ldots\\
			\shoveright{\ldots\cdot \binom{\ell+1}{p} \binom{\ell-p}{q} \binom{\ell+r}{r} \binom{m+n-2-\ell-r}{s} \binom{m+n-2-p-q-r-s}{t} \cdot\ldots} \\
			\ldots\cdot \binom{m+n-1}{i-p-s} \binom{m+n-1-i+p+s}{j-q-r-t} (-1)^{\frac{3(i+j)+1}{2}+p+s+r+t}.
		\end{multlined}
	\end{align}
	Furthermore, $c^{m,n}_{i,j}$ is zero for $i+j>2(m+n)-3$.
\end{proposition}

\begin{proof}
	Letting $v(y):=\sqrt{c^2-y}$ and $w(z):=\sqrt{d^2-z}$ (which have the same behaviour but different constants to their namesakes above), we begin by rewriting \eqref{app:eq:Imn_as_coeff},
	\begin{align}
		I_{m,n} &=\coeff*{y^{m-1} z^{n-1}} \frac{\pi}{vw}\parens{\frac{v+aw-\ii ab}{(v+aw+\ii ab)(v+aw-\ii ab)} + \frac{\ii ab}{(v+aw)^2+a^2b^2}}\\
		&=\Re \coeff*{y^{m-1} z^{n-1}}  \frac{\pi}{v(y)w(z)}\frac{1}{v(y)+aw(z)+\ii ab}, \label{app:eq:Imn_simplif_frac}
	\end{align}
	since, clearly, $I_{m,n}$ is always real (and the second summand always imaginary).
	
	For now, we assume that both $n,m\ge2$, and apply \eqref{app:eq:faa_di_bruno} to derive $n-1$ times in $z$, arriving at
	\begin{align}
		I_{m,n}&=\Re\coeff*{y^{m-1}} \frac{\pi}{v(y)} \frac{1}{(n-1)!}\sum_{\ell=1}^{n-1} B_{n-1,\ell}\parens[\Big]{\D{z}w(z),\ldots}\Bigr|_{z=0} \Dn{\ell}{w} \frac{1}{w}\frac{1}{v(y)+aw+\ii ab}\Bigr|_{w=d}
	\end{align}
	Using the PFD in $w$, we see that
	\begin{align}
		\Dn{\ell}{w} \frac{1}{w}\frac{1}{v(y)+aw+\ii ab}\Bigr|_{w=d} 
		&= \Dn{\ell}{w} \frac{1}{v(y)+\ii ab}\parens[\Big]{\frac{1}{w}-\frac{a}{v(y)+aw+\ii ab}}\Bigr|_{w=d} \\
		&=\frac{\ell!(-1)^\ell}{v(y)+\ii ab}\parens[\Big]{\frac{1}{w^{\ell+1}}-\frac{a^{\ell+1}}{(v(y)+aw+\ii ab)^{\ell+1}}}\Bigr|_{w=d} \\
		&=\frac{\ell!(-1)^\ell}{v(y)+\ii ab}\frac{1}{d^{\ell+1}}\parens[\Big]{1-\frac{(ad)^{\ell+1}}{(v(y)+ad+\ii ab)^{\ell+1}}}.
	\end{align}
	Together with \eqref{app:eq:bell_sqrt}, we can continue manipulating $I_{m,n}$,
	\begin{align}
	I_{m,n}&\!\begin{multlined}[t][\linewidth-\mathindent-\widthof{$I_{m,n}$}-\multlinegap]
		= \Re\coeff*{y^{m-1}} \frac{\pi}{v(y)}\frac{1}{v(y)+\ii ab} \frac{1}{(n-1)!}\cdot\ldots \\
		\ldots\cdot\sum_{\ell=1}^{n-1} \frac{(-2d)^{\ell-2n+2}(2n-\ell-3)!}{(\ell-1)!(n-\ell-1)!} \frac{\ell!(-1)^\ell}{d^{\ell+1}}\parens[\Big]{1-\frac{(ad)^{\ell+1}}{(v(y)+ad+\ii ab)^{\ell+1}}}
	\end{multlined}\\
	&= \Re\coeff*{y^{m-1}} \frac{\pi}{v(y)}\frac{1}{v(y)+\ii ab} \sum_{\ell=1}^{n-1} \frac{2^{\ell+1}}{(2d)^{2n-1}} \frac{\ell}{n-1} \binom{2n-\ell-3}{\phantom{2}n-\ell-1} \parens[\Big]{1-\frac{(ad)^{\ell+1}}{(v(y)+ad+\ii ab)^{\ell+1}}}.
	\end{align}
	Using
	\begin{multline}
		1-\frac{(ad)^{\ell+1}}{(v(y)+ad+\ii ab)^{\ell+1}} \\
		=\frac{1}{(v(y)+ad+\ii ab)^{\ell+1}}\parens[\bigg]{\sum_{p=0}^{\ell+1}\binom{\ell+1}{p}(ad)^p (v(y)+\ii ab)^{\ell+1-p} -(ad)^{\ell+1}}
	\end{multline}
	this collapses further to
	\begin{align}
		I_{m,n}&=\Re\coeff*{y^{m-1}} \frac{\pi}{v(y)} \sum_{\ell=1}^{n-1} \sum_{p=0}^{\ell}  \frac{2^{\ell+1}}{(2d)^{2n-1}} \frac{\ell}{n-1} \binom{2n-\ell-3}{\phantom{2}n-\ell-1} \binom{\ell+1}{p} \frac{(ad)^p (v(y)+\ii ab)^{\ell-p}}{(v(y)+ad+\ii ab)^{\ell+1}}.
	\end{align}
	Applying the product rule, we observe first that
	\begin{align}
		\Dn{j}{v}\frac{1}{v(v+ad+\ii ab)^{\ell+1}}
		&=\sum_{r=0}^{j}\binom{j}{r}(-1)^j \frac{r!(\ell+j-r)!}{\ell!} \frac{1}{v^{r+1}(v+ad+\ii ab)^{\ell+1+j-r}}\\
		&=j!(-1)^j\sum_{r=0}^{j}\binom{\ell+r}{r}\frac{1}{v^{j-r+1}(v+ad+\ii ab)^{\ell+1+r}},
	\end{align}
	where we reversed the order of summation for the last equation. Consequently, with $j=k-q$, once more due to the product rule, $I_{m,n}$ is equal to
	\begin{align}
	&\!\begin{multlined}[t][\linewidth-\mathindent-\multlinegap]
		\phantomrel \Re \frac{\pi}{(m-1)!} \sum_{\ell=1}^{n-1} \sum_{p=0}^{\ell}  \frac{2^{\ell+1}}{(2d)^{2n-1}} \frac{\ell}{n-1} \binom{2n-\ell-3}{\phantom{2}n-\ell-1} \! \binom{\ell+1}{p} \! \sum_{k=1}^{m-1} \! B_{m-1,k}\parens[\Big]{\D{y}v(y),\ldots}\Bigr|_{y=0} \! \cdot\ldots\hspace{-2cm}\\
		\ldots\cdot\sum_{q=0}^{\mathclap{\min(k,\ell-p)}} \hspace{0.3cm}\binom{k}{q}\frac{(\ell-p)!}{(\ell-p-q)!} (c+\ii ab)^{\ell-p-q}  \sum_{r=0}^{k-q} \binom{\ell+r}{r}\frac{(ad)^p (k-q)!(-1)^{k-q}}{c^{k-q-r+1}(c+ad+\ii ab)^{\ell+1+r}}
	\end{multlined}\\
	&\!\begin{multlined}[t][\linewidth-\mathindent-\multlinegap]
		=\Re \pi \sum_{k=1}^{m-1}\sum_{\ell=1}^{n-1} \frac{2^{k+1}}{(2c)^{2m-1}} \frac{k}{m-1} \binom{2m-k-3}{\phantom{2}m-k-1} \frac{2^{\ell+1}}{(2d)^{2n-1}} \frac{\ell}{n-1} \binom{2n-\ell-3}{\phantom{2}n-\ell-1} \cdot\ldots\\
		\ldots\cdot\sum_{p=0}^{\ell}\sum_{q=0}^{k} \sum_{r=0}^{k-q} \binom{\ell+1}{p} \binom{\ell-p}{q} \binom{\ell+r}{r} \frac{(-1)^q  c^{q+r} (ad)^p (c+\ii ab)^{\ell-p-q}}{(c+ad+\ii ab)^{\ell+r+1}}.
	\end{multlined}
	\end{align}
	To find the respective powers of $c$ and $ad$, we pull out the maximal power of the denominator --- multiplied by its conjugate to make it real; recalling $\delta_+=(c+ad)^2+a^2b^2=(c+ad+\ii ab)(c+ad-\ii ab)$ --- to find
	\begin{multline}\label{eq:Imn_coeff_unified_denom}
		I_{m,n}
		=\frac{4\pi}{(4\delta_+)^{m+n-1}c^{2m-1}d^{2n-1}} \Re\!\!\sum_{k,\ell,p,q,r} \!\frac{2^k k}{m-1} \binom{2m-k-3}{\phantom{2}m-k-1} \frac{2^\ell \ell}{n-1} \binom{2n-\ell-3}{\phantom{2}n-\ell-1} \! \binom{\ldots}{p,q,r}  \cdot\ldots\\
		\ldots\cdot(-1)^q c^{q+r} (ad)^p (c+\ii ab)^{\ell-p-q}(c+ad+\ii ab)^{m+n-2-\ell-r} (c+ad-\ii ab)^{m+n-1},
	\end{multline}
	where we will denote
	\begin{align}
		\binom{\ldots}{k,\ell,m,n,p,q,r}:=\frac{2^k k}{m-1} \binom{2m-k-3}{\phantom{2}m-k-1} \frac{2^\ell \ell}{n-1} \binom{2n-\ell-3}{\phantom{2}n-\ell-1} \binom{\ell+1}{p} \binom{\ell-p}{q} \binom{\ell+r}{r}
	\end{align}
	for brevity, and will often collect consecutive letters with a hyphen, e.g.~$k-n$ for $k,\ell,m,n$. As the penultimate step, we expand the last line of \eqref{eq:Imn_coeff_unified_denom} to
	\begin{multline}
		\sum_{s\ge 0} \sum_{t\ge 0} \sum_{g\ge 0} \sum_{h\ge 0} \binom{m+n-2-\ell-r}{s} \binom{m+n-2-p-q-r-s}{t}  \binom{m+n-1}{g} \cdot\ldots \\
		\ldots\cdot \binom{m+n-1-g}{h} (-1)^{g+h+m+n+q+1} c^{h+q+r+t} (ad)^{g+p+s} (\ii ab)^{2(m+n)-3-g-h-p-q-r-s-t}.
	\end{multline}
	Finally, we collect like powers of $c$ and $ad$ as powers of $i=h+q+r+t$ and $j=g+p+s$, respectively,
	\begin{align}
	I_{m,n}
	&\!\begin{multlined}[t][\linewidth-\mathindent-\widthof{$I_{m,n}$}-\multlinegap]\label{eq:Imn_repr_coeff}
		=\phantom{:}\frac{\pi}{\delta_+^{m+n-1}c^{2m-1}d^{2n-1}} \Re\!\! \sum_{i,j=0}^{2(m+n)-3} \!\! \sum_{k,\ell,p,q,r,s,t} \binom{\ldots}{k,\ell,m,n,p,q,r,s,t} \binom{m+n-1}{j-p-s} \cdot\ldots\\
		\ldots\cdot \binom{m+n-1-j+p+s}{i-q-r-t} \frac{(-1)^{i+j+p+r+s+t+m+n+1}}{4^{m+n-2}} c^i (ad)^j (\ii ab)^{2(m+n)-3-i-j}
	\end{multlined}\\
	&\!\begin{multlined}[t][\linewidth-\mathindent-\widthof{$I_{m,n}$}-\multlinegap]
		=\phantom{:}\frac{\pi}{\delta_+^{m+n-1}c^{2m-1}d^{2n-1}} \Re\!\! \sum_{i,j=0}^{2(m+n)-3} \!\! \sum_{k,\ell,p,q,r,s,t} \binom{\ldots}{k,\ell,m,n,p,q,r,s,t} \binom{m+n-1}{j-p-s} \cdot\ldots\\
		\ldots\cdot \binom{m+n-1-j+p+s}{i-q-r-t} \frac{\ii^{-i-j+1}(-1)^{p+r+s+t}}{4^{m+n-2}} c^i (ad)^j (ab)^{2(m+n)-3-i-j}
	\end{multlined}\\
	&=:\frac{\pi}{\delta_+^{m+n-1}c^{2m-1}d^{2n-1}} \sum_{i,j=0}^{2(m+n)-3} c^{m,n}_{i,j} c^i (ad)^j (ab)^{2(m+n)-3-i-j}. 
	\end{align}
	Observe that due to the fact that we only consider the real parts of the sum, $i+j$ has to be an odd number, so that the power of $\ii$ is even.
	
	The claim that $i+j$ has to be less than $2(m+n)-3$ follows easily from the binomial coefficient in $i$, because it implies that (estimating $t$ by its maximal value, which can be read off from the corresponding binomial coefficient)
	\begin{multline}
		i-q-r-t\le m+n-1-j+p+s \\
		\Longrightarrow \quad
		i+j\le m+n-1+p+q+r+s+t\le 2(m+n)-3.\tag*{\qedhere}
	\end{multline}
\end{proof}

Now, after having disassembled $I_{m,n}$ into its parts, we are able to insert the indicators we later want to look for --- meaning we multiply by $v^i w^j y^m z^n$ --- and try to reassemble the whole thing once more.

\begin{proposition}\label{app:prop:genfunc_coeff}
	Denoting $g(v,w):=(cv+adw)^2+a^2b^2$, it can be seen that the generating function of the numerator in $I_{m,n}$ is
	\begin{align}
	\MoveEqLeft
		N(v,w,y,z):=\sum_{i,j\ge 0,\,m,n\ge 1} c^{m,n}_{i,j} (cv)^i (adw)^j  (ab)^{2(m+n)-3-i-j} y^m z^n\\
		&=\Re \frac{yz}{\sqrt{1-g(v,w)y}\sqrt{1-g(v,w)z}} \frac{g(v,w)}{cv\sqrt{1-g(v,w)y} +adw\sqrt{1-g(v,w)z}+\ii ab} \label{app:eq:Imn_ops_numerator}\\
		&=\frac{g(v,w)yz}{\sqrt{1-g(v,w)y}\sqrt{1-g(v,w)z}} \frac{cv\sqrt{1-g(v,w)y}+adw\sqrt{1-g(v,w)z}}{(cv\sqrt{1-g(v,w)y}+adw\sqrt{1-g(v,w)z})^2+a^2b^2}, 
	\end{align}
	and for the coefficients alone, with $h(v,w):=(v+w)^2+1$,
	\begin{align}
		C(v,w,y,z)
		&:=\sum_{i,j\ge 0,\,m,n\ge 1} c^{m,n}_{i,j} v^i w^j y^m z^n\\
		&\phantom{:}=\Re \frac{yz}{\sqrt{1-h(v,w)y}\sqrt{1-h(v,w)z}} \frac{h(v,w)}{v\sqrt{1-h(v,w)y}+w\sqrt{1-h(v,w)z}+\ii}\\
		&\phantom{:}=\frac{h(v,w)yz}{\sqrt{1-h(v,w)y}\sqrt{1-h(v,w)z}} \frac{v\sqrt{1-h(v,w)y}+w\sqrt{1-h(v,w)z}}{(v\sqrt{1-h(v,w)y}+w\sqrt{1-h(v,w)z})^2+1}.
	\end{align}
\end{proposition}

\begin{remark}
	Using \eqref{app:eq:Imn_as_coeff} and \eqref{app:eq:coeff_factor}, we can deduce that
	\begin{align}
		I_{m,n}=\frac{\pi}{\delta_+^{m+n-1}c^{2m-1}d^{2n-1}}\coeff*{y^{m}z^{n}}
		\Re \frac{yz}{\sqrt{1-\delta_+y}\sqrt{1-\delta_+z}} \frac{\delta_+}{c\sqrt{1-\delta_+y}+ad\sqrt{1-\delta_+z}+\ii ab}\qquad  \label{app:eq:Imn_ops_rescaled}
	\end{align}
	Comparing this with \eqref{app:eq:Imn_ops_numerator}, we immediately see the correspondence (by setting $v,w=1$, because $g\bigr|_{v=1,w=1}=\delta_+$) --- the knowledge we gain with $N(v,w,y,z)$ is the ability to determine a single coefficient as opposed to the whole numerator.
\end{remark}

\begin{proof}[Proof of \autoref{app:prop:genfunc_coeff}]
	As outlined in \autoref{ssec:genfunc}, we take \eqref{eq:Imn_repr_coeff} and extend the summation to $\infty$ for all parameters except $\ell$ (the binomial coefficients ensure that only finitely many terms are non-zero, except for $r$, where we need change the order of summation and use $k\ge q+r$ to enforce $r\le k-q$), and interchange order so that we can sum, first in $j$ and $i$,
	\begin{align}
		N(v,w,y,z)&\!\begin{multlined}[t][\linewidth-\mathindent-\widthof{$N(v,w,y,z)$}-\multlinegap]
			=\Re\sum_{j,k,\ell,m,n,p,q,r,s,t}  \binom{\ldots}{k,\ell,m,n,p,q,r,s,t} \binom{m+n-1}{j-p-s} \frac{(-1)^{m+n+1+j+p+r+s+t}}{4^{m+n-2}} \cdot\ldots \\
			\ldots\cdot   (adw)^j (\ii ab)^{2(m+n)-3-j} y^m z^n \sum_{i=q+r+t}^\infty \binom{m+n-1-j+p+s}{i-q-r-t} \parens[\Big]{\frac{-cv}{\ii ab}}^i
		\end{multlined}\\
		&\!\begin{multlined}[t][\linewidth-\mathindent-\widthof{$N(v,w,y,z)$}-\multlinegap]
			=\Re\sum_{j,k,\ell,m,n,p,q,r,s,t}  \binom{\ldots}{k-n,p-t} \binom{m+n-1}{j-p-s} \frac{(-1)^{m+n+1+j+p+r+s+t}}{4^{m+n-2}} \cdot\ldots \\
			\ldots\cdot (adw)^j  (\ii ab)^{2(m+n)-3-j} y^m z^n \parens[\Big]{\frac{-cv}{\ii ab}}^{q+r+t} \parens[\Big]{1-\frac{cv}{\ii ab}}^{m+n-1-j+p+s}
		\end{multlined}\\
		&\!\begin{multlined}[t][\linewidth-\mathindent-\widthof{$N(v,w,y,z)$}-\multlinegap]
			=\Re\sum_{k,\ell,m,n,p,q,r,s,t}  \binom{\ldots}{k-n,p-t} \frac{(-1)^{m+n+1+q}}{4^{m+n-2}} (cv)^{q+r+t} (\ii ab)^{2(m+n)-3-q-r-t} \cdot\ldots \\
			\ldots\cdot   y^m z^n \parens[\Big]{1-\frac{cv}{\ii ab}}^{m+n-1} \parens[\Big]{\frac{adw}{\ii ab}}^{p+s} \sum_{j=0}^\infty \binom{m+n-1}{j} \parens[\Big]{\frac{-adw}{\ii ab-cv}}^{j}
		\end{multlined}\\
		&\!\begin{multlined}[t][\linewidth-\mathindent-\widthof{$N(v,w,y,z)$}-\multlinegap]
			=\Re\sum_{k,\ell,m,n,p,q,r,s,t}  \binom{\ldots}{k-n,p-t} \frac{(-1)^{m+n+1+q}}{4^{m+n-2}} \cdot\ldots \\
			\ldots\cdot  (cv)^{q+r+t} (adw)^{p+s} (\ii ab)^{2(m+n)-3-p-q-r-s-t} y^m z^n \parens[\Big]{1-\frac{cv+adw}{\ii ab}}^{m+n-1},
		\end{multlined}
	\end{align}
	where the last line uses $(1-\frac{cv}{\ii ab})(1-\frac{adw}{\ii ab-cv})=1-\frac{cv+adw}{\ii ab}$. We continue by summing in $t$ and $s$, which are just applications of the binomial theorem,
	\begin{align}
		&\!\begin{multlined}[t][\linewidth-\mathindent-\multlinegap]
			=\Re  \sum_{k-n,p-s}  \binom{\ldots}{k-n,p-s} \frac{(-1)^{m+n+1+q}}{4^{m+n-2}} (cv)^{q+r} (adw)^{p+s}  (\ii ab)^{2(m+n)-3-p-q-r-s} \cdot\ldots \\
			\ldots\cdot  y^m z^n\parens[\Big]{1-\frac{cv+adw}{\ii ab}}^{m+n-1}   \sum_{t=0}^{m+n-2-\smash{p-q}-r-s} \binom{m+n-2-p-q-r-s}{t} \parens[\Big]{\frac{cv}{\ii ab}}^t
		\end{multlined}\\
		&\!\begin{multlined}[t][\linewidth-\mathindent-\multlinegap]
			=\Re \sum_{k-n,p-r} \binom{\ldots}{k-n,p-r} \frac{(-1)^{m+n+1+q}}{4^{m+n-2}}    (\ii ab)^{2(m+n)-3-p-q-r}  y^m z^n \parens[\Big]{1-\frac{cv+adw}{\ii ab}}^{m+n-1} \!\!\!\! \cdot\ldots \\
			\ldots\cdot  (cv)^{q+r}  (adw)^{p}  \parens[\Big]{1+\frac{cv}{\ii ab}}^{m+n-2-p-q-r} \sum_{s=0}^{m+n-2-\ell-r} \binom{m+n-2-\ell-r}{s} \parens[\Big]{\frac{adw}{cv+\ii ab}}^s
		\end{multlined}\\
		&\!\begin{multlined}[t][\linewidth-\mathindent-\multlinegap]
			=\Re \sum_{k-n,p-r} \binom{\ldots}{k-n,p-r} \frac{(-1)^{m+n+1+q}}{4^{m+n-2}} (cv)^{q+r}  (adw)^{p}   (\ii ab)^{2(m+n)-3-p-q-r}  y^m z^n  \cdot\ldots \\
			\ldots\cdot  \parens[\Big]{1-\frac{cv+adw}{\ii ab}}^{m+n-1}  \parens[\Big]{1+\frac{cv+adw}{\ii ab}}^{m+n-2-p-q-r}  \parens[\Big]{1+\frac{adw}{cv+\ii ab}}^{p+q-\ell}.
		\end{multlined}
	\end{align}
	Next we sum in $n$, setting $g(v,w):=(cv+adw)^2+a^2b^2=-(\ii ab)^2(1-\frac{cv+adw}{\ii ab})(1+\frac{cv+adw}{\ii ab})$,
	\begin{align}
		&\!\begin{multlined}[t][\linewidth-\mathindent-\multlinegap]
			=\Re \sum_{k-m,p-r} \binom{\ldots}{k-m,p-r} \frac{(-1)^{m+1+q}}{4^{m-2}} (cv)^{q+r} (adw)^{p} (\ii ab)^{2m-3-p-q-r}  y^m \parens[\Big]{1-\frac{cv+adw}{\ii ab}}^{m-1}\!\!\!\!\! \cdot\ldots \\
			\ldots\cdot   \parens[\Big]{1+\frac{cv+adw}{\ii ab}}^{m-2-p-q-r}  \parens[\Big]{1+\frac{adw}{cv+\ii ab}}^{p+q-\ell} \sum_{n=\ell+1}^{\infty}\frac{2^\ell \ell}{n-1} \binom{2n-\ell-3}{\phantom{2}n-\ell-1} \parens[\Big]{\frac{gz}{4}}^n
		\end{multlined}\\
		&\!\begin{multlined}[t][\linewidth-\mathindent-\multlinegap]
			=\Re \sum_{k-m,p-r} \binom{\ldots}{k-m,p-r} \frac{(-1)^{m+1+q}}{4^{m-2}} (cv)^{q+r} (adw)^{p} (\ii ab)^{2m-3-p-q-r}  y^m \parens[\Big]{1-\frac{cv+adw}{\ii ab}}^{m-1}\!\!\!\!\! \cdot\ldots \\
			\ldots\cdot \parens[\Big]{1+\frac{cv+adw}{\ii ab}}^{m-2-p-q-r}  \parens[\Big]{1+\frac{adw}{cv+\ii ab}}^{p+q-\ell}\parens[\Big]{\frac{gz}{4}}^{\ell+1} 2^\ell \sum_{n=0}^{\infty}\frac{\ell}{n+\ell}  \binom{2n+\ell-1}{n} \parens[\Big]{\frac{gz}{4}}^n
		\end{multlined}\\
		&\!\begin{multlined}[t][\linewidth-\mathindent-\multlinegap]
			=\Re \sum_{k-m,p-r} \binom{\ldots}{k-m,p-r} \frac{(-1)^{m+q}}{4^{m-1}} (cv)^{q+r} (adw)^{p} (\ii ab)^{2m-1-p-q-r}  y^m z  \cdot\ldots \\
			\ldots\cdot \parens[\Big]{1-\frac{cv+adw}{\ii ab}}^{m} \parens[\Big]{1+\frac{cv+adw}{\ii ab}}^{m-1-p-q-r}  \parens[\Big]{1+\frac{adw}{cv+\ii ab}}^{p+q-\ell}  \parens*{1-\sqrt{1-gz}}^\ell,
		\end{multlined}
	\end{align}
	where we used \eqref{app:eq:ops_sqrt_k} for the last equation and distributed one surplus power back to its constituent factors.
	
	In the case of $\ell=0$, \eqref{app:eq:ops_sqrt_k} is not applicable, so we would have to subtract the entire term with $\ell$ set to $0$. As is immediately obvious from the last display, this only affects the case that $n=1$ (because the only remaining power of $z$ is $z^1$), and we will not deal with the case $m=1 \lor n=1$ here, as they can be done in exactly the same way as for the general case (e.g.~starting from \eqref{app:eq:Imn_simplif_frac}, where either $y$ or $z$ can then be set to zero and the derivatives only need to be considered for the other variable).
	
	Similarly, we tackle the sum in $m$ (again ignoring the error for $k=0$), allowing us to also deal with the sum in $k$,
	\begin{align}
		&\!\begin{multlined}[t][\linewidth-\mathindent-\multlinegap]
			=\Re  \sum_{k,\ell,p-r}  \binom{\ldots}{p-r} 4(-1)^{q} (cv)^{q+r} (adw)^{p} (\ii ab)^{-1-p-q-r} z \parens[\Big]{1+\frac{cv+adw}{\ii ab}}^{-1-p-q-r} \cdot\ldots \\
			\ldots\cdot   \parens[\Big]{1+\frac{adw}{cv+\ii ab}}^{p+q-\ell}  \parens*{1-\sqrt{1-gz}}^\ell 2^k \parens[\Big]{\frac{gy}{4}}^{k+1}\sum_{m=0}^\infty \frac{ k}{m+k} \binom{2m+k-1}{m} \parens[\Big]{\frac{gy}{4}}^{m}
		\end{multlined}\\
		&\!\begin{multlined}[t][\linewidth-\mathindent-\multlinegap]
			=\Re  \sum_{k,\ell,p-r}  \binom{\ldots}{p-r} (-1)^{q+1} (cv)^{q+r} (adw)^{p}  (\ii ab)^{1-p-q-r} yz \parens[\Big]{1-\frac{cv+adw}{\ii ab}} \cdot\ldots \\
			\ldots\cdot  \parens[\Big]{1+\frac{cv+adw}{\ii ab}}^{-p-q-r} \parens[\Big]{1+\frac{adw}{cv+\ii ab}}^{p+q-\ell}  \parens*{1-\sqrt{1-gz}}^\ell \parens*{1-\sqrt{1-gy}}^k
		\end{multlined}\\
		&\!\begin{multlined}[t][\linewidth-\mathindent-\multlinegap]
			=\Re  \sum_{\ell,p-r}  \binom{\ldots}{p-r} (-1)^{q+1} (cv)^{q+r} (adw)^{p}  (\ii ab)^{1-p-q-r} yz \parens[\Big]{1-\frac{cv+adw}{\ii ab}}  \cdot\ldots \\
			\ldots\cdot \parens[\Big]{1+\frac{cv+adw}{\ii ab}}^{-p-q-r} \parens[\Big]{1+\frac{adw}{cv+\ii ab}}^{p+q-\ell}  \parens*{1-\sqrt{1-gz}}^\ell \sum_{k=q+r}^\infty \parens*{1-\sqrt{1-gy}}^k
		\end{multlined}\\
		&\!\begin{multlined}[t][\linewidth-\mathindent-\multlinegap]
			=\Re\frac{yz}{\sqrt{1-gy}}\sum_{\ell,p-r}  \binom{\ldots}{p,q,r}  (-1)^{q+1} (cv)^{q+r} (adw)^{p}  (\ii ab)^{1-p-q-r} \parens[\Big]{1-\frac{cv+adw}{\ii ab}}  \cdot\ldots \\
			\ldots\cdot  \parens[\Big]{1+\frac{cv+adw}{\ii ab}}^{-p-q-r} \parens[\Big]{1+\frac{adw}{cv+\ii ab}}^{p+q-\ell}  \parens*{1-\sqrt{1-gz}}^\ell \parens*{1-\sqrt{1-gy}}^{q+r}.
		\end{multlined}
	\end{align}
	Next, we sum in $r$, by \eqref{app:eq:ops_geom_alpha},
	\begin{align}
		&\!\begin{multlined}[t][\linewidth-\mathindent-\multlinegap]
			=\Re\frac{yz}{\sqrt{1-gy}} \sum_{\ell,p,q} \binom{\ldots}{p,q}
			(-1)^{q+1} (cv)^{q} (adw)^{p}  (\ii ab)^{1-p-q} \parens[\Big]{1-\frac{cv+adw}{\ii ab}} \parens[\Big]{1+\frac{cv+adw}{\ii ab}}^{-p-q} \cdot\ldots \\
			\ldots\cdot \parens[\Big]{1+\frac{adw}{cv+\ii ab}}^{p+q-\ell}  \parens*{1-\sqrt{1-gz}}^\ell \parens*{1-\sqrt{1-gy}}^{q} \;\smash[b]{\underbrace{\sum_{r=0}^\infty \binom{\ell+r}{r} \parens[\bigg]{\frac{cv\parens*{1-\sqrt{1-gy}}}{cv+adw+\ii ab}}^r,}_{\parens*{\frac{cv+adw+\ii ab}{cv\sqrt{1-gy}+adw+\ii ab}}^{\ell+1}}} \vphantom{\underbrace{\sum_{r=0}}_{a}}
		\end{multlined}
	\end{align}
	then in $q$,
	\begin{align}
		&\!\begin{multlined}[t][\linewidth-\mathindent-\multlinegap]
			=\Re\frac{-yz}{\sqrt{1-gy}}\sum_{\ell,p}  \binom{\ell+1}{p} (adw)^{p}  (\ii ab)^{1-p} \parens[\Big]{1-\frac{cv+adw}{\ii ab}} \parens[\Big]{1+\frac{cv+adw}{\ii ab}}^{-p} \parens[\Big]{1+\frac{adw}{cv+\ii ab}}^{p-\ell} \!\!\!\!\! \cdot\ldots \\
			\ldots\cdot \parens*{1-\sqrt{1-gz}}^\ell \parens[\Big]{\frac{cv+adw+\ii ab}{cv\sqrt{1-gy}+adw+\ii ab}}^{\ell+1} \sum_{q=0}^{\infty} \binom{\ell-p}{q} \parens[\bigg]{\frac{-cv\parens*{1-\sqrt{1-gy}}}{cv+\ii ab}}^q,
		\end{multlined}
	\end{align}
	as well as in $p$,
	\begin{align}
		&\!\begin{multlined}[t][\linewidth-\mathindent-\multlinegap]
			=\Re\frac{-yz}{\sqrt{1-gy}}\sum_{\ell=0}^\infty \ii ab \parens[\Big]{1-\frac{cv+adw}{\ii ab}}  \parens[\Big]{\frac{cv+adw+\ii ab}{cv+\ii ab}}^{-\ell} \parens*{1-\sqrt{1-gz}}^\ell \cdot\ldots \\
			\ldots\cdot  \parens[\Big]{\frac{cv+adw+\ii ab}{cv\sqrt{1-gy}+adw+\ii ab}}^{\ell+1} \parens[\bigg]{\frac{cv\sqrt{1-gy}+\ii ab}{cv+\ii ab}}^\ell \sum_{p=0}^{\ell} \binom{\ell+1}{p} \parens[\bigg]{\frac{adw}{cv\sqrt{1-gy}+\ii ab}}^p
		\end{multlined}\\
		&\!\begin{multlined}[t][\linewidth-\mathindent-\multlinegap]
			=\Re\frac{yz}{\sqrt{1-gy}}  \sum_{\ell=0}^\infty \frac{\parens{cv+adw-\ii ab}\parens{cv+adw+\ii ab}}{cv\sqrt{1-gy}+adw+\ii ab} \parens*{1-\sqrt{1-gz}}^\ell \cdot\ldots \\
			\ldots\cdot   \parens[\bigg]{\frac{cv\sqrt{1-gy}+\ii ab}{cv\sqrt{1-gy}+adw+\ii ab}}^\ell \parens{\frac{\parens{cv\sqrt{1-gy}+adw+\ii ab}^{\ell+1}-(adw)^{\ell+1}}{(cv\sqrt{1-gy}+\ii ab)^{\ell+1}}}.
		\end{multlined}
	\end{align}
	Finally, we wrap things up by summing in $\ell$,
	\begin{align}
		&\!\begin{multlined}[t][\linewidth-\mathindent-\multlinegap]
			=\Re\frac{yz}{\sqrt{1-gy}} \frac{(cv+adw)^2+a^2b^2}{cv\sqrt{1-gy}+\ii ab} \Biggl(\sum_{\ell=0}^\infty  \parens*{1-\sqrt{1-gz}}^{\ell} - \ldots \\
			\ldots -\sum_{\ell=0}^\infty\frac{adw}{cv\sqrt{1-gy}+adw+\ii ab}\parens[\bigg]{\frac{adw\parens*{1-\sqrt{1-gz}}}{cv\sqrt{1-gy}+adw+\ii ab}}^\ell\Biggr)
		\end{multlined}\\
		&=\Re\frac{yz}{\sqrt{1-gy}} \frac{g}{cv\sqrt{1-gy}+\ii ab} \parens[\bigg]{\frac{1}{\sqrt{1-gz}}-\frac{adw}{cv\sqrt{1-gy}+adw\sqrt{1-gz}+\ii ab}}\\
		&=\Re\frac{yz}{\sqrt{1-gy}\sqrt{1-gz}} \frac{g}{cv\sqrt{1-gy}+adw\sqrt{1-gz}+\ii ab}.
	\end{align}
	As it turns out, this formula also holds for the case $m=1\lor n=1$ that we have omitted thus far. The proof is left as an exercise to the reader.
	
	The form of the function $C(v,w,y,z)$ is easily seen by using \eqref{app:eq:coeff_factor} --- or more heuristically, setting $a=b=c=d=1$ in $N(v,w,y,z)$. This finishes the proof.
\end{proof}

\begin{proposition}\label{app:prop:coeff_zero}
	If any of the following conditions is \emph{not} met, the coefficient $c^{m,n}_{i,j}$ is zero,
	\begin{align}
	i+j&\equiv 1\bmod{2}, & i+j&\le 2(m+n)-3, & (i&\ge 2m-1 \lor j\ge 2n-1).
	\end{align}
	If the first two conditions \emph{are} met, the coefficients can be calculated as follows
	\begin{align}
		c^{m,n}_{i,j}
		\begin{multlined}[t][\linewidth-\mathindent-\widthof{$c^{m,n}_{i,j}$}-\multlinegap]
			=\sum_{r=0}^{i}\sum_{s=0}^{j} \delta_{\{r+s\equiv1\bmod{2}\}} (-1)^{m+n+\frac{r+s-1}{2}} \binom{r+s}{s}  \cdot\ldots \\
			\ldots \cdot  \binom{\frac{r-1}{2}}{m-1}\binom{\frac{s-1}{2}}{n-1} \binom{m+n-1}{m+n-1-\frac{i+j-r-s}{2}} \binom{i+j-r-s}{i-r}.
		\end{multlined}
	\end{align}
\end{proposition}

\begin{proof}
	Instead of explicit differentiation, we will reverse the process of calculating a generating function, and ``disassemble'' $C(v,w,y,z)$ in a particular way. However, we will still use the information we have already, i.e.~that $i+j$ must be odd and less than $2(m+n)-3$ (e.g.~from the representation in \autoref{app:prop:coeff_expl_long}). First we calculate
	\begin{align}
		c^{m,n}_{i,j}
		&=\Re\coeff*{v^i w^j y^{m}z^{n}}\frac{yz}{\sqrt{1-hy}\sqrt{1-hz}} \frac{h}{v\sqrt{1-hy}+w\sqrt{1-hz}+\ii} \\
		&=\Re\coeff*{v^i w^j y^{m}z^{n}} \frac{hyz}{\sqrt{1-hy}\sqrt{1-hz}} \sum_{r\ge 0}\frac{(-v\sqrt{1-hy})^r}{(w\sqrt{1-hz}+\ii)^{r+1}}\\
		&=\Re\coeff*{v^i w^j y^{m}z^{n}} \frac{hyz}{\sqrt{1-hz}} \sum_{r\ge 0} (-v)^r (\sqrt{1-hy})^{r-1} (-\ii)^{r+1} \sum_{s\ge 0} \binom{r+s}{s} \parens*{\ii w\sqrt{1-hz}}^s\\
		&=\Re\coeff*{v^i w^j y^{m}z^{n}} -hyz \sum_{r,s\ge 0} \ii^{r+s+1} \binom{r+s}{s} v^r w^s \sum_{k\ge 0} \binom{\frac{r-1}{2}}{k} (-hy)^k \sum_{\ell\ge 0} \binom{\frac{s-1}{2}}{\ell} (-hz)^\ell,
	\end{align}
	having used \eqref{app:eq:ops_geom}, \eqref{app:eq:ops_geom_k} and \eqref{app:eq:ops_geom_alpha}. Extracting the appropriate powers of $y,z$, we continue
	\begin{align}
		c^{m,n}_{i,j}
		&=\Re\coeff*{v^i w^j} \sum_{r,s\ge 0} (-1)^{m+n+1} \ii^{r+s+1} \binom{r+s}{s} \binom{\frac{r-1}{2}}{m-1} \binom{\frac{s-1}{2}}{n-1} v^r w^s h^{m+n-1}\\
		&\!\begin{multlined}[t][\linewidth-\mathindent-\widthof{$c^{m,n}_{i,j}$}-\multlinegap]
			=\Re\coeff*{v^i w^j} \sum_{r,s,t,u\ge 0}  (-1)^{m+n+1} \ii^{r+s+1} \binom{r+s}{s} \binom{\frac{r-1}{2}}{m-1} \binom{\frac{s-1}{2}}{n-1} \cdot\ldots \\
			\ldots\cdot\binom{m+n-1}{t} \binom{2(m+n-1-t)}{u} v^{r+u} w^{s+2(m+n-1-t)-u} 1^t
		\end{multlined}\\
		&\!\begin{multlined}[t][\linewidth-\mathindent-\widthof{$c^{m,n}_{i,j}$}-\multlinegap]
			=\Re\coeff*{v^i w^j} \sum_{i',j',r,s\ge 0} (-1)^{m+n+1} \ii^{r+s+1} \binom{r+s}{s} \binom{\frac{r-1}{2}}{m-1} \binom{\frac{s-1}{2}}{n-1} \cdot\ldots \\
			\ldots\cdot \binom{m+n-1}{m+n-1-\frac{i'+j'-r-s}{2}} \binom{i'+j'-r-s}{i'-r} v^{i'} w^{j'}
		\end{multlined}\\
		&\!\begin{multlined}[t][\linewidth-\mathindent-\widthof{$c^{m,n}_{i,j}$}-\multlinegap]
			= \Re\sum_{r=0}^{i}\sum_{s=0}^{j} (-1)^{m+n+1} \ii^{r+s+1} \binom{r+s}{s} \binom{\frac{r-1}{2}}{m-1} \binom{\frac{s-1}{2}}{n-1} \cdot\ldots \\
			\ldots\cdot\binom{m+n-1}{m+n-1-\frac{i+j-r-s}{2}} \binom{i+j-r-s}{i-r}.
		\end{multlined}
	\end{align}
	Firstly, the restriction of the range in $r,s$ comes from the last binomial coefficient, because clearly the lower entry needs to be positive, and the upper one needs to be larger than the lower one. Also, like $i+j$, we see that $r+s$ has to be odd for the term to contribute to the (real part of the) sum. From this we can read off the formula for the coefficients claimed in the proposition.
	
	We use $r+s\equiv 1\bmod{2}$ to split the sum into ($r=2p$ even, $s=2q+1$ odd) and ($r=2p+1$ odd, $s=2q$ even) for $p,q\in\bbN$,
	\begin{align}
		c^{m,n}_{i,j}
		&\!\begin{multlined}[t][\linewidth-\mathindent-\widthof{$c^{m,n}_{i,j}$}-\multlinegap]
			=\sum_{p=0}^{\floor{\frac{i}{2}}}  \sum_{q=0}^{\floor{\frac{j-1}{2}}}  (-1)^{m+n+p+q} \binom{2p+2q+1}{2q+1}  \binom{p-\frac{1}{2}}{m-1}  \binom{q}{n-1} \cdot\ldots \\
			\shoveright{\ldots\cdot\binom{m+n-1}{m+n+p+q-\frac{i+j+1}{2}}  \binom{i+j-2p-2q-1}{i-2p}  +\ldots}\\
			\shoveleft{\phantomrel \ldots+ \sum_{p=0}^{\floor{\frac{i-1}{2}}}  \sum_{q=0}^{\floor{\frac{j}{2}}}  (-1)^{m+n+p+q} \binom{2p+2q+1}{2q}  \binom{p}{m-1} \cdot \ldots}\\
			\ldots\cdot\binom{q-\frac{1}{2}}{n-1}  \binom{m+n-1}{m+n+p+q-\frac{i+j+1}{2}}  \binom{i+j-2p-2q-1}{i-2p-1}.
		\end{multlined}
	\end{align}
	Now, if $i\le 2(m-1)$, we see that $p<m-1$ in the second sum, and similarly, if $j\le 2(n-1)$ then $q<n-1$ in the first sum, which means all terms in the sums are zero due to the binomial coefficients $\binom{p}{m-1}$ and $\binom{q}{n-1}$. This finishes the proof.
\end{proof}

\begin{remark}\label{rem:conject_coeff_pos}
	After having investigated the quantity $I_{m,n}$ intensely while trying to prove \autoref{app:th:Imn_est}, we are convinced that
	\begin{align}
		c^{m,n}_{i,j}\ge 0, \qquad \forall m,n\in\bbN,\, \forall i,j\in\bbN_0.
	\end{align}
	However, we have not seriously attempted to prove this.
\end{remark}

\pagestyle{plain}

\printbibliography[heading=bibintoc]

\end{document}